\documentclass{article}
\usepackage{amssymb}
\usepackage{graphicx}
\usepackage{amsmath}
\usepackage{authblk}  
\usepackage{amsmath} 
\usepackage{bm} 
\usepackage{amsthm}
\usepackage{geometry}
\usepackage{booktabs}
\usepackage{booktabs}
\usepackage{algorithm}
\usepackage{tikz}
\usepackage{mathrsfs}
\usepackage{dsfont}
\usetikzlibrary{shapes.geometric, arrows, positioning, calc}
\usepackage{subcaption}
\usepackage{algorithmic}
\usepackage{xcolor}
\numberwithin{equation}{section}
\newtheorem{definition}{Definition}
\newtheorem{assumption}{Assumption}
\newtheorem{example}{Example}
\newtheorem{lemma}{Lemma}
\newtheorem{proposition}{Proposition}
\newtheorem{theorem}{Theorem}
\newtheorem{remark}{Remark}

\newcommand{\vt}{{\vartheta}}
\newcommand{\dd}{\mathsf {d\kern -0.07em l}} 
\newcommand{\bgeqn}{\begin{eqnarray}}
\newcommand{\edeqn}{\end{eqnarray}}
\newcommand{\bgeq}{\begin{eqnarray*}}
\newcommand{\edeq}{\end{eqnarray*}}
\newcommand{\bec}{\begin{center}}

\def\R{\mathbb{R}}

\title{A Bayesian Composite Risk Approach for Stochastic Optimal Control and Markov Decision Processes}

\author[1,2]{Wentao Ma\thanks{Email: \texttt{mwtmwt7@stu.xjtu.edu.cn}}}
\author[1,2]{Zhiping Chen\thanks{Email: \texttt{zchen@mail.xjtu.edu.cn}}}
\author[3]{Huifu Xu\thanks{Corresponding author. Email: \texttt{hfxu@se.cuhk.edu.hk}}}

\affil[1]{School of Mathematics and Statistics, Xi'an Jiaotong University, Xi'an, Shaanxi, P. R. China}
\affil[2]{Research Center for Optimization Technology and Intelligent Game, Xi'an International Academy for Mathematics and Mathematical Technology, Xi'an, P. R. China}
\affil[3]{Department of Systems Engineering and Engineering Management, The Chinese University of Hong Kong, Hong Kong}
\date{}

\begin{document}

\maketitle
\begin{abstract}
Inspired by 
Shapiro et al.~\cite{shapiro2023episodic},
we consider a
stochastic optimal control (SOC) and Markov decision process (MDP) 
where the risks arising 
from epistemic and aleatoric uncertainties are assessed using
Bayesian composite risk (BCR) measures. 
The time dependence of the risk measures allows us to capture 
the decision maker's (DM) dynamic risk preferences 
opportunely as increasing information about both uncertainties is obtained. This makes the new BCR-SOC/MDP model more flexible than 
conventional risk-averse SOC/MDP models.  Unlike
\cite{shapiro2023episodic} where 
the control/action at each episode is 
based on 
the current state alone, 
the new model 
allows 
the control 
to depend on the probability distribution of the epistemic uncertainty, which reflects the fact that in many practical instances the cumulative 
information about 
epistemic uncertainty often
affects the DM's belief about 
the future aleatoric uncertainty and hence 
the DM's action \cite{strens2000bayesian}.
The new modeling paradigm incorporates several
existing SOC/MDP models including distributionally robust 
SOC/MDP models and Bayes-adaptive MDP models and generates so-called preference robust SOC/MDP models.
Moreover, we derive conditions under which the BCR-SOC/MDP model is well-defined,
demonstrate that finite-horizon BCR-SOC/MDP models can be 
solved using dynamic programming techniques,
and extend the discussion to the infinite-horizon case. By using Bellman equations,  
we show that under some standard 
conditions, asymptotic convergence of the optimal values and optimal policies
as the episodic variable goes to infinity is achieved. We introduce a novel hyper-parameter approach for discretizing posterior distributions to enhance the practicality and efficiency of proposed algorithms.
Finally,   
we carry out numerical tests on  a finite horizon spread betting problem and
an inventory control problem and 
show the effectiveness of the proposed model and numerical schemes.
\\

\end{abstract}

\section{Introduction}
Stochastic optimal control (SOC)  and Markov decision process (MDP) are
the fundamental models for sequential decision-making where the outcomes are partly random and partly under the control of a decision maker (DM). In recent years, MDP has become even more popular because of its close relationship to reinforcement learning.
In a SOC/MDP, the environment is characterized by a finite or infinite set of states, and the DM's actions influence the transitions between these states. Each action taken in a given state determines the next state probabilistically, thereby capturing the dynamic and stochastic nature of the environment and the long-term consequences of decisions.
SOC/MDP have been applied extensively in diverse fields such as inventory control \cite{puterman2014markov}, 
hydro-thermal planning \cite{guigues2023risk}, economics and behavioral ecology \cite{bertsekas2012dynamic}, among other fields \cite{feinberg2012handbook}.

A critical component of a SOC/MDP model's validity lies in its ability to characterize uncertainty effectively. The uncertainties addressed by such models can arise from both epistemic uncertainty—that is, arising from limited data availability that affects the estimation of model parameters—and the inherent randomness of the environment, referred to as aleatoric uncertainty \cite{wang2024aleatoric}, which are both inevitable in the real world. 
Aleatoric uncertainty pertains to 
the variability in the outcome of an experiment arising from inherently random effects. Consequently, even the most accurate model of this process can only offer probabilities for different outcomes, 
rather than exact answers. 
On the other hand, epistemic uncertainty arises from a lack of knowledge about the correct model. In other words, it reflects the ignorance of the DM regarding the epistemic state, rather than the inherent randomness of the system. Unlike uncertainty caused by inherent randomness, epistemic uncertainty can, in principle, be reduced by acquiring additional information. We will review existing SOC/MDP models that address these two types of uncertainty below.

The primary goal of SOC/MDP is to develop a policy that on average minimizes a specified cost function under uncertainty, typically by mapping states to actions. However, this risk-neutral SOC/MDP neglects critical low-probability events with potentially severe consequences. To address this limitation, risk-averse SOC/MDP has been proposed to minimize not only the expected costs but also the variability and associated risks of outcomes. 
These models typically handle aleatoric uncertainty using a fixed distribution. 
Early work in this area adopted the expected utility framework, where a convex utility function is adopted to achieve risk aversion \cite{howard1972risk}. Other studies have explored  a static risk-averse SOC/MDP, where a risk measure is applied to the total accumulated cost, e.g., the mean-variance \cite{sniedovich1980variance,mannor2011mean}, or value-at-risk (VaR) \cite{filar1995percentile}. However, these measures are not coherent, failing to satisfy properties consistent with rational decision-making. Various approaches have emerged \cite{chow2015risk,tamar2016sequential} within the realm of coherent risk measures. Nonetheless, SOC/MDPs employing a static risk-averse formulation cannot guarantee time-consistent optimal policies, where a current optimal policy may lose its optimality when new realizations emerge. As highlighted in \cite{ruszczynski2010risk},  utilizing nested compositions of risk transition mappings at each episode facilitates time-consistent decision-making through dynamic programming solutions. Notable research has explored dynamic formulation on finite and infinite horizons further, including average value-at-risk (AVaR) \cite{bauerle2011markov}, quantile-based risk measure \cite{jiang2018risk}, and entropic risk measure (ERM) \cite{pichler2023risk,hau2023entropic}, respectively.

Although SOC/MDP and their risk-averse variants provide valuable frameworks for decision-making under aleatoric uncertainty, a common practical challenge is the neglect of epistemic uncertainty in model parameters. In addition, the inherent randomness of the environment complicates the task of maintaining the robustness of learned policies against small perturbations in real-world applications (\cite{mannor2004bias}). Distributionally robust optimization (DRO) has emerged as a powerful approach which takes epistemic uncertainty into consideration by constructing an ambiguity set for possible distributions of the environment. 
Ambiguity sets in DRO can generally be categorized into two main types: moment-based and distance-based. In moment-based DRO \cite{delage2010distributionally,wiesemann2014distributionally,xu2018distributionally}, the DM has certain information about the moments of the random variable's distribution. In distance-based DRO, the DM has a known reference distribution and considers a region of uncertainty around it, typically defined by a probability metric such as Kullback-Leibler divergence \cite{philpott2018distributionally} or Wasserstein distance \cite{gao2023distributionally}. 
This has led to the development of distributionally robust SOC/MDP, as introduced in \cite{nilim2005robust,iyengar2005robust}.
Extensive discussions of distributionally robust SOC/MDP can be found in studies such as \cite{mannor2019data,wiesemann2013robust,xu2010distributionally,shapiro2022distributionally,yang2020wasserstein}, which usually rely on robust Bellman equations. However, it is essential to note that an inappropriately constructed ambiguity set may lead to an over conservative consideration of epistemic uncertainty, which will generate a policy that performs poorly under far more realistic scenarios than the worst case  \cite{wu2018bayesian}.

The existing SOC/MDP models discussed above highlight an important issue: there is a large middle ground between an optimistic belief in a fixed environment, and a pessimistic focus on worst-case scenarios when dealing with epistemic uncertainty. Further, since both risk-averse SOC/MDP and distributionally robust SOC/MDP operate within the domain of continuous decision-making, the dynamics of certain potential parameters may vary between interaction episodes and remain elusive.
When faced with unfamiliar scenarios, the DM is compelled to learn continuously through a single episode of interaction, striving to deduce the underlying model parameters while simultaneously aiming to optimize the objective function (see, \cite{birge2021dynamic,jalota2024stochastic}). Although relying on point estimates of potential parameters can yield feasible strategies within the current risk-averse SOC/MDP framework, and taking worst-case scenarios of these parameters into consideration can produce feasible strategies in the distributionally robust SOC/MDP framework, the absence of adaptive learning between episodes often results in insufficiently robust or over conservative strategies. This issue highlights a common shortcoming that is inadequately addressed in conventional SOC/MDP models, which assume no prior knowledge and depend solely on learning through repetitive interactions within the same system.
These challenges can be addressed effectively through a Bayesian approach to MDP, known as Bayes-adaptive MDP \cite{strens2000bayesian}. In Bayes-adaptive MDPs, the DM utilizes prior knowledge expressed as a belief distribution over a range of potential environments, continually updating this belief using Bayes' rule during engagement with the system \cite{duff2002optimal}.
By incorporating Bayesian posterior information instead of concentrating solely on extreme scenarios, the Bayes-adaptive MDP model offers a more balanced and practical approach than distributionally robust SOC/MDP  problems where uncertainty plays a critical role. Building on the Bayes-adaptive MDP, a significant body of work \cite{dearden2013model,osband2013more} enhances our understanding and application of Bayesian approaches in MDP, contributing to a broader landscape of intelligent decision-making systems. 
In this framework, the posterior distribution is treated as an augmented state, complementing the original physical state, and is updated at each episode based on the observed reward and transition, see \cite{ross2011bayesian}. 
In essence, this Bayesian approach not only enhances a deeper understanding of possible scenarios but also enables a more adaptive and flexible optimization policy.

In conclusion, the risk-averse SOC/MDP model captures low-probability, high-impact events effectively under aleatoric uncertainty, while the Bayes-adaptive MDP model adapts to and delineates epistemic uncertainty robustly through a learning process. These models complement each other, enhancing DM's management of risk. Therefore, exploring the integration of these models into a risk-averse Bayes-adaptive SOC/MDP framework is essential to capture and manage risks effectively, particularly those involving cognitive uncertainties and the integration of observational data and prior knowledge. However, existing literature on this topic remains sparse, with most studies primarily focusing on optimizing risk measures using finite and discrete states over MDPs such as \cite{sharma2019robust}. In addition, previous studies have predominantly addressed static risk, overlooking the dynamic evolution of risk over time, which is crucial for comprehensive and timely risk assessment \cite{rigter2021risk,chow2015risk}. 
However, in the absence of the nested risk function structure, prior studies such as \cite{osogami2012time,iancu2015tight} have demonstrated that optimizing a static risk function can also potentially lead to time-inconsistent behavior. Hence, investigating dynamic risk becomes paramount, especially in the context of risk-averse Bayes-adaptive SOC/MDPs with a continuous state space.
In a more recent development, 
a Bayesian risk MDP (BR-MDP) framework is  proposed in \cite{lin2022bayesian,wang2023bayesian} to dynamically solve a risk-neutral Bayes-adaptive MDP, whereby 
some specific 
coherent risk measures 
are used to
quantify 
the risk of the expected value of 
the cost function 
associated with epistemic uncertainty.
To reduce computational challenges, Shapiro et al.~\cite{shapiro2023episodic} proposed a suboptimal episodic Bayesian SOC framework
compared to BR-MDP, 
which is like conventional SOC/MDPs in avoiding the incorporation of the
posterior distribution
as an augmented state. 
They demonstrate convergence 
of the optimal values like
that of
conventional SOC/MDP models.
In this paper, we follow this strand of research by adopting general law invariant risk measures to quantify the risks associated with 
both 
epistemic and aleatoric uncertainties and
allow the explicit dependence of the optimal policy at each episode on the posterior distribution.
The new model provides a more general adaptive framework for 
risk management of dynamic decision-making.

Building on these insights, it becomes evident that a unified framework integrating risk-averse SOC/MDP, distributionally robust SOC/MDP, and Bayes-adaptive MDPs 
may help to address the limitations of existing models. Such a framework not only captures the dynamic interplay between epistemic and aleatoric uncertainties but also integrates these uncertainties into objective functions tailored to a wide range of risk preferences. Conventional approaches, while  handling specific aspects of uncertainty effectively, often fail to achieve this level of flexibility and comprehensiveness. 
To leverage fully the Bayesian approach in modeling decision-making under distributional uncertainty, given its suitability for such problems, we adopt a perspective akin to composite risk optimization (CRO). This approach, introduced by \cite{qian2019composite}, has proven effective for static (single-stage) optimization. By extending this approach to dynamic SOC/MDP settings, our research aims to analyze the temporal evolution of risk and its impact on state transitions and decisions, offering a refined strategy for timely, adaptive risk management policies. This perspective offers a robust foundation for unifying the existent SOC/MDP models within a coherent and dynamic framework.
The main contributions of 
this paper can be summarized as follows.
\begin{itemize}
\item \textbf{Modeling}. 
We propose an adaptive  
SOC/MDP model where the risks arising 
from epistemic and aleatoric uncertainties are assessed by
Bayesian composite risk (BCR) measures. 
Unlike the
episodic Bayesian SOC model \cite{shapiro2023episodic}, where the control or action at each episode is based on the current state alone and ignores any future revelation of the stochastic process, 
the BCR-SOC/MDP 
model allows the action to  
depend explicitly on the probability distribution of the epistemic uncertainty \cite{strens2000bayesian}, 
which reflects the fact that, in many practical situations, accumulated information about epistemic uncertainty can influence the DM's belief about future aleatoric uncertainty, thereby impacting the DM's actions.
Moreover, we show 
through examples
that the proposed BCR-SOC/MDP model
subsumes a number of important existing SOC/MDP models (Examples~\ref{bcr-1}-\ref{bcr-6}), and that it can be linked to preference robust SOC/MDP models (Example~\ref{bcr-9}).
Note that 
it is well known in the risk neutral setting that 
SOC and MDP models are equivalent \cite{shapiro2022distributionally}. 
The equivalence also holds in risk averse setting.
We use SOC/MDP instead of merely 
SOC or MDP as both are used in the literature \cite{lin2022bayesian,ruszczynski2010risk}.

\item \textbf{Analysis and computation}. 
 We derive conditions under which  
the BCR-SOC/MDP model is well defined (Propositions~\ref{convex} and \ref{thm:V-cont}) and
demonstrate how the BCR-SOC/MDP model can be solved
in both the finite and infinite horizon cases (Algorithms~\ref{alg:A}-\ref{alg:B}). Furthermore, in the case when the posterior distribution converges to the Dirac distribution at the true parameter value ($\theta^c$) and the inner risk measure is H\"older continuous in the parameter at $\theta^c$, we show as in \cite{shapiro2023episodic} that 
the optimal value and optimal policy derived from the proposed infinite-horizon BCR-SOC/MDP model converge to their respective true optimal counterparts with respect to increment of sample data (Theorem \ref{thm-canew-1}). When the risk measures take specific forms such as expectation, VaR, robust spectral risk measure (SRM) and AVaR,
we derive error bounds for the approximate value function in terms of the error of the posterior mean 
and the posterior variance (Theorem~\ref{thm-ca-2}). 
To ensure tractability of the proposed algorithms, we propose a novel hyper-parameter approach for discretizing the posterior distribution (Section~6), which 
differs significantly from the direct discretization 
approach in the augmented state space (\cite{lin2022bayesian}).
The main difference is that the former operates in a finite‑dimensional space whereas the latter operates
in the augmented state space
which is typically infinite dimensional (due to the posterior belief state). Further, we explore the tractability and probabilistic guarantees of specific models, including the VaR-Expectation SOC/MDP and the AVaR-AVaR SOC/MDP.

\item \textbf{Applications}.
We examine performances of the proposed BCR-SOC/MDP 
models and computational schemes 
by applying them to two 
conventional 
SOC/MDP problems: a finite-horizon spread betting problem (Section 7.1) and an infinite-horizon inventory control problem  (Section 7.2). 
In the spread betting problem, 
the BCR-SOC/MDP model is more robust and has lower variability 
than the standard risk-averse and distributionally robust SOC/MDP models, 
particularly 
when the number of historical  records of market movement is small.
In the inventory control problem, the BCR-SOC/MDP framework exhibits convergence of
the optimal values and optimal policies to their true counterparts as information 
on aleatoric uncertainty accumulates, demonstrates its adaptability and dynamic learning capabilities. 
The preliminary numerical results
show the efficiency and effectiveness of the BCR-SOC/MDP model in tackling complex decision-making problems with evolving uncertainties in real-world applications.
\end{itemize}

The rest of the paper is 
structured as follows. In
Section 2, we revisit some 
basic notions and results 
in SOC/MDP and risk measures which are needed 
throughout the paper. 
In Section 3, 
we introduce the BCR-SOC/MDP 
model and demonstrate
how it can subsume 
some existing SOC/MDP models.
In Section 4, we briefly discuss the dynamic programming 
formulation of 
finite horizon BCR-SOC/MDP.
In Section 5, we establish the existence, uniqueness, and convergence of a stationary optimal policy for the infinite-horizon BCR-SOC/MDP. 
In Section 6, we introduce an adaptive hyper-parameter algorithm for discretizing posterior distributions and present a computationally efficient SAA algorithm for two special cases. Finally, in Section 7,
we report numerical test results 
about the BCR-SOC/MDP model, displaying their superior performance and applicability.

\section{Preliminaries}

In this section, we revisit
some basic notions and results in SOC/MDP and risk measurement 
that are needed throughout the paper.

\subsection{Stochastic optimal control/Markov decision process}	

By convention (see e.g.~\cite{puterman2014markov}), we express a discounted MDP 
by a 5-tuple $(\mathcal{S}, \mathcal{A}, \mathcal{P}, \mathcal{C},\gamma)$, 
where $\mathcal{S}$ and $\mathcal{A}$  
denote the  state space of the system and the 
action space respectively, $\mathcal{P}$ represents the 
transition probability matrix, with $\mathcal{P}\left(s_{t+1}| s_t, a_t\right)$ signifying 
the probability of transition from state $s_t$ to state $s_{t+1}$ under action $a_t$.
Earlier research has shown that the state transition equation in the SOC model, coupled with a random variable $\xi_t$, determines the transition probability matrix in the MDP model at episode $t$. 
In particular, the state transition process in the MDP model can be described by the equation $s_{t+1} = g_t(s_t, a_t, \xi_t)$, where $g_t: \mathcal{S} \times \mathcal{A} \times \Xi \to \mathcal{S}$ is a state transition mapping (see e.g., \cite{lin2022bayesian, shapiro2022distributionally}). 
Here, $\xi_t$ is 
a random vector mapping from
probability space $(\Omega,{\cal F},\mathbb{P})$ with support set $\Xi$.
This relationship can be established by setting $\mathcal{P}(s_{t+1} = s' \mid s_t, a_t) = P(\xi_t = \xi')$, where $\xi' \in \Xi$ is a realization of $\xi_t$ such that $g_t(s_t, a_t, \xi') = s'$.
Indeed, 
it is well known in the risk neutral setting, that 
SOC and MDP models are equivalent (\cite{shapiro2022distributionally}). 
The cost function
$\mathcal{C}_t\left(s_t, a_t, \xi_{t}\right)$
quantifies the immediate cost incurred at episode $t$ and the parameter 
$\gamma\in (0,1]$ represents
the discount factor of the cost function. 
To ease the exposition, 
we 
restrict the discussions to deterministic Markovian policies 
and 
justify it in Section 5. 
A deterministic Markovian policy is a sequence $\pi = (\pi_1,\cdots,\pi_T),$ where $ \pi_t$  maps each state in $\mathcal{S}$ to an action in $\mathcal{A}$.
The objective of a standard SOC/MDP is to find an optimal policy $\pi^*$ that minimizes the expected cumulative cost over all initial states $s_1\in\mathcal{S}$, formulated as 
\begin{eqnarray} 
\label{eq:MDP-old}
\min_\pi \mathbb{E}^{\pi}\left[\sum_{t=1}^{T}\gamma^t \mathcal{C}_t\left(s_t, a_t, \xi_{t}\right)\right],
\end{eqnarray} 
where $\mathbb{E}^{\pi}$ 
denotes the mathematical 
expectation with respect to (w.r.t.~for short) the joint probability distribution of 
$(\xi_1,\cdots,\xi_T)$ for each fixed policy $\pi$ and 
$T \in {\mathbb{N}^+ \cup \{\infty\}}$ represents time horizon. In the case that $T=\infty$, (\ref{eq:MDP-old})
becomes a SOC/MDP with infinite time horizon. In that case, the discount factor $\gamma$ is restricted to taking values over $(0,1)$.

\subsection{Topology of weak convergence and risk measures}
Consider a probability space $(\Omega, \mathcal{F}, \mathbb{P})$ and 
an $\mathcal{F}$-measurable function $X: \Omega \to \mathbb{R}$, 
where $X$ represents random losses.
Let $P=\mathbb{P}\circ X^{-1}$ denote 
the probability measure/distribution \footnote{Throughout the paper,
we use the terminologies probability measure and probability distribution interchangeably. } on $\mathbb{R}$ induced by $X$, let $F_P$ denote the cumulative distribution function (cdf) of $X$ and $F_P^{\leftarrow}$
be the left quantile function, that is,
$
F_P^{\leftarrow}(t)= \inf\{x: F_P(x)\geq t,\ \forall t\in [0,1]\}.
$
Let $p\in [0,\infty)$ and  $L^p$ denote the 
space of random variables defined over $(\Omega,{\cal F},\mathbb{P})$ with $p$-th order finite moments $
\mathbb{E}_{\mathbb{P}}[|X|^p]:=
\int_\mathbb{R} |X(\omega)|^p\mathbb{P}(d\omega)<\infty$.   Let $\mathscr{P}(\mathbb{R})$ denote the space of all probability measures over $(\mathbb{R},{\cal B})$ and 
\begin{eqnarray} 
\label{eq:M_1^p}
{\cal M}^p_1 =\left\{Q\in \mathscr{P}(\mathbb{R}): 
\int_\mathbb{R} |x|^p Q(dx)
<\infty\right\}.
\end{eqnarray}
 In the case when $p=0$, ${\cal M}_1^0=\mathscr{P}(\mathbb{R})$ and in the case  when  $p\to\infty$, 
 ${\cal M}_1^p$
 reduces to 
 the set of all probability measures $Q$ with
 $Q([-1,1])>0$ and $Q(\mathbb{R}\backslash [-1,1])=0$. 
Note also that 
$X\in L^p$ iff
$Q
=\mathbb{P}\circ X^{-1}
\in {\cal M}_1^p$
in that
$\int_\mathbb{R} |x|^p Q(dx)
= 
\int_\Omega |X(\omega)|^p \mathbb{P}(d\omega)$. 
Define $\mathfrak{C}^p_1$ as the linear space of all continuous functions $h:\mathbb{R}\to \mathbb{R}$ for which there exists a positive constant $c$ such that
\begin{eqnarray}
|h(x)|\leq c(|x|^p+1),\ \forall x\in \mathbb{R}.
\label{eq:ht-phi-topo}
\end{eqnarray}
In the case when $p=\infty$,
$h(x)$ is bounded by $c$ 
over $[-1,1]$. 
By \cite[Proposition 2.1]{zhang2024statistical},
the Fortet-Mourier metric $\mathsf {d\kern -0.07em l}_{FM}$ (see definition in appendix)
metricizes the $\psi$-weak topology on 
${\cal M}^p_1$ for $\psi(x) =c(|x|^p+1)$, 
denoted by $\tau_{|\cdot|^p}$, which is the coarsest topology on
${\cal M}_1^p$ for which the mapping $g_h:{\cal M}_1^p\to \mathbb{R}$ defined by
$$
g_h(Q) :=\int_{\mathbb{R}^k} h(x) Q(dx),\; h\in \mathfrak{C}_1^p
$$
 is continuous. A sequence  $\{P_k\} \subset {\cal M}_1^p$ is said to converge $|\cdot|^p$-weakly to $P\in {\cal M}_1^p$, written
 ${P_k} \xrightarrow[]{|\cdot|^p} P$, if it converges with respect to (w.r.t.) $\tau_{|\cdot|^p}$.
 In the case when $p=0$, it reduces to the usual topology of weak convergence.
 Recall that a function $\rho(\cdot):
 L^p\rightarrow \mathbb{R}$ is
 called a {\em monetary  risk measure} if it satisfies: 
 (a) For any $X,\ Y \in 
 L^p
 $,
$X(\omega) \geqslant Y(\omega)$ for all $\omega \in \Omega$
implies that $\rho(X) \geqslant \rho(Y)$, and 
	(b)  $\rho(X+c)=\rho(X)+c$ for any $X \in L^p$ and real number $c \in \mathbb{R}$. 
A monetary risk measure $\rho(\cdot)$ is said to be {\em convex} 
 if it also satisfies:	(c) for any $X, Y \in L^p$ and $\lambda \in[0,1],$ 
 $\rho(\lambda X+(1-\lambda)Y) \leq \lambda\rho(X)+(1-\lambda)\rho(Y)$.
	A convex risk measure $\rho(\cdot)$ is said to be {\em coherent} 
 if it further satisfies
	(d)  $\text{for any } X \in L^p \text{ and any } \lambda \geqslant 0, \text{ it holds that } \rho(\lambda X) = \lambda\rho(X).$
A risk measure $\rho$ is {\em law invariant} if 
$\rho(X)=\rho(Y)$  
for any random variables $X,Y\in L^p$ with the same probability distribution.  

Note that  in \cite{shapiro2021lectures}, $F^{\leftarrow}_P(1-\alpha)$
is called {\em value-at-risk} (VaR)\footnote{In some references,
VaR is defined as the quantile function at $\alpha$, see e.g.~\cite{rockafellar2000optimization,acerbi2002spectral}. Here we follow the definition in \cite{shapiro2021lectures}.}, which is a monetary risk measure satisfying positive homogeneity but not convexity. By convention, 
we denote it by $\text{VaR}^\alpha_{P}(X)$, for a random variable $X$ with $\alpha\in [0,1]$.
Let $\sigma:[0,1]\to \mathbb{R}$ be a non-negative, non-decreasing function with the normalized property $\int_0^1\sigma(t)dt=1$ and $\mathfrak{S}$ denote the set of all such $\sigma$.
The {\em spectral risk measure} of $X$ with $\sigma\in \mathfrak{S}$, is defined as
\begin{eqnarray}
\text{(SRM)} \quad M_\sigma(X) :=\int_0^1  F^{\leftarrow}_P(t)\sigma(t)dt
= \int_0^1 \text{VaR}^{1-t}_{P}(X)\sigma(t)dt,
\label{eq:SRM}
\end{eqnarray}
where $\sigma$ is called a {\em risk spectrum} representing the DM's risk preference.
Obviously $\sigma$ 
plays a role of weighting and $M_\sigma(X)$ is the weighted average loss of $X$. Moreover, by changing variables in the integral (\ref{eq:SRM}) (setting 
$x=F^{\leftarrow}_P(t)$), 
we obtain
\begin{eqnarray}
\label{eq:SRM-DRO}
M_\sigma(X)=\int_{-\infty}^{\infty} x\sigma(F_{P}(x))dF_{P}(x),
\end{eqnarray}
in which case $\sigma$ may be interpreted as a distortion function of probability.
Thus SRM is also known as distortion risk measure (\cite{gzyl2008relationship}) although 
the latter has a slightly different representation.  
SRM is introduced by Acerbi \cite{acerbi2002spectral}.
It shows that $M_\sigma(X)$ is a law invariant coherent risk measure.
The following example lists a few well-known spectral risk measures in the literature.

\begin{example} 
\label{ex:spectral-risk-measure-list}
By choosing some specific risk spectra, we can recover a number of well-known law invariant coherent risk measures.
\begin{itemize}
    
\item[(i)] 
Average Value-at-Risk (AVaR).\footnote{
AVaR is also known as conditional value-at-risk (CVaR), see \cite{rockafellar2000optimization}.
Here we follow the definition and terminology from \cite{shapiro2021lectures}.
The relationship between the two is
$\text{AVaR}^\alpha_P(X)=\text{CVaR}^{1-\alpha}_P(X)$.
} 
Let $\sigma_{\alpha}(\tau):=\frac{1}{\alpha}\mathds{1}_{[1-\alpha,1]}(\tau)$,
where $\mathds{1}_{[1-\alpha,1]}$ is the indicator function and  $\alpha \in (0,1)$. 
Then
\begin{eqnarray} 
  M_{\sigma_{\alpha}}(X) = \frac{1}{\alpha}\int_{1-\alpha}^1  F^{\leftarrow}_P(t)dt
  =: \text{AVaR}^\alpha_P(X).
\end{eqnarray} 
\item[(ii)] Wang's proportional hazards transform or power distortion \cite{wozabal2014robustifying}.
Let
$
\sigma_{\nu}(\tau) = \nu(1-\tau)^{\nu-1},
$
where $\nu\in (0,1]$ is a parameter.
Then
$
 M_{\sigma_{\nu}}(X) = \int_{0}^{\infty}(1-F_P(t))^{\nu}dt.
$

\item[(iii)] Gini's measure \cite{wozabal2014robustifying}. Let
$
\sigma_s(\tau) = (1-s) +2s\tau,
$
where $s\in (0,1)$.
Then
$
M_{\sigma_{s}}(X)  =\mathbb{E}_P[X]+s\mathbb{E}_P(|X-X'|)=: \text{Gini}_s(X),
$
where $X'$ is an independent copy of $X$.

\item[(iv)] Convex combination of expected value and $\text{AVaR}$.
Let
\begin{eqnarray*} 
\sigma_{\lambda}(\tau)=\lambda \mathds{1}_{[0,1]}(\tau)+(1-\lambda)\frac{1}{\alpha}\mathds{1}_{[1-\alpha,1]}(\tau)
\end{eqnarray*} 
for some $\lambda\in [0,1]$, where $\mathds{1}_{[a,b]}$ is the indicator function over interval $[a,b]$. Then
\begin{eqnarray} 
M_{\sigma_{\lambda}}(X) =\lambda \mathbb{E}_P[X]+(1-\lambda)\text{AVaR}_P^{\alpha}(X).
\end{eqnarray} 

\end{itemize}
\end{example}

It may be helpful to note that any law invariant
risk measure defined over a non-atomic $L^p$ space can be represented as a risk functional over the space of the probability distributions of the random variables in the space.
Specifically, there exists a unique 
functional $\varrho: {\cal M}_1^p\to 
\mathbb{R}$ such that
\begin{eqnarray} 
\label{eq:varrho}
\rho(X) = \varrho(P) := \rho(F_P^{-1}(U)),
\end{eqnarray} 
where $U$ is a random variable uniformly distributed over $[0,1]$. 
In the case when $\rho$ is a coherent risk measure and $X,Y\in L^1$,
we have
\begin{equation}
\label{eq:Lip-RM-rho}
    |\rho(X)-\rho(Y)|=|\varrho(P)
    -\varrho(Q)|
    \leq L\mathsf {d\kern -0.07em l}_K(P,Q),
\end{equation}
where $L$ is the Lipschitz modulus of $\rho$, 
$Q$ denotes the probability distribution of $Y$, 
and $\mathsf {d\kern -0.07em l}_K$ denotes the Kantorovich metric (see the definition in the appendix),
see the proof of \cite[Corollary 4.8]{wang2021quantitative}.
Moreover, 
any monetary risk measure $\rho:L^\infty\to\mathbb{R}$ 
is Lipschitz continuous with respect to the Wasserstein distance $\mathsf {d\kern -0.07em l}_W^\infty$, that is,
    \begin{eqnarray} 
    |\rho(X)-\rho(Y)|\leq \mathsf {d\kern -0.07em l}_W^\infty(P,Q)=
    \sup_{0<u<1}|F_P^{\leftarrow}(u)-F_Q^{\leftarrow}(u)| =\|X-Y\|_\infty,
    \label{eq:Lip-non-expansive}
    \end{eqnarray} 
    see e.g.~\cite[Lemma 2.1]{weber2006distribution}. The continuity of 
    a risk measure of a random function 
    w.r.t.~the parameters of the function is slightly more complicated. The next lemma addresses this.

{\color{black}
\begin{lemma}
\label{Prop:continuity}
Let $X\in L^p$ be a random variable with support ${\cal X}$ and $z\in
{\cal Z}\subset \mathbb{R}^n$ be a decision vector.
Let $f:\mathcal{Z}\times {\cal X}\to\mathbb{R}$ be a continuous function  such that
\begin{eqnarray} 
|f(z,x)| \leq \kappa(|x|^\upsilon+1),\ \forall 
z\in {\cal Z}, x\in {\cal X}
\label{eq:growth}
\end{eqnarray} 
for some positive constants $\kappa,\upsilon$.
Let $P_{X} \in {\cal M}_1^{\upsilon p}$ denote the probability distribution of $X$ over $\mathbb{R}$, and for any fixed $z$, let  
$\nu_{z,P_X} = P_X \circ f^{-1}(z,\cdot)\in {\cal M}_1^{p}$
be a probability measure on $\mathbb{R}$ induced by 
$f(z,X)$.
Let $F_{\nu_{z,P_{X} }}(x)=: \nu_{z,P_X}((-\infty,x])$
for any $x\in\mathbb{R}$.
Then the following assertions hold.
\begin{itemize}
\item[(i)]  $\nu_{z,P_X }$ is continuous in  $(z,P_X)$ with respect to topology $\tau_\mathbb{R}\times \tau_\mathbb{R}\times \tau_{|\cdot|^{\upsilon p}}$ and 
$\tau^p$.

\item[(ii)]  Let
\begin{eqnarray} M_\sigma(f(z,X)):=\int_0^1
F_{\nu_{z, P_X}}^{\leftarrow}(x)
\sigma(x)dx
\end{eqnarray} 
be a spectral risk measure of $f(z,X)$,
where \(\sigma\) is a risk spectrum.
Let $\mathfrak{A}$ be a set of risk spectra. Suppose that there exists a positive number $q>1$ such that
$\mathfrak{A}\subset {\cal L}^q[0,1]$
and there is a positive constant $p$ with $\frac{1}{p}+\frac{1}{q}=1$ such that
$P_X\in {\cal M}_1^{\upsilon p}$ and $\mathfrak{A}$ is a compact set. Then
\end{itemize}
\begin{eqnarray} 
\left|\sup_{\sigma\in \mathfrak{A}} M_\sigma(f(z',X'))
-\sup_{\sigma\in \mathfrak{A}} M_\sigma(f(z,X))
\right|
\leq \mathsf {d\kern -0.07em l}_W^p(\nu_{z',P_{X'}},\nu_{z,P_X})\sup_{\sigma\in \mathfrak{A}}\left[\int_0^1\sigma(x)^qdx\right]^{\frac{1}{q}},
\end{eqnarray}
where $\mathsf {d\kern -0.07em l}_W^p$ is the Wasserstein distance of order $p$.
In the special case that $\mathfrak{A}$ is a singleton, 
$M_\sigma(f(z,X))$ is continuous in $(z,P_X)$ with respect to topology $\tau_\mathbb{R}\times \tau_\mathbb{R}\times \tau_{|\cdot|^{\upsilon p}}$ and 
$\tau_\mathbb{R}$.

\item[(iii)] For any law invariant continuous risk measure $\rho$, 
$\rho(f(z,X))$
is continuous in
$(z,P_X)$ with respect to topology $\tau_\mathbb{R}\times \tau_\mathbb{R}\times \tau_{|\cdot|^{\upsilon p}}$ and 
$\tau_\mathbb{R}$.

\item[(iv)] In the case that ${\cal X}$ is
compact,
$\rho(f(z,X))$
is continuous in
$(z, P_X)$ for any law invariant monetary risk measure $\rho$
with respect to topology $\tau_\mathbb{R}\times \tau_\mathbb{R}\times \tau_{|\cdot|^0}$ and 
$\tau_\mathbb{R}$, where $\tau_{|\cdot|^0}$ denotes the weak topology.

\end{lemma}

\noindent
\textbf{Proof.}
Part (i) follows directly from \cite[Lemma 2.74]{claus2016advancing}.
Part (ii) follows from \cite[Proposition 5.1]{wang2020robust}.
Part (iii) follows from Part (i).
Part (iv) follows from the fact that $f$ is bounded.
\hfill $\Box$}

\subsection{Bayesian composite risk measure}

Consider a random variable $X:(\Omega,{\cal F},P)\to\mathbb{R}$
whose true distribution belongs to a family of parametric distributions, denoted by $\{P_{\theta}\}\subset\mathcal{M}_1^p$, where $\theta\in\Theta\subseteq\mathbb{R}^d$
is a vector of parameters.
The true probability distribution is associated with an unknown underlying parameter ${{\theta}^c}$, i.e., $P^c=P_{{\theta}^c}$  for some ${\theta}^c\in\Theta\subseteq\mathbb{R}^d$.
To ease the exposition,
we consider the case that $d=1$. All of our results in the forthcoming discussions are applicable to the case when $d\geq 1$. 
Since the information on $\theta^c$ is 
incomplete, it might be sensible 
to describe $\theta$ as a random variable 
with prior distribution $\mu$
from a modeler's perspective. This distribution may be obtained from partially available information and/or subjective judgment.
For each realization of $\theta$, $P_\theta$ gives rise to a probability distribution of $X$.
In practice, we may not be able to obtain a closed form of $P_\theta$, rather we may 
obtain samples of $X$,
denoted by $\pmb{X}^N = \{X^1, \cdots, X^N\}$ 
and use them to estimate $\theta^c$ by maximum likelihood. 
Alternatively, we may use Bayes' formula to obtain a posterior distribution $Q_\mu$ of $\theta$, whose probability density function (pdf) is defined as
	\begin{equation}		\mu(\theta|\pmb{X}^N)=\frac{p(\pmb{X}^N|\theta)\mu(\theta)}{\int_\Theta p(\pmb{X}^N|\theta)\mu(\theta)d\theta},
 \label{eq:Bayes-post-distr}
	\end{equation}
where $p(\cdot|\theta)$ is the pdf of $X$ associated with $P_{\theta}$, $p(\pmb{X}^N|\theta)= \prod_{i=1}^N p\left(X^i| \theta\right)$ is the likelihood function and $\int_\Theta p(\pmb{X}^N|\theta)\mu(\theta)d\theta$ is the normalizing factor. As more and more samples are gathered, $\mu(\theta|\pmb{X}^N)$
converges to the Dirac distribution of $\theta$ at $\theta^c$ and subsequently the BCR converges to $\rho_{P_{\theta^c}}$.	
In this paper, we focus on the latter approach and apply it within the framework of SOC/MDP.	
 
To quantify the risk of $X$ (which represents losses), we may adopt a law invariant risk measure, that is, $\rho(X)$. The notation does not capture the probability distribution of $X$. However, if we use the risk functional 
$\varrho$ defined in (\ref{eq:varrho}), then equivalently we can write 
$\rho(X)$ as $\varrho(P_\theta)$. Since $\theta$ is a random variable, then $\varrho(P_\theta)$ is a random function of $\theta$. There are two ways to quantify the risk of $\varrho(P_\theta)$. One is to consider the mean value of $\varrho(P_\theta)$ with respect to the posterior distribution of $\theta$, that is, $\mathbb{E}_{\mu(\theta|\pmb{X}^N)}(\varrho(P_\theta))$.
The other is to adopt a law invariant risk measure $\hat{\rho}$ to quantify the risk of $\varrho(P_\theta)$, that is, 
$\hat{\rho}(\varrho(P_\theta))$. By using the equivalent risk function of $\hat{\rho}$, written $\hat{\varrho}$, we 
can write $\hat{\rho}(\varrho(P_\theta))$ as
$\hat{\varrho}(\eta)$, where $\eta$ denotes the probability distribution of $\varrho(P_\theta)$.
Unfortunately, this notation is too complex. To ease the exposition, we write 
$\rho_{P_\theta}(X)$ for $\rho(X)$ and $\rho_\mu\circ\rho_{P_\theta}(X)$ for $\hat{\varrho}(\eta)$ even though this notation is not consistent with our previous discussions.
We call $\rho_\mu\circ\rho_{P_\theta}(\cdot)$ a {\em Bayesian composite risk (BCR) measure} with $\rho_{P_{\theta}}$ being the {\em inner} risk measure conditional on $\theta$ and $\rho_{\mu}$ being the {\em outer} risk measure. 
This composite framework 
was proposed by Qian et al.~\cite{qian2019composite}.
In a particular case that 
$\rho_{P_{\theta}}$ is a spectral risk measure parameterized by $\theta$ and 
$\rho_\mu$ is the expectation, 
the composite risk measure recovers
the average randomized spectral risk measure 
introduced by Li et al.~\cite{li2022randomization}.
The next proposition states the well-definedness of the BCR.

\begin{proposition}
\label{Prop:BCD}
Consider the setting 
of Lemma~\ref{Prop:continuity}.
Assume: (a) the inequality (\ref{eq:growth}) holds, (b) $p(\cdot|\theta)$ is continuous in $\theta$ and (c) $\rho_{P_\theta}$ is any law invariant continuous risk measure. 
Then 
$\rho_\mu\circ\rho_{P_\theta}({f}(z,X))$ is well defined.
\end{proposition}

 \section{The BCR-SOC/MDP model} 

In this section,
we move on from the conventional SOC/MDP model (\ref{eq:MDP-old}) to propose a BCR-SOC/MDP model 
by adopting BCR to quantify the risk of loss during
each episode  and discuss its relationship with some existing SOC/MDP models in the literature. 

Specifically we consider 
	\begin{subequations}
 \label{eq:MDP-BCR}
	\begin{eqnarray}
	    	\min _\pi &&\mathcal{J}_T^\gamma(\pi,s_1,\mu_1):=\rho_{\mu_1}\circ\rho_{P_{\theta_1}}\left[\mathcal{C}_1\left(s_1, a_1, \xi_1\right)+\cdots+\gamma\rho_{ \mu_{T-1}}\circ \rho_{P_{\theta_{T-1}}}\left[\mathcal{C}_{T-1}\left(s_{T-1}, a_{T-1}, \xi_{T-1}\right)+\gamma\mathcal{C}_T\left(s_T\right)\right]\right]
      \label{eq:MDP-BCR-a} \nonumber\\
			\text {s.t.} && s_{t+1}=g_t\left(s_t, a_t, \xi_t\right),\ 
   a_t=\pi_t(s_t,\mu_t),\ t=1, \cdots, T-1,
    \label{eq:MDP-BCR-b}
   \\
			&& \mu_{t+1}(\theta)=\frac{p\left(\xi_t |\theta\right)\mu_t(\theta) }{\int_\Theta p\left(\xi_t |\theta\right)\mu_t(\theta)  d \theta}, t=1, \cdots, T-1,
 \label{eq:MDP-BCR-c}
\end{eqnarray}		
 \end{subequations}
where 
$\theta_t$, $t=1,\cdots, T-1$, 
is a random 
variable with 
pdf $\mu_t$ at episode $t$,
the cost function $\mathcal{C}_t(s_t,a_t,\xi_t)$ is 
continuous with respect to $\xi_t$ and $\pi_t: \mathcal{S} \times \mathcal{D}_1^p \rightarrow \mathcal{A}$ is the policy for action at episode $t$, where $\mathcal{D}_1^p$ denotes the set of all pdfs whose corresponding probability distributions are in ${\cal M}_1^p$.  Compared to (\ref{eq:MDP-old}),
the BCR-SOC/MDP model has several new features.

\begin{itemize}

\item[(i)] For $t=1,\cdots, T-1$, the true probability distribution of $\xi_t$ is unknown, but it is known that it belongs to a family of parametric distributions, written as $P_{\theta}$, where
  $\theta$ is a random variable with prior distribution $\mu_1(\theta)$,  see \cite{shapiro2023episodic}.
In the objective function of the BCR-SOC/MDP model (\ref{eq:MDP-BCR}), we write $\theta_t$
for $P_\theta$ at episode $t$ to indicate that 
$\theta_t$ corresponds to the pdf $\mu_t$.
In the initial episode, the epistemic uncertainty is based on prior data and/or subjective judgment, which leads to a prior distribution $\mu:=\mu_1$ over $\Theta$.
As the stochastic process $\boldsymbol{\xi}^{t}:=\{\xi_1,\xi_2, \cdots, \xi_t\}$ is observed,
 the posterior distribution $\mu_{t+1}(\theta):=\mu(\theta|\boldsymbol{\xi}^{t})$ is updated with the 
 Bayes' formula:
 \begin{eqnarray}
  \mu_{t+1}(\theta)=\frac{p(\boldsymbol{\xi}^{t}|\theta)\mu_1(\theta)}{\int_\Theta p(\boldsymbol{\xi}^{t}|\theta)\mu_1(\theta)d\theta}=\frac{p\left(\xi_t |\theta\right)\mu_t(\theta) }{\int_\Theta p\left(\xi_t |\theta\right)\mu_t(\theta)  d \theta}. 
 \end{eqnarray}
Moreover, instead of using the mean values as in \cite{shapiro2023episodic}, we propose to use law invariant risk measures
 $\rho_{P_{\theta}}$ and $\rho_{\mu_t}$ to quantify 
the risks arising from  aleatoric uncertainty and  epistemic uncertainty respectively during each episode. This is 
intended to provide a complement to
\cite{shapiro2023episodic},
where some DMs are risk-averse 
rather than risk neutral. Here, $\rho_{P_{\theta_t}}$ quantifies
the risk arising from aleatoric uncertainty $\xi_t$ (which is standard in risk management),
and $\rho_{\mu_t}$ assesses the effect of epistemic uncertainty $\theta_t$ on 
the performance of 
$\rho_{P_{\theta_t}}(\mathcal{C}_t(s_t,a_t,\xi_t))$.
In the case that 
$\theta_t$ is contextual information
related to $\xi_t$, we may 
replace the composite risk measure with a single risk measure defined over the joint distribution of $\theta_t$ and $\xi_t$, see \cite{tao2025risk}.
The above two ways may differ depending on the DM's risk preferences in the two types of uncertainties.
The time dependence of the risk measures allows us to capture 
the DM's dynamic risk preferences 
as more information about both uncertainties is obtained. This  makes the new BCR-SOC/MDP model more flexible than the conventional SOC/MDP model (\ref{eq:MDP-old}).

\item[(ii)] 
Unlike \cite{shapiro2023episodic}
where $\pi_t$ is purely based on the physical state $s_t$, the policy $\pi_t$ in the BCR-SOC/MDP model depends not only on $s_t$ but also on the belief 
$\mu_t$, as considered in \cite{lin2022bayesian}. This is 
motivated by the fact that $\mu_t$ signifies how $\theta_t$ takes values over $\Theta$ and subsequently 
affects
the probability distribution $P_{\theta_t}$ of $\xi_t$. If we interpret $\xi_t$ as the underlying uncertainty that affects the loss function ${\cal C}_t$, then $\mu_t$ may be regarded as the DM's belief about the prospect of the future market, such as a bull market or bear market, at state $t$. Such belief is accumulated over the decision-making process over the past $t-1$ episodes and affects the DM's decision at episode $t$. Nevertheless, the above augmented pair of states might 
bring some computational difficulties 
as envisaged by Shapiro et al.~\cite{shapiro2023episodic}, 
where 
the authors 
propose a so-called episodic Bayesian stochastic optimal control framework to avoid explicit 
inclusion of $\mu_t$ as 
an augmented state. In the forthcoming discussions, we will demonstrate how to overcome
the resulting computational difficulties.

\item[(iii)] The SOC/MDP model serves as an effective tool to demonstrate various foundational concepts because it can be formulated within the multi-stage stochastic programming (MSP) framework \cite{dommel2021foundations,dupavcova2002comparison,shapiro2021tutorial}, by considering $x_t=(x_t^s,x_t^a)$ with $x_t^s=s_t$ and $x_t^a=a_t$, as an augmented decision variable.
Specifically, the BCR-SOC/MDP model can be reformulated as a risk-averse MSP model in \cite{pflug2014multistage} as follows:
	\begin{subequations}
	\begin{eqnarray}
		\min _{x} && \rho_{\xi_1}\left[\mathcal{C}_1\left(x_1, \xi_1\right)+\cdots+\gamma\rho_{\xi_{T-1}|\boldsymbol{\xi}^{T-2}}\left[\mathcal{C}_{T-1}\left(x_{T-1}, \xi_{T-1}\right)+\gamma\mathcal{C}_T\left(x_T\right)\right]\right] \\
		\text { s.t. } && x_{t+1}^s=g_t\left(x_t, \xi_t\right), 
		 x_{t+1}^a\in\mathcal{A}_t, t=1, \cdots, T-1 ,
	\end{eqnarray}
\end{subequations}
where $\rho_{\xi_{t+1}|\boldsymbol{\xi}^{t}}:= \rho_{\mu_{t+1}}\circ\rho_{P_{\theta_{t+1}}}$ with $\mu_t$ 
being updated according to (\ref{eq:MDP-BCR-c}) and the stochastic process $\boldsymbol{\xi}^{t}=\{\xi_1,\xi_2, \cdots, \xi_t\}$. For a detailed example, we refer the readers to \cite{shapiro2021lectures}, which illustrates how the classical inventory model can be adapted to fit the above framework. Unlike usual risk‐averse MSP models, the BCR-MDP model \eqref{eq:MDP-BCR}
explicitly distinguishes aleatoric uncertainty $\xi_t$ and epistemic uncertainty $\theta_t$
to reflect learnability of the latter during 
the dynamic decision-making process.
\end{itemize}

To ensure the BCR-SOC/MDP model to be well-defined, we make the following assumption.

 \begin{assumption}
	\label{ass-bcr-1} 
(a) For $t=1,\cdots,T$, the cost function ${\cal C}_t(s_t,a_t,\xi_t):\mathcal{S}\times\mathcal{A}\times \Xi
\to \mathbb{R}$ is continuous,  
and  there exist  positive constants $\kappa_{\cal C},\upsilon$ such that 
\begin{eqnarray} 
|{\cal C}_t(s_t,a_t,\xi_t)| \leq \kappa_{\cal C}(|\xi_t|^\upsilon+1),\ \forall  s_t,a_t,\xi_t,\ t=1,\cdots,T-1 
\label{eq:growth-of-C}
\end{eqnarray} 
holds, and
 \begin{eqnarray} 
 \label{eq:growth-of-C-T}
|\mathcal{C}_{T}\left( g_{T-1}\left(s_{T-1}, a_{T-1}, \xi_{T-1}\right)\right)|\leq \kappa_{\cal C}(|\xi_{T-1}|^\upsilon+1), \ \forall  s_{T-1},a_{T-1},\xi_{T-1}.
 \end{eqnarray} 
(b) 
For $t=1,\cdots, T-1$, 
$\rho_{P_{\theta_t}}:L^\upsilon\to \mathbb{R}$
is
a continuous law invariant monetary 
risk measure such that
$\rho_{P_{\theta_t}}(|\xi_{t}|^\upsilon)<\infty,\ \forall \theta_t\in\Theta$, and
 \begin{eqnarray} 
 \label{eq-condition-of-risk}
\rho_{\mu_t}\circ\rho_{P_{\theta_t}}(|\xi_{t}|^\upsilon)\leq\kappa_\rho(|\xi_{t-1}|^\upsilon+1),
 \end{eqnarray} 
 where $\kappa_\rho$ is a positive constant depending on $t$.
(c) The random variables $\xi_1,\cdots,\xi_T$
 are independent and identically distributed (i.i.d.) 
 with  true unknown distribution $P_{\theta^c}$ which belongs to 
 a parametric family of distributions 
 $\{P_{\theta}\mid \theta \in \Theta\}$, where 
 $P_\theta\in {\cal M}_1^{\upsilon p}$ for all $\theta\in\Theta$. 
(d) The pdf 
$p(\xi|\theta)$ is continuous in $\theta$ over $\Theta\subset \mathbb{R}$. (e) The action space $\mathcal{A}$ is  a compact and convex set.
 \end{assumption}

The following comments about how the specified conditions in the 
assumption can be satisfied might be helpful. 
Inequalities 
(\ref{eq:growth-of-C})-(\ref{eq:growth-of-C-T}) in 
Assumption~\ref{ass-bcr-1} (a) specify the growth conditions of cost functions in terms of $\xi_t$. This type of condition is widely used in stochastic programming, see~e.g., \cite{guo2022robust} and references therein, but it is a new consideration
for SOC/MDP models. 
In the literature of SOC/MDP models,
the cost functions are often assumed to be bounded. Here we relax the boundedness condition
to cover the case where the dependence of ${\cal C}_t$ on $\xi_t$ is non-linear and  the support set of $\xi_t$ can be unbounded. Condition (b)
is closely related to (a). In 
the case where $\upsilon=0$ in (\ref{eq:growth-of-C-T}), 
condition (a) subsumes the case that 
${\cal C}_t$ is bounded. By
setting $\kappa_\rho=1$, condition (b) will also accommodate this case.
We will return to this
shortly to show how condition (b) can be satisfied under some specific circumstances. 
Assumption~\ref{ass-bcr-1} (c) is the set-up of the
BCR-SOC/MDP model. The independence of $\xi_1,\cdots,\xi_T$
is usually assumed in SOC/MDP models and holds automatically in many scenarios.
Assumption~\ref{ass-bcr-1} (d) is 
also satisfied by many probability distributions and
is widely used in the literature, see e.g., \cite{shapiro2023episodic}. 

Inequality (\ref{eq-condition-of-risk})
in Assumption~\ref{ass-bcr-1} (b)
poses a bound on $|\xi_t|^\upsilon$
under the BCR. Since  ${\mu_t}$ depends on $\xi_{t-1}$, the upper bound of the BCR depends on $\xi_{t-1}$.
Inequality 
(\ref{eq-condition-of-risk})
requires the bound 
to be controllable 
by $|\xi_{t-1}|^\upsilon$.
This kind of condition is less intuitive,
so we use an example to explain how the condition can be satisfied in some important cases.

\begin{example}\label{example-2}
 Inequality 
 \eqref{eq-condition-of-risk} 
 holds under the following 
 specific cases. 

\begin{itemize}
    
\item[(a)] \underline{The support set of $\xi_t$, $\Xi$, is bounded}.
In this case,
$
\rho_{\mu_t}\circ\rho_{P_{\theta_t}}(|\xi_{t}|^\upsilon)$ is 
bounded by a positive constant $C$.

\item[(b)] 
\underline{Both $\xi_t$ and $\theta_t$ follow a normal distribution with 
$\xi_t\sim \mathcal{N}(\theta_t, \sigma^2)$ and 
$\theta_t\sim \mathcal{N}(m_t, d_t^2)$}.
Let $\mu_0(\theta) := \mathcal{N}(m_0, d_0^2)$. 
Then
\begin{equation}
\label{eq-update1}
        m_t = \lambda_t m_{t-1} + (1 - \lambda_t) \xi_{t-1}, \quad \lambda_t = \frac{d_{t-1}^2}{d_{t-1}^2 + \sigma^2}, \quad d_{t}^{-2}=d_{t-1}^{-2}+\sigma^{-2}.
\end{equation}
Since $\xi_t = \theta_t + \sigma Z$, 
then
\begin{eqnarray} 
|\xi_t|^\upsilon = |\theta_t + \sigma Z|^\upsilon\leq 2^{\upsilon - 1} \left( |\theta_t|^\upsilon + |\sigma Z|^\upsilon \right),
\end{eqnarray} 
where $Z \sim \mathcal{N}(0,1)$ and the inequality is based on the generalized Minkowski inequality.
By applying 
the inner risk measure $\rho_{P_{\theta_t}}$
to both sides of the inequality, 
we obtain
\begin{equation}\label{eq-example1}
   \rho_{P_{\theta_t}}(|\xi_t|^\upsilon) \leq 2^{\upsilon - 1} \left( |\theta_t|^\upsilon + \sigma^\upsilon \rho(|Z|^\upsilon)\right)=C_1|\theta_t|^\upsilon +C_2, 
\end{equation}
where $C_1 = 2^{\upsilon - 1}$ and $C_2 = C_1 \sigma^\upsilon \rho(|Z|^\upsilon)$.
Since $\theta_t \sim \mathcal{N}(m_t, d_t^2)$, 
then 
\begin{eqnarray} 
|\theta_t|^\upsilon \leq 2^{\upsilon - 1} \left( |m_t|^\upsilon + d_t^\upsilon |Z|^\upsilon \right).
\end{eqnarray} 
By applying the outer risk measure $\rho_{\mu_t}$ to both sides of the inequality above, we obtain
\begin{equation}\label{eq-example2}
    \rho_{\mu_t}(|\theta_t|^\upsilon) \leq 2^{\upsilon - 1} \left( |m_t|^\upsilon + d_t^\upsilon\rho(|Z|^\upsilon)\right).
\end{equation}
A combination of 
\eqref{eq-example1}-\eqref{eq-example2} yields
\begin{eqnarray} 
\rho_{\mu_t}\circ\rho_{P_{\theta_t}}(|\xi_t|^\upsilon) \leq C_3 |m_t|^\upsilon + C_4 d_t^\upsilon + C_2,
\label{eq:bnd-BCR-xi_t-normal-12}
\end{eqnarray} 
where $C_3 = C_1 \cdot 2^{\upsilon - 1}$, $C_4 = C_1 2^{\upsilon - 1} \rho(|Z|^\upsilon)$.
By applying \eqref{eq-update1} to the right-hand side of \eqref{eq:bnd-BCR-xi_t-normal-12}, 
we obtain
\[
\rho_{\mu_t}\circ\rho_{P_{\theta_t}}(|\xi_t|^\upsilon)\leq K |\xi_{t-1}|^\upsilon + C,
\]
where $K$ and $C$ are constants defined as
$
K = C_3 \cdot 2^{\upsilon - 1}$, $C = K  m_{t-1}^\upsilon + C_4 d_t^\upsilon + C_2
$ depending on $t$.

\item[(c)]
\underline{$\xi_t \sim \text{Exponential} 
 (\theta_t)$
and  $\theta_t\sim \text{Gamma}(m_t, d_t)$}.
Let $\mu_0(\theta)= \text{Gamma}(m_0, d_0)$.
Then the posterior distribution $\mu_t(\theta)= \text{Gamma}(m_t, d_t)$ can be updated based on $\xi_{t-1}$ as follows:
    \begin{eqnarray} 
    \label{eq:bnd-BCR-xi_t-exp-1}
    m_t = m_{t-1} + 1, \quad d_t = d_{t-1} + \xi_{t-1}.
    \end{eqnarray} 
We consider BCR with both the inner and outer risk measures being  expectation. 
Since $\xi_t \sim \text{Exponential}(\theta_t)$, then
\[
\mathbb{E}_{P_{\theta_t}}(|\xi_t|^\upsilon) = \int_0^\infty \xi_t^\upsilon \theta_t e^{-\theta_t \xi_t} d\xi_t = \frac{\Gamma(\upsilon + 1)}{\theta_t^\upsilon}
\]
and 
subsequently 
\[
\rho_{\mu_t}\circ\rho_{P_{\theta_t}}(|\xi_t|^\upsilon)=\mathbb{E}_{\mu_t}\mathbb{E}_{P_{\theta_t}}(|\xi_t|^\upsilon) = \Gamma(\upsilon + 1) \mathbb{E}_{\mu_t}\left( \frac{1}{\theta_t^\upsilon} \right).
\]
Moreover, since $\mu_t =\text{Gamma}(m_t, d_t)$, then we have
\[
\mathbb{E}_{\mu_t}\left( \frac{1}{\theta_t^\upsilon} \right) = \frac{d_t^{m_t}}{\Gamma(m_t)} \int_0^\infty \theta_t^{m_t - \upsilon - 1} e^{-d_t \theta_t} d\theta_t = \frac{d_t^\upsilon \Gamma(m_t - \upsilon)}{\Gamma(m_t)}
\]
for $m_t > \upsilon$. 
By exploiting \eqref{eq:bnd-BCR-xi_t-exp-1},
we obtain
\[
d_t^\upsilon = (d_{t-1} + \xi_{t-1})^\upsilon \leq 2^{\upsilon - 1} \left( d_{t-1}^\upsilon + |\xi_{t-1}|^\upsilon \right).
\]
Let 
\[
K = \Gamma(\upsilon + 1) \frac{\Gamma(m_t - \upsilon)}{\Gamma(m_t)} 2^{\upsilon - 1}, \quad C = K d_{t-1}^\upsilon.
\]
Then
\[
\rho_{\mu_t}\circ\rho_{P_{\theta_t}}(|\xi_t|^\upsilon)=\mathbb{E}_{\mu_t}\mathbb{E}_{P_{\theta_t}}(|\xi_t|^\upsilon)\leq K |\xi_{t-1}|^\upsilon + C.
\]

Next, we consider the case where the inner risk measure is $\rho_{P_{\theta_t}}=\text{AVaR}^\beta_{P_{\theta_t}}$. Since $\text{AVaR}^\beta_{P_{\theta_t}}(|\xi_t|^\upsilon)\leq \frac{1}{\beta}\mathbb{E}_{P_{\theta_t}}(|\xi_t|^\upsilon)$, the result follows in a similar manner to the case where $\rho_{P_{\theta_t}}=\mathbb{E}_{P_{\theta_t}}$.
Furthermore, consider the case where the inner risk measure is $\rho_{P_{\theta_t}}=\text{VaR}^\beta_{P_{\theta_t}}$. Since $\text{VaR}^\beta_{P_{\theta_t}}(|\xi_t|^\upsilon)\leq \text{AVaR}^\beta_{P_{\theta_t}}(|\xi_t|^\upsilon)$, the result can be deduced in a similar manner to the $\rho_{P_{\theta_t}}=\text{AVaR}^\beta_{P_{\theta_t}}(|\xi_t|^\upsilon)$ case.
For the outer risk measures $\rho_{\mu_t}=\text{VaR}^\alpha_{\mu_t}$ and $\rho_{\mu_t}=\text{AVaR}^\alpha_{\mu_t}$, similar results hold.

\item[(d)] 
\underline{$\xi_t\sim \text{Poisson}(\theta_t)$ and $\theta_t\sim \text{Gamma}(m_t, d_t)$}.
    By setting the prior distribution $\mu_0(\theta)=\text{Gamma}(m_0, d_0)$, the posterior distribution $\mu_t(\theta)= \text{Gamma}(m_t, d_t)$ can be updated based on sample 
    $\xi_{t-1}$ as follows:
    \begin{eqnarray} 
    m_t = m_{t-1} + \xi_{t-1}, \quad d_t = d_{t-1} + 1.
    \end{eqnarray} 
Consider the case that 
both the inner and outer risk measures are  expectations.
For $\xi_t\sim \text{Poisson}(\theta_t)$, the expectation is given by:
\[
\mathbb{E}_{P_{\theta_t}}(|\xi_t|^\upsilon) = \mathbb{E}_{P_{\theta_t}}(\xi_t^\upsilon) = \theta_t^\upsilon + \text{terms involving lower powers of } \theta_t.
\]
For simplicity, we can find constants $C_1$ and $C_2$ such that 
$\mathbb{E}_{P_{\theta_t}}(\xi_t^\upsilon) \leq C_1 \theta_t^\upsilon+C_2.$

Since $\mu_t =\text{Gamma}(m_t, d_t)$, we have:
\[
\mathbb{E}_{\mu_t}(\theta_t^\upsilon) = \frac{\Gamma(m_t + \upsilon)}{\Gamma(m_t)} d_t^{-\upsilon}.
\]
Substituting $m_t = m_{t-1} + \xi_{t-1}$ and $d_t = d_{t-1} + 1$ into the equation above, we obtain
\[
\mathbb{E}_{\mu_t}(\theta_t^\upsilon) = \frac{\Gamma(m_{t-1} + \xi_{t-1} + \upsilon)}{\Gamma(m_{t-1} + \xi_{t-1})} (d_{t-1} + 1)^{-\upsilon}.
\]
Using the properties of the Gamma function, we have 
\[
\frac{\Gamma(m_{t-1} + \xi_{t-1} + \upsilon)}{\Gamma(m_{t-1} + \xi_{t-1})} = \prod_{k=0}^{\upsilon - 1} (m_{t-1} + \xi_{t-1} + k).
\]
Consequently we can establish
\[
\mathbb{E}_{\mu_t}\mathbb{E}_{P_{\theta_t}}(|\xi_t|^\upsilon)\leq C_1\mathbb{E}_{\mu_t}(\theta_t^\upsilon)+C_2 \leq C_1(\xi_{t-1} + m_{t-1} + \upsilon - 1)^\upsilon (d_{t-1} + 1)^{-\upsilon}+C_2\leq K |\xi_{t-1}|^\upsilon + C,
\]
where $K = C_1(d_{t-1} + 1)^{-\upsilon}$ and $C = K (m_{t-1} + \upsilon - 1)^\upsilon+C_2$.
Similar results can be established when  VaR and AVaR are chosen for the inner and outer risk measures respectively. We omit the details.
\end{itemize}
   
\end{example}

We are now ready to address the well-definedness of the BCR-SOC/MDP problem (\ref{eq:MDP-BCR}).

\begin{proposition}
  Under Assumption~\ref{ass-bcr-1},
  the BCR-SOC/MDP problem (\ref{eq:MDP-BCR}) is well-defined.
\end{proposition}

\noindent
\textbf{Proof.}
 Let 
\begin{eqnarray*} 
\vt_{T}(s_{T}, a_{T}, \xi_{T},\mu_T) &:=&{\cal C}_T(s_T),\\
\vt_{t}(s_{t}, a_{t}, \xi_{t},\mu_{t}) &:=&
\mathcal{C}_{t}\left(s_{t}, a_{t}, \xi_{t}\right)+\gamma\rho_{\mu_{t+1}}\circ\rho_{P_{\theta_{t+1}}}\left(\vt_{t+1}(s_{t+1}, a_{t+1}, \xi_{t+1},\mu_{t+1})\right), t=1,\cdots, T-1.
\end{eqnarray*}
We begin by considering the case $t:=T-1$. Since $s_{T}=g_{T-1}\left(s_{T-1}, a_{T-1}, \xi_{T-1}\right)$, then 
\begin{eqnarray*}
\vt_{T-1}(s_{T-1}, a_{T-1}, \xi_{T-1},\mu_{T-1})&=&\mathcal{C}_{T-1}\left(s_{T-1}, a_{T-1}, \xi_{T-1}\right) +\gamma\rho_{\mu_{T}}\circ\rho_{P_{\theta_{T}}}\left(\vt_{T}(s_{T}, a_{T}, \xi_{T},\mu_T) \right)\\
&=&\mathcal{C}_{T-1}\left(s_{T-1}, a_{T-1}, \xi_{T-1}\right) +
\gamma\mathcal{C}_{T}\left( g_{T-1}\left(s_{T-1}, a_{T-1}, \xi_{T-1}\right)\right).
\end{eqnarray*}
Under Assumption~\ref{ass-bcr-1} (c),
\begin{eqnarray} 
|
\vt_{T-1}(s_{T-1}, a_{T-1}, \xi_{T-1},\mu_{T-1})| \leq 
(1+\gamma)\kappa_{\cal C}(|\xi_{T-1}|^\upsilon+1). 
\end{eqnarray} 
By setting $\kappa:=(1+\gamma)\kappa_{\cal C}$, 
we know by Proposition~\ref{Prop:BCD}
that $\rho_{\mu_{T-1}}
\circ\rho_{P_{\theta_{T-1}}}\left(
\vt_{T-1}(s_{T-1}, a_{T-1}, \xi_{T-1},\mu_{T-1})
\right)$ is well defined  and 
is continuous
in $a_{T-1}$.
Repeating the analysis above, we can show by induction that
\begin{eqnarray} 
\begin{aligned}
  |
\vt_{t}(s_{t}, a_{t}, \xi_{t},\mu_{t})|&\leq \mathcal{C}_{t}\left(s_{t}, a_{t}, \xi_{t}\right) +\gamma\rho_{\mu_{t+1}}\circ\rho_{P_{\theta_t}} |
\vt_{t+1}(s_{t+1}, a_{t+1}, \xi_{t+1},\mu_{t+1})|\\
&\leq \kappa_\mathcal{C}\frac{1-(\gamma\kappa_\rho)^{T-t+1}}{1-\gamma\kappa_\rho}(|\xi_t|^\upsilon+1)+\gamma\frac{1-(\gamma\kappa_\rho)^{T-t}}{1-\gamma\kappa_\rho}
\end{aligned}
\end{eqnarray}
and hence the well-definedness 
of
$
\rho_{\mu_{t}}
\circ\rho_{P_{\theta_{t}}}\left(\vt_{t}(s_{t}, a_{t}, \xi_{t},\mu_t) \right)$
by virtue of Proposition~\ref{Prop:BCD}  
and  continuity 
in $a_t$ by Lemma \ref{Prop:continuity}
for $t=T-2,\cdots,1$.
\hfill $\Box$

The proposition applies to the case where 
$T<\infty$, that is, 
the BCR-SOC/MDP is a finite-horizon problem. 
In the case of $T=\infty$,
we need to assume 
that $\mathcal{C}_t$  
is bounded and 
$\gamma<1$. The conditions are widely used in the literature of infinite horizon SOC/MDP. 
In the context of Assumption \ref{ass-bcr-1}, 
it corresponds to 
$\upsilon=0$ and $\kappa_\rho=1$. 
Consequently, 
$\vartheta_1(s_1,a_1,\xi_1,\mu_1)$ is bounded by a constant $\frac{2\kappa_\mathcal{C}+\gamma}{1-\gamma}$ for all $(s_1,a_1,\xi_1,\mu_1)\in\mathcal{S}\times\mathcal{A}\times\Xi\times\mathcal{D}_1^0$, which leads to the well-definedness of BCR-SOC/MDP problem \eqref{eq:MDP-BCR} with $T=\infty$.

\subsection{BCR-SOC/MDP vs conventional SOC/MDP models}

The proposed BCR-SOC/MDP model
subsumes several important 
SOC/MDP models in the literature. Here we list some of them.

	\begin{example}[Conventional risk-averse SOC/MDPs]
		\label{bcr-1}
 Conventional risk-averse SOC/MDP models (see e.g.~\cite{ahmadi2021constrained,petrik2012approximate,jiang2018risk}) assume that
 the transition probability matrix $\mathcal{P}$ or the distribution $P$ for $\xi$ is known. 
These models can be regarded as 
a special case of the 
BCR-SOC/MDP model
by setting
$\mu_1(\theta)=\delta_{\theta^c}$,
where $\delta_{\theta^c}$ is the 
Dirac probability measure at $\theta^c$.
Consequently the posterior distribution $\mu_t(\theta)$ remains $\delta_{\theta^c}$ 
for all $t=1,\cdots,T$, because of the Bayes' updating mechanism
$\delta_{\theta^c}\equiv\frac{p\left(\xi |\theta\right)\delta_{\theta^c} }{\int_\Theta p\left(\xi |\theta\right)\delta_{\theta^c}  d \theta}$. 
The resulting 
objective function of problem (\ref{eq:MDP-BCR}) can be reformulated as:
	\begin{equation}
		\min_\pi \rho_{P_{\theta^c}}\left[\mathcal{C}_1\left(s_1, a_1, \xi_1\right)+\cdots+\gamma\rho_{P_{\theta^c}}\left[\mathcal{C}_{T-1}\left(s_{T-1}, a_{T-1}, \xi_{T-1}\right)+\gamma\mathcal{C}_T\left(s_T\right)\right]\right].\label{ramdp}
	\end{equation}
	 Moreover, by adopting the inner risk measure $\rho_{P_{\theta^c}}$ as expectation, VaR and AVaR, respectively, 
  the BCR-SOC/MDP model recovers
  the relevant 
  risk-averse SOC/MDP models in the literature.
	 \end{example}
	
	\begin{example}[Distributionally robust SOC/MDPs]
		\label{bcr-2} In the distributionally robust SOC/MDP models (see e.g.~ \cite{shapiro2022distributionally,wiesemann2013robust}), 
  the DM 
  uses partially available information to construct an ambiguity set of transition probability matrices 
  to mitigate the risk arising from incomplete information about the true transition probability matrix $\mathcal{P}\left(s_{t+1} =s| s_t, a_t\right)$ for $s_t\in\mathcal{S}$ and $a_t\in\mathcal{A}$. 
  The optimal decision is based on the worst expected cost calculated with the ambiguity set of transition matrices. 
The distributionally robust SOC/MDP models  
can be recast as BCR-SOC/MDP models.
\begin{itemize}
    \item The ambiguity set of transition probabilities can equivalently be described by the ambiguity set of 
    the distributions
    of the random parameters within the BCR-SOC/MDP framework.  
    For instance, if we set the inner risk  measure  $\rho_{P_{\theta}}$ 
    as the expectation $\mathbb{E}_{P_{\theta}}$ 
and the outer risk measure 
as essential supremum 
(\text{$\text{VaR}_{\mu_t}^{\alpha}$} with $\alpha=0$), 
then the objective function of problem (\ref{eq:MDP-BCR}) becomes
		\begin{equation*}
			\min_\pi \text{VaR}_{\mu_1}^{\alpha}\circ\mathbb{E}_{P_{\theta_{1}}}\left[\mathcal{C}_1\left(s_1, a_1, \xi_1\right)+\cdots+\gamma\text{VaR}_{\mu_{T-1}}^{\alpha}\circ\mathbb{E}_{P_{\theta_{T-1}}}\left[\mathcal{C}_{T-1}\left(s_{T-1}, a_{T-1}, \xi_{T-1}\right)+\gamma\mathcal{C}_T\left(s_T\right)\right]\right] 
		\end{equation*}
or equivalently		
  \begin{equation}
			\min_\pi \max_{\theta_1\in\Theta_1}\mathbb{E}_{P_{\theta_{1}}}\left[\mathcal{C}_1\left(s_1, a_1, \xi_1\right)+\cdots+\gamma\max_{\theta_{T-1}\in\Theta_{T-1}}\mathbb{E}_{P_{\theta_{T-1}}}\left[\mathcal{C}_{T-1}\left(s_{T-1}, a_{T-1}, \xi_{T-1}\right)+\gamma\mathcal{C}_T\left(s_T\right)\right]\right]\label{drmdp}, 
		\end{equation}
		where $\Theta_{t}$ is the support set of $\mu_t$ for 
        $t=1,\cdots,
        T-1$. As demonstrated in \cite{ma2024multi}, this reformulation is in alignment with         distributionally robust SOC/MDP models
        if we 
  treat
  $\{P_{\theta_t},\theta_t\in\Theta_{t}\}$ as an ambiguity set of 
  $P_{\theta^c}$. If the inner risk measure is set using a general risk measure instead of the expectation,
  the BCR-SOC/MDP model is equivalent to the distributionally robust risk-averse SOC/MDP model proposed in \cite{ruan2023risk}.
  		 Moreover, by choosing \( \alpha \in (0,1) \) for outer risk measure $\text{VaR}_{\mu_t}^{\alpha}$, 
the BCR-SOC/MDP model effectively transforms into a less conservative chance-constrained SOC/MDP form, as proposed in \cite{delage2010percentile}. This BCR-SOC/MDP model can thus be viewed as a distributionally robust SOC/MDP with respect to the Bayesian ambiguity set defined in \cite{gupta2019near}. Further elaboration on this equivalence will be provided in Section 6.1.
		
	\item 	
 In view of the robust representation of coherent risk measures, employing some coherent risk measure as the outer risk measure $\rho_{\mu_t}$ gives our BCR-SOC/MDP framework (\ref{ramdp}) an interpretation aligned with distributionally robust SOC/MDP models concerning the corresponding ambiguity sets. To be consistent with the setting of our article, we refer to \cite{shapiro2022distributionally} for the detailed equivalence between the distributional robust model and the risk-averse model under the premise of parametric family of probability distributions.
 \end{itemize}
	\end{example}
   
 \begin{example}[Bayes-adaptive MDP]
 	\label{bcr-3} Like the concept behind BCR-SOC/MDP, the state space of Bayes-adaptive MDPs consists of the physical state space $\mathcal{S}$ and the posterior distribution space $\mathcal{D}_1^p$, as introduced in \cite{ross2011bayesian}. Specifically, as demonstrated in \cite{chen2024bayesian}, the objective function of the episodic Bayesian SOC in \cite{shapiro2023episodic} can be formulated as the following Bayes-adaptive form:
 	\begin{equation}
 		\min_\pi \mathbb{E}_{\hat{P}_{{\theta}_{1}}}\left[\mathcal{C}\left(s_1, a_1, \xi_1\right)+\cdots+\gamma\mathbb{E}_{\hat{P}_{{\theta}_{T-1}}}\left[\mathcal{C}\left(s_{T-1}, a_{T-1}, \xi_{T-1}\right)+\gamma\mathcal{C}\left(s_T\right)\right]\right]. 
 	\end{equation}
 	Here, $\hat{P}_{{\theta}_{t}}$ satisfying $d\hat{P}_{{\theta}_{t}}:=\int_{\Theta}p(\xi|\theta)\mu_t(\theta)d\theta$,
  represents the posterior estimate of the underlying distribution ${P_{\theta^c}}$.
 	By setting both the inner risk measure $\rho_{P_{\theta_{t}}}$ and the outer risk measure $\rho_{\mu_t}$ as the expectation operators $\mathbb{E}_{P_{\theta_{t}}}$ and $\mathbb{E}_{{\mu_t}}$ respectively, the objective function of the BCR-SOC/MDP model reduces to
 	\begin{equation}
 		\min_\pi \mathbb{E}_{{\mu_1}}\circ\mathbb{E}_{P_{\theta_{1}}}\left[\mathcal{C}\left(s_1, a_1, \xi_1\right)+\cdots+\gamma\mathbb{E}_{{\mu_{T-1}}}\circ\mathbb{E}_{P_{\theta_{T-1}}}\left[\mathcal{C}\left(s_{T-1}, a_{T-1}, \xi_{T-1}\right)+\gamma\mathcal{C}\left(s_T\right)\right] \right],
 	\end{equation}
 	which is equivalent to that of Bayes-adaptive MDP. This can easily be derived by exchanging the order of integrals. Furthermore, the objective function of the BR-MDP in \cite{lin2022bayesian}
\begin{equation}
 		\min_\pi \rho_{{\mu_1}}\circ\mathbb{E}_{P_{\theta_{1}}}\left(\mathcal{C}\left(s_1, a_1, \xi_1\right)+\cdots+\gamma\rho_{{\mu_{T-1}}}\circ\mathbb{E}_{P_{\theta_{T-1}}}\left[\mathcal{C}\left(s_{T-1}, a_{T-1}, \xi_{T-1}\right)+\gamma\mathcal{C}\left(s_T\right)\right] \right)
 	\end{equation} can be derived from the BCR-SOC/MDP model by only setting the inner risk measure $\rho_{P_{\theta_{t}}}$ as the expectation operator $\mathbb{E}_{P_{\theta_{t}}}$.
 \end{example} 
\begin{example}\label{bcr-6}
The  BCR-MDP model \eqref{eq:MDP-BCR}
may be viewed as a generalized 
contextual MDP model, where the contextual information 
$\theta_t$ is learnable through Bayesian updates.
For each contextual variable $\theta_t$, we can find the  transition kernel $\mathcal{P}(s_{t+1}\!\mid s_t,a_t,\theta_t)$ such that
    $
    \mathcal{P}(s_{t+1}\!\mid s_t,a_t,\theta_t)
      \;=\;
      P_{\theta_t}(\xi_t=\xi'),
    $
    where \(\xi'\in\Xi\) satisfies \(g_t(s_t,a_t,\xi')=s'\).
 In the case
 that there is no Bayesian update (i.e., \(\mu_t\equiv\mu_1\) for all \(t\), which means $\theta_t$ is time independent), and the DM is risk neutral in both 
$\xi_t$ and $\theta_t$ (i.e.,  \(\rho_{P_{\theta_t}}=\mathbb E_{P_{\theta_t}}\) and \(\rho_{\mu_t}=\mathbb E_{\mu_t}\)),
the BCR-MDP model \eqref{eq:MDP-BCR} reduces to the contextual MDP by Hallak et al.~\cite{hallak2015contextual}.
\end{example}

\subsection{
  BCR-SOC/MDP with randomized VaR and SRM -- a preference robust viewpoint
  }

BCR can be understood 
as the risk measure of a
random risk measure. 
In decision analytics,
$P_{\theta}$ is DM's belief 
about uncertainty of the state of nature and $\rho_{P_\theta}(X)$
is a measure of the risk of
such uncertainty. The risk measure captures the DM's risk preference 
and the uncertainty of $\theta$
makes $\rho_{P_\theta}(X)$ a random variable.
In this subsection, we will use two examples to compare the effect of empirical distributions and ambiguity sets for these purposes, and highlight the importance and practical relevance of the BCR 
model in real-world problems.

In a data-driven environment, a popular approach 
is to use samples to derive an estimation of $\theta$
and subsequently an approximation of $P_{\theta^c}$. 
However, the approach may not work well 
when the sample size is small.
The next example illustrates this.

\begin{example} 
Consider an example of predicting the price movement of a stock. The true probability distribution of the stock price increasing is $P_\theta$, where $\theta$ signifies the likelihood of an upward movement (stock price increase). When $\theta \in (0, 0.5)$, it means the probability that the stock price will increase is less than 0.5, indicating a generally downward market trend.
Let us consider the case where the stock price movement is neutral, i.e., $\theta^c = 0.5$, meaning the probability of an upward movement is 50\%. Suppose that we observe the stock price movement over five trading days, resulting in 2 days of price increases (up days) and 3 days of price decreases (down days). In this case, we may deduce that $\theta = 0.4$, suggesting a higher probability of a downward movement than the true $\theta^c = 0.5$. This estimate deviates significantly from the true $\theta^c = 0.5$ because the sample size is small.
Indeed, by using the Binomial distribution, we have
\[
\begin{aligned}
    &P_{0.4}(\pmb{\xi}^5=\{\text{2 up days and 3 down days}\}):=P(\pmb{\xi}^5=\{\text{2 up days and 3 down days}\}|\theta=0.4)=0.34\\
    \approx&P_{0.5}(\pmb{\xi}^5=\{\text{2 up days and 3 down days}\}):=P(\pmb{\xi}^5=\{\text{2 up days and 3 down days}\}|\theta=0.5)=0.31,
\end{aligned}\] 
which means that the probability of observing 2 up days and 3 down days with a stock market with a neutral price movement is very close to that of a stock market where $\theta=0.4$.
To address the issue, we can use a convex combination
$0.4\rho_{P_{0.4}}(X)+0.6\rho_{P_{0.5}}(X)$, instead of  $\rho_{P_{0.4}}(X)$, to 
evaluate the risk of loss in the market. This concept is used by Li et al.~\cite[Example 1]{li2022randomization} in the context of randomized 
spectral risk measures.

It might be tempting to adopt a robust approach by considering the worst-case probability distribution of $P_\theta$ from an ambiguity set 
such as $\{P_{\theta}|\theta\in[0.3,0.6]\}$. In that case, we have
\[
\begin{aligned}
    &P_{0.4}(\pmb{\xi}^5=\{\text{2 up days and 3 down days}\}):=P(\pmb{\xi}^5=\{\text{2 up days and 3 down days}\}|\theta=0.4)=0.34\\
   >&P_{0.6}(\pmb{\xi}^5=\{\text{2 up days and 3 down days}\}):=P(\pmb{\xi}^5=\{\text{2 up days and 3 down days}\}|\theta=0.6)=0.23,
\end{aligned}\] 
which means that the worst-case probability distribution is too conservative for use. This is because $\theta$ is assumed to have equal probability
over the range $[0.3,0.6]$ in the robust argument.
The simple calculation shows that
adopting a distributional form for estimation of $\theta$ 
which weights the different effects of the $\theta$ values on the distribution of $\xi$
might be more appropriate in the absence of complete information about $\theta$ based on a small sample size of $\xi$.

We can cope with the above issue from a Bayesian perspective: before observing any data, we assume that no specific information about the true parameter $\theta^c$ is known by the DM, and thus the DM can consider $\theta^c$ to be a realization of a random variable $\theta$ following a uniform distribution over the interval [0, 1], i.e., $\mu_0(\theta)=1$ for all $\theta\in[0,1]$. After observing data with 2 up days and 3 down days, we easily derive using Bayes' formula that the posterior distribution $\mu(\theta|\pmb{\xi}^5=\text{\{2 up days and 3 down days\}})$ follows a Beta(3,4) distribution, as illustrated in Figure \ref{fig:5}.
\begin{figure}[ht] 
	\centering \includegraphics[width=0.6\textwidth]{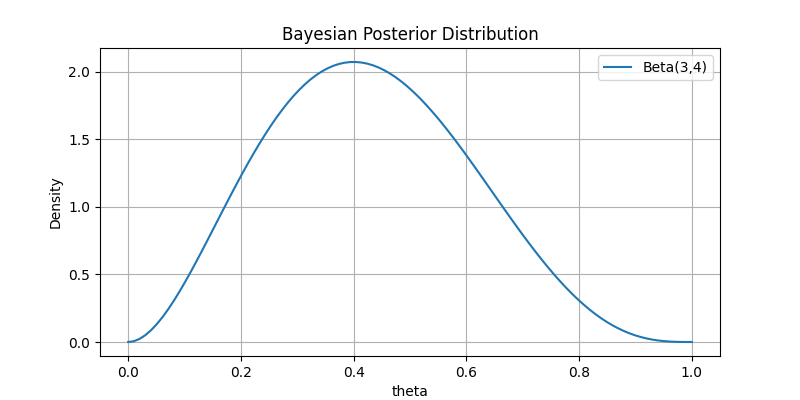} 
	\caption{Posterior distribution 
    of $\theta$ 
    based on observations with 2 up days and 3 down days} 
	\label{fig:5} 
\end{figure}
As depicted, in such cases,  employing a distribution $\mu$ to randomize $\theta$ allows us to capture the variability in environmental performance more effectively. From a modeling perspective, evaluating $\rho_{P_\theta}(X)$ under this framework can quantify the DM's risk preference with respect to the environment through $\rho_\mu\circ\rho_{P_{\theta}}(X)$. Therefore, utilizing $\mu$ as the probability distribution of $\theta$ is more reasonable.

\end{example}

{\color{black}
It is possible to express BCR-SOC/MDP as a preference robust SOC/MDP under some specific circumstances. 
To see this, we consider a BCR where 
$\rho_{P_\theta} = \text{VaR}_{P_\theta}^\alpha$ for some $\alpha\in (0,1)$ and $\rho_\mu = \mathbb{E}_\mu$.  The resulting BCR is
\[ \mathbb{E}_\mu\circ\text{VaR}_{P_\theta}^\alpha(X)=\int_{\Theta}\text{VaR}_{P_\theta}^{\alpha}(X)\mu(\theta)d\theta. \]
We will show shortly that the BCR may be expressed as an SRM.
Assume,
for fixed $\alpha$, that there exists a bijective function 
$f:\Theta\to(0,1)$ such that 
$\text{VaR}_{P_\theta}^{\alpha}(X)=\text{VaR}_{P_{\theta^c}}^{f(\theta)}(X)$,
 where $P_{\theta^c}$ is the true probability distribution of $X$.
For fixed $\mu$ and $f$, let
\begin{equation}
  t=1-f(\theta) \text{ and } \sigma(t)=\mu(f^{-1}(1-t))\cdot\left|\frac{d}{dt}f^{-1}(1-t)\right|.
    \label{eq:delta-spectrum}
\end{equation}
Then
\begin{eqnarray*} \mathbb{E}_\mu\circ\text{VaR}_{P_\theta}^{\alpha}(X)&&=\int_{\Theta}\text{VaR}_{P_\theta}^{\alpha}(X)\mu(\theta)
d\theta = \int_{\Theta}\text{VaR}_{P_{\theta^c}}^{f(\theta)}(X)\mu(\theta)d\theta \\&&= \int_{0}^{1}\text{VaR}_{P_{\theta^c}}^{1-t}(X)\sigma(t)dt=\int_{0}^{1}F^{\leftarrow}_{P_{\theta^c}}(t)\sigma(t)dt.
\end{eqnarray*}
The right-hand side of the equation above is a spectral risk measure, which can be viewed as the weighted average of 
the randomized $F^{\leftarrow}_{P_{\theta^c}}(t)$. 
The next example explains how it may work.

\begin{example}
Consider a random variable 
\(X\) which 
follows an exponential distribution 
parameterized by  \(\theta\).
It is easy to derive that
\begin{eqnarray*} 
\text{VaR}_{P_{\theta}}^{\alpha}(X)=\frac{-\ln \alpha}{\theta}.
\end{eqnarray*} 
Suppose that
the true probability distribution of 
\(X\) is 
\(P_{\theta^c}\).
By definition, the SRM of $X$ is
 \begin{eqnarray*} 
 M_\sigma(X)=\int_0^1\text{VaR}_{P_{\theta}^c}^{1-t}(X)\sigma(t)dt=\int_0^1\frac{-\ln (1-t)}{\theta^c}\sigma(t)dt.
 \end{eqnarray*}
On the other hand, the BCR of $X$ is
\begin{eqnarray*} \mathbb{E}_{\mu}\circ\text{VaR}_{P_\theta}^{\alpha}(X)=\int_\Theta\frac{-\ln\alpha}{\theta}\mu(\theta)d\theta.
\end{eqnarray*}
By setting 
$f(\theta)=\alpha^{\frac{\theta^c}{\theta}}
$ and $
t=1-f(\theta)=1-\alpha^{\frac{\theta^c}{\theta}}$,
we demonstrate that
$
\text{VaR}_{P_\theta}^{\alpha}(X)=\text{VaR}_{P_{\theta^c}}^{1-t}(X).
$
Thus, by letting risk spectrum 
\begin{eqnarray*} 
\sigma(t) :=\mu\left(\frac{\theta^c\ln\alpha}{\ln (1-t)}\right)\frac{\theta^c\ln1/\alpha}{(1-t)(\ln (1-t))^2},
\end{eqnarray*}
we obtain
\(\mathbb{E}_{\mu}\circ\text{VaR}_{P_\theta}^{\alpha}(X)=M_\sigma(X)\).
\end{example}
}

Generalizing from this, we consider a generic law invariant coherent risk measure $\rho$.
We know from \cite{shapiro2021lectures} that $\rho$ can be represented as
\[\rho(X)=\sup_{\zeta\in\mathfrak{D}}\mathbb{E}_\zeta X,
\]
where $\mathfrak{D}$ represents a set of pdfs.

By adopting 
a coherent risk measure
for the outer risk measure $\rho_\mu$, 
we have
\begin{equation}\label{probcr}
	\rho_\mu
 \circ\text{VaR}_{P_\theta}^{\alpha}(X)
 := \sup_{\tilde{\mu}\in\mathfrak{D}_{\rho_\mu}}\int_{\Theta}\text{VaR}_{P_\theta}^{\alpha}(X)\tilde{\mu}(\theta)d\theta=\sup_{\sigma\in\mathfrak{A}_{\rho_\mu}}\int_0^1\text{VaR}_{P_{\theta^c}}^{1-t}(X)\sigma(t)dt\equiv\sup_{\sigma\in\mathfrak{A}_{\rho_\mu}}M_\sigma(X),
\end{equation}
where $\mathfrak{D}_{\rho_\mu}$ is the domain of the Fenchel's conjugate of $\rho_\mu$ (see \cite[Theorem 6.5]{shapiro2021lectures}),
and $\mathfrak{A}_{\rho_\mu}$ is a set of risk spectra 
corresponding to $\mathfrak{D}_{\rho_\mu}$
based on the one-to-one correspondence between $\mu$ and $\sigma$ as specified  
in (\ref{eq:delta-spectrum}).
The second term in equation (\ref{probcr}) can be viewed as 
distributionally robust formulation
of 
$\rho_\mu
\circ\text{VaR}_{P_\theta}^{\alpha}(X)$, while 
the third  term of the equation 
is preference robustness of SRM. The robustness in the former is concerned with the ambiguity of the environmental 
risk represented by $\tilde{\mu}$ whereas the robustness in the latter 
is concerned with 
ambiguity of the DM's preference (represented by $\sigma$) associated with the ambiguity in the environmental 
risk.
Note that despite $\mathfrak{D}_{\rho_\mu}$ is independent of $X$, 
set $\mathfrak{A}_{\rho_\mu}$ depends on $X$ because otherwise 
$\sup_{\sigma\in\mathfrak{A}_{\rho_\mu}}M_\sigma(\cdot)$ would be a coherent risk measure whereas 
$\rho_\mu
 \circ\text{VaR}_{P_\theta}^{\alpha}(\cdot)$ is not.

\begin{example}[Preference robust SOC/MDP]\label{bcr-9}
  Let the inner risk measure $\rho_{P_{\theta}}$ be $\text{VaR}_{P_{\theta}}^\beta$ and the outer risk measure $\rho_{\mu_t}$ be a coherent risk measure. In this case, the objective function of the BCR-SOC/MDP model (\ref{eq:MDP-BCR}) becomes
  	\begin{equation*}
  		\min_\pi \rho_{\mu_1}\circ\text{VaR}_{P_{\theta_1}}^\beta\left(\mathcal{C}\left(s_1, a_1, \xi_1\right)+\cdots+\gamma\rho_{\mu_{T-1}}\circ\text{VaR}_{P_{\theta_{T-1}}}^\beta\left[\mathcal{C}\left(s_{T-1}, a_{T-1}, \xi_{T-1}\right)+\gamma\mathcal{C}\left(s_T\right)\right]\right).
  	\end{equation*}
    By virtue of (\ref{probcr}), we can recast the problem
    as a   {\em preference robust SOC/MDP problem}: 
  	  	\begin{equation*}
  		\min_\pi 	\sup_{\sigma_1\in\mathfrak{A}_{\rho_{\mu_1}}}M_{\sigma_1}\left[\mathcal{C}\left(s_1, a_1, \xi_1\right)+\cdots+\gamma\sup_{\sigma_{T-1}\in\mathfrak{A}_{\rho_{\mu_{T-1}}}}M_{\sigma_{T-1}}\left[\mathcal{C}\left(s_{T-1}, a_{T-1}, \xi_{T-1}\right)+\gamma\mathcal{C}\left(s_T\right)\right]\right], 
  	\end{equation*}
  	where $\mathfrak{A}_{\rho_{\mu_{t}}}$, $1\leq t\leq T-1$, 
    is a set of risk spectra 
    representing 
    the uncertainty of a DM's risk attitude 
    at episode $t$. 
    Based on Bayes' rule, 
    we understand that the updates of $\mu_{t}$ can be established by 
    a continuous process of learning and correction. Therefore, in the corresponding preference robust SOC/MDP model, the DM's ambiguity set of risk attitude $\mathfrak{A}_{\mu_{t}}$ also evolves continually as the sample data accumulates.

A similar correspondence 
between BCR-SOC/MDP and preference robust SOC/MDP
can be established when  
the inner risk measure $\rho_{P_{\theta}}$ is
$\text{AVaR}_{P_{\theta}}^\alpha$ and the outer risk measure $\rho_{ \mu}$ is a coherent risk measure. 
In that case, 
we can derive a representation 
akin to the Kusuoka's representation 
\cite{kusuoka2001law}, 
with 
Kusuoka's ambiguity set 
being 
induced by $\mu$. 
All these will lead to 
another preference robust SOC/MDP model 
with Kusuoka's ambiguity set representing 
DM's preference uncertainty. 
    \end{example}

In
summary, 
the examples discussed above show
that 
the BCR-SOC/MDP model displays 
the breath of its scope and appropriateness 
for adaptive and self-learning systems within the realm of risk control.
Moreover, by judiciously selecting risk measures and configuring posterior distributions, 
we can align BCR-SOC/MDP models with preference robust SOC/MDP models.

\section{Finite-horizon BCR-SOC/MDP}

We now turn to discussing numerical methods for solving
(\ref{eq:MDP-BCR}). As with conventional SOC/MDP models,
we need to separate the discussions depending on whether
$T<\infty$ and $T=\infty$. 
We begin with the former in this section.
The finite horizon BCR-SOC/MDP may be regarded as  an extension of 
the risk-neutral MDP \cite{lin2022bayesian}.
The key step is to derive dynamic recursive 
equations. In the literature of SOC/MDP and multistage stochastic programming, this relies heavily on interchangeability 
and the decomposability of the objective function \cite{shapiro2022distributionally}. However,
the BCR-SOC/MDP does not require these conditions. 
The next proposition
states this.
		
\begin{proposition}
		\label{lem-bcr-1} 
  Let 
  \begin{subequations}
  \begin{eqnarray} 
  V_T^*\left(s_T, \mu_T\right)
  &=&
  \mathcal{C}_T\left(s_T\right), \forall\left(s_T, \mu_T\right) \in \mathcal{S} \times \mathcal{D}_1^p,
  \label{eq:V_T-recursive-T-MDP}\\	
  V_t^*\left(s_t, \mu_t\right)
  &=&
  \inf _{a_t \in \mathcal{A}} \rho_{\mu_t} \circ\rho_{P_{\theta_t}}\left[\mathcal{C}_t\left(s_t, a_t, \xi_t\right)+\gamma V_{t+1}^*\left(s_{t+1}, \mu_{t+1}\right)\right],
  \nonumber\\ 
  &&\qquad\qquad\qquad \qquad\qquad\qquad\forall (s_t, \mu_t)\in\mathcal{S} \times \mathcal{D}_1^p,\ t=1,\cdots,T-1,
  \label{eq:V_t-recursive-t-MDP}
\end{eqnarray} 
\end{subequations}
where $s_t$ and $\mu_t$ satisfy equations \eqref{eq:MDP-BCR-b} and \eqref{eq:MDP-BCR-c}, respectively.    
Suppose that for each stage $t=1,\cdots,T-1$ and every $(s_t,\mu_t)\in\mathcal S\times\mathcal D_1^p$, the infimum
in \eqref{eq:V_t-recursive-t-MDP} is attainable.
   Let  $\mathcal{J}_T^{\gamma,*}(s_1,\mu_1)$  denote the optimal value of  problem (\ref{eq:MDP-BCR}).
    Under Assumption~\ref{ass-bcr-1},
\begin{eqnarray} \label{eq-dp-result}
\mathcal{J}_T^{\gamma,*}(s_1,\mu_1)=V_1^*(s_1,\mu_1).
\end{eqnarray} 
    Moreover, the optimal policy $\pi^* =\{ \pi_1^*,\cdots,\pi_{T-1}^* \}$ corresponding to $\mathcal{J}_T^{\gamma,*}(s_1,\mu_1)$ exists and can be determined recursively by:
	\begin{equation}\label{pi-true}
		\pi_t^*\left(s_t, \mu_t\right)\in\arg\min_{a\in\mathcal{A}}\rho_{\mu_t}\circ\rho_{P_{\theta_t}}\left[\mathcal{C}_t\left(s_t, a_t, \xi_t\right)+\gamma V_{t+1}^*\left(s_{t+1}, \mu_{t+1}\right)\right].
	\end{equation}
\end{proposition}

\begin{proof} 

By the monotonic increasing property
of the BCRs, 
 we can swap minimization operation with the BCRs and subsequently establish
    \begin{equation*}
        \begin{aligned}
            &\min _{\pi_1,\cdots,\pi_{T-1}} \rho_{\mu_1}\circ\rho_{P_{\theta_1}}\left[\mathcal{C}_1\left(s_1, a_1, \xi_1\right)+\cdots+\gamma\rho_{ \mu_{T-1}}\circ \rho_{P_{\theta_{T-1}}}\left[\mathcal{C}_{T-1}\left(s_{T-1}, a_{T-1}, \xi_{T-1}\right)+\gamma\mathcal{C}_T\left(s_T\right)\right]\right]\\
            =&\min_{\pi_1,\cdots,\pi_{T-2}} \rho_{\mu_1}\circ\rho_{P_{\theta_1}}\left[\mathcal{C}_1\left(s_1, a_1, \xi_1\right)+\cdots+\gamma\min_{\pi_{T-1}}\rho_{ \mu_{T-1}}\circ \rho_{P_{\theta_{T-1}}}\left[\mathcal{C}_{T-1}\left(s_{T-1}, a_{T-1}, \xi_{T-1}\right)+\gamma V_T^*\left(s_T, \mu_T\right)\right]\right]
            \\
            =&\min_{\pi_1,\cdots,\pi_{T-2}} \rho_{\mu_1}\circ\rho_{P_{\theta_1}}\left[\mathcal{C}_1\left(s_1, a_1, \xi_1\right)+\cdots+\gamma V_T^*\left(s_{T-1}, \mu_{T-1}\right)\right].
        \end{aligned}
    \end{equation*}
 Repeating the swap above for $t=T-2,\cdots,1$, we obtain 
\eqref{eq-dp-result}.
By 
the definition of 
time consistency of dynamic risk measures
in \cite{ruszczynski2010risk}, 
 the first equality implies that 
 the nested BCRs are dynamically time consistent.
\end{proof}
In alignment with the discussions in \cite{delage2015robust,shapiro2021tutorial}, the BCR-SOC/MDP formulation (\ref{eq:MDP-BCR}) emphasizes the critical aspect of time consistency within the nested risk function framework. 
Thus, for the proposed BCR-SOC/MDP model with a finite horizon, we can recursively solve it using a dynamic programming solution procedure such
as that outlined in 
Algorithm \ref{alg:A}. 
We will discuss the details of solving
the Bellman equations (\ref{eq:V_T-recursive-T-MDP})-(\ref{eq:V_t-recursive-t-MDP}) 
in Section 6.

\begin{algorithm}
	\caption{Dynamic Programming Method for Finite-horizon BCR-SOC/MDP (\ref{eq:MDP-BCR}).
    }
	\begin{algorithmic}[1]\label{alg:A}
		\STATE Initialization: $V_T^*\left(s_T, \mu_T\right)=\mathcal{C}_T\left(s_T\right)$, for all $\left(s_T, \mu_T\right) \in \mathcal{S} \times \mathcal{D}_1^p$
		\FOR{$t = T-1, T-2, \cdots, 1$}
		\FORALL{$\left(s_t, \mu_t\right) \in \mathcal{S} \times \mathcal{D}_1^p$}
		\STATE Solve problem
             \begin{subequations}\label{eq-algorithm}
  \begin{eqnarray}
  \label{eq:BCR-Alg1}
  \min_{a\in\mathcal{A}} && \rho_{\mu_t} \circ\rho_{P_{\theta_t}}\left[\mathcal{C}_t\left(s_t, a_t, \xi_t\right)+\gamma V_{t+1}^*\left(s^\prime, \mu^\prime\right)\right]\\
  \text{s.t.}
  &&s^\prime=g_t(s_t,a_t,\xi_{t}),\ \mu^\prime=\frac{p\left(\xi_t |\theta\right)\mu_t(\theta) }{\int_\Theta p\left(\xi_t |\theta\right)\mu_t(\theta)  d \theta}
  \end{eqnarray}
    \end{subequations}
  to obtain an optimal policy $\pi_t^*\left(s_t, \mu_t\right)$.
  		\STATE Setting $
				V_t^*\left(s_t, \mu_t\right) := \rho_{\mu_t} \circ\rho_{P_{\theta_t}}\left[\mathcal{C}_t\left(s_t, \pi_t^*\left(s_t, \mu_t\right), \xi_t\right)+\gamma V_{t+1}^*\left(s^\prime,\mu^\prime \right)\right]$.
		\ENDFOR
		\ENDFOR
  \STATE Setting $\mathcal{J}_T^{\gamma,*}(s,\mu) :=V_1^*\left(s_t, \mu_t\right)$ for all $(s,\mu)\in\mathcal{S}\times\mathcal{D}_1^p$.
		\RETURN Optimal policy $\pi^* = \{\pi^*_1, \pi^*_1, \cdots, \pi^*_{T-1}\}$ and the optimal value function $\mathcal{J}_T^{\gamma,*}$.
	\end{algorithmic}
\end{algorithm}

A key step in Algorithm~\ref{alg:A} is to solve problem (\ref{eq-algorithm}) for all $\left(s_t, \mu_t\right) \in \mathcal{S} \times \mathcal{D}_1^p$. This is implementable in practice 
only when 
$\mathcal{S} \times \mathcal{D}_1^p$
is a discrete set with finite elements.
Moreover, it is often difficult to obtain 
a closed form of the objective function when $\xi$ is continuously distributed.
All these mean that the algorithm is conceptual 
as it stands. We will come back to address the
issue in Section 6.
Another issue is the attainability property of problem (\ref{eq-algorithm}). We need to ensure that this is a convex program so that $\pi_t^*\left(s_t, \mu_t\right)$ is an optimal solution which can be achieved under the assumption in Proposition \ref{lem-bcr-1}. It is sufficient to ensure that $V_{t+1}^*(g_t(s_t,a_t,\xi_{t}),\mu^\prime)$ is convex in $a_t$.
We address this below. The following assumption is required.

\begin{assumption}
\label{ass-convex}
 For all $\xi_t\in\Xi$, $\mathcal{C}_t(s_t, a_t, \xi_t)$ and $g_t(s_t,a_t,\xi_t)$ are jointly convex in 
 $(s_t,a_t)$ over $\mathcal{S}\times\mathcal{A}$
and they are non-decreasing in $s_t$ for each fixed pair of $(a_t,\xi_t)$.
\end{assumption}
By comparison with the assumptions in the existing literature (e.g., \cite{guigues2023risk,shapiro2023episodic}), Assumption \ref{ass-convex} is mild. In these works and in many practical applications, it is typically assumed that $g_t$ is affine. 
Note that the convexity of $g$ is required
to ensure that problem (\ref{eq:V_t-recursive-t-MDP}) 
is convex which will facilitate numerical solution of the problem.
The next proposition states the convexity of 
 problem (\ref{eq:V_t-recursive-t-MDP}).

\begin{proposition}\label{convex}
Let Assumptions~\ref{ass-bcr-1} 
and \ref{ass-convex} hold.
If 
$\rho_{\mu_t}(\cdot)$ and $\rho_{P_{\theta_t}}(\cdot)$ are convex risk measures for $t=1,\cdots,T$, then the following assertions hold.

\begin{itemize}

\item[(i)] Problem (\ref{eq:V_t-recursive-t-MDP}) is a convex program.
\item[(ii)]
If, in addition, $\mathcal{C}_t(s_t, a_t, \xi_t)$ and $g_t(s_t, a_t, \xi_t)$ are Lipschitz continuous 
in $s_t \in \mathcal{S}$ uniformly 
for all $a_t$ and $\xi_t$
with modulus 
$L_{\mathcal{C}}$ and $L_g$, respectively, 
then $V^*_t(s_t,\mu_t)$ is Lipschtz continuous 
in $s_t$, i.e.,
\begin{eqnarray} 
\label{eq:Lip-V_t-BCR-SOC/MDP-s_t}
|V_{t}^*(s_t^2,\mu_t)-V_{t}^*(s_t^1,\mu_t)|	\leq (L_{\mathcal{C}}+\gamma L_{V_{t+1}}L_{g})\left\|s_t^1-s_t^2\right\|,\ \forall s_t^1, s_t^2 \in {\cal S},
\end{eqnarray} 
for $t=1,\cdots,T-1$,
where $L_{V_t}=\sum_{i=0}^{T-t}(\gamma L_g)^iL_\mathcal{C}$.

\item[(iii)] 
If, further,  $\mathcal{C}_t(s_t, a_t, \xi_t)$ and $g_t(s_t, a_t, \xi_t)$ 
are Lipschitz continuous 
 in $a_t\in\mathcal{A}$ uniformly 
 for all $s_t$ and $\xi_t$, 
 then ${\cal C}_t(s_t,a_t,\xi_t) +\gamma V^*_{t+1}(s_{t+1},\mu_{t+1})$ is Lipschtz continuous 
in $a_t$, i.e.,
\begin{eqnarray} 
\label{eq:Lip-V_t-BCR-SOC/MDP-a_t}
&&|\mathcal{C}_t\left(s_t,a_t^1,\xi_t\right)+\gamma V^*_{t+1}\left(g_t(s_t,a_t^1,\xi_t), \mu^{\prime}\right) -\mathcal{C}_t\left(s_t,a_t^2,\xi_t\right)+\gamma V^*_{t+1}\left(g_t(s_t,a_t^2,\xi_t), \mu^{\prime}\right)|	\nonumber
\\ 
&\leq& (L_{\mathcal{C}}+\gamma L_{V_{t+1}}L_{g})\left\|a_t^1-a_t^2\right\|,\ \forall a_t^1, a_t^2 \in {\cal A},
\end{eqnarray} 
for $t=1,\cdots,T-1$.

\end{itemize}

\end{proposition}

\noindent
\textbf{Proof.}
Part (i). 
Under Assumption \ref{ass-bcr-1} (e), 
it suffices to show that 
$\rho_{\mu_t}\circ\rho_{P_{\theta_t}}\left[\mathcal{C}(s_t, a_t, \xi_t)+\gamma V^*_{t+1}\left(s_{t+1}, \mu_{t+1}\right) \right]$ 
is jointly convex
in $(s_t,a_t)$.
To this end, we need to show that
$V^*_{t+1}\left(s_{t+1}, \mu_{t+1}\right)$ is 
convex in $a_t$. 
Under Assumption~\ref{ass-convex},
this is guaranteed if
$V^*_{t+1}$ is convex and nondecreasing in $s_{t+1}$.
This is because the BCR is convex and monotone since $\rho_{\mu_t}(\cdot)$ and $\rho_{P_{\theta_t}}(\cdot)$ are convex risk measures. 
In what follows, we use backward induction to 
show that 
$V^*_{t}$ is convex and non-decreasing  in $s_t$.
Observe that $V_T^*\left(s_T, 
\mu_T\right)=\mathcal{C}\left(s_T\right), \forall\left(s_T, \mu_T\right) \in \mathcal{S} \times \mathcal{D}_1^p$. Thus, the conclusion holds trivially at episode $T$.
Assuming that the convexity and the non-decreasing property
hold for episodes from $t+1$ to $T$,
we demonstrate 
the same properties are satisfied by $V^*_t$ at episode $t$. 
Since $V^*_{t+1}$ is convex in $s_{t+1}$ and non-decreasing in the argument, then $V_{t+1}^*\left(s_{t+1}, \mu_{t+1}\right)$  is convex in $a_t$ under Assumption~\ref{ass-convex}.
Consequently,
$\rho_{\mu_t}\circ\rho_{P_{\theta_t}}\left[\mathcal{C}_t\left(s_t, a_t, \xi_t\right)+\gamma V_{t+1}^*\left(s_{t+1}, \mu_{t+1}\right)\right]$ is jointly convex in $(a_{t},s_{t})$.
The convexity of $V^*_t$ in 
$s_t$ follows from the fact that minimum of a jointly convex function
w.r.t.~one argument ($a_t$)
makes the resulting optimal value function to be convex in the other argument ($s_t$).
To see the non-decreasing 
property of $V_t^*$ in $s_t$,
we note that 
$\mathcal{C}_t\left(s_t, a_t, \xi_t\right)+\gamma V_{t+1}^*\left(g_t(s_t,a_t,\xi_t), \mu_{t+1}\right)$ is 
nondecreasing in $s_t$ 
(given that $V_{t+1}^*$ is nondecreasing and $g_t$ is non-decreasing in $s_t$)
and 
this property is preserved 
by the BCR operation and the minimization in $a_t$.

Part (ii). We also use 
backward induction
to establish 
Lipschitz continuity of $V^*_t(s_t,\mu_t)$ 
in $s_t$. At episode $T$, the Lipschtz continuity is evident since $V_T^*\left(s_T, \mu_T\right)=\mathcal{C}\left(s_T\right), \forall\left(s_T, \mu_T\right) \in \mathcal{S} \times \mathcal{D}_1^p$. 
Assuming that the Lipschtz continuity holds from episodes $t+1$ to $T$, we demonstrate that
the Lipschitz continuity also holds at episode $t$. 
Under the hypothesis of the 
induction, 
$ V^*_{t+1}$
is Lipschtz continuous 
in $s_{t+1}$ 
with 
modulus $L_{V_{t+1}}$.
For any ${s}^1_t,s^2_t \in \mathcal{S}$, 
$$
\begin{aligned}
	&\left| \rho_{\mu_t}\circ\rho_{P_{\theta_t}}\left[\mathcal{C}_t\left(s_t^1, {a}_t, \xi_t\right)+\gamma V^*_{t+1}\left(g_t(s_t^1,a_t,\xi_t), \mu^{\prime}\right) \right]- \rho_{\mu_t}\circ\rho_{P_{\theta_t}}\left[\mathcal{C}_t\left(s_t^2, {a}_t, \xi_t\right)+\gamma V^*_{t+1}\left(g_t(s_t^2,a_t,\xi_t), \mu^{\prime}\right) \right]\right|\\
	\leq&  \rho_{\mu_t}  \circ\rho_{P_{\theta_t}}\left|\mathcal{C}_t\left(s_t^1, {a}_t, \xi_t\right)+\gamma V^*_{t+1}\left(g_t(s_t^1,a_t,\xi_t), \mu^{\prime}\right) - \mathcal{C}_t\left(s^2_t, {a}_t, \xi_t\right)-\gamma V^*_{t+1}\left(g_t(s_t^2,a_t,\xi_t), \mu^{\prime}\right) \right| \\
	\leq & \rho_{\mu_t} \circ\rho_{P_{\theta_t}}\left(L_{\mathcal{C}}\left\|s_t^1-s_t^2\right\|+\gamma L_{V_{t+1}}\left\|g_t(s_t^1,a_t,\xi_t)-g_t(s_t^2,a_t,\xi_t)\right\|\right) \\
	\leq & (L_{\mathcal{C}}+\gamma L_{V_{t+1}}L_{g})\left\|s_t^1-s_t^2\right\|.
\end{aligned}
$$
Thus
$$
\begin{aligned}
	&V_{t}^*(s_t^1,\mu_t)-V_{t}^*(s_t^2,\mu_t)\\
	\leq& \sup_{a_t\in\mathcal{A}} \left|\rho_{\mu_t} \circ \rho_{P_{\theta_t}}[\mathcal{C}_t\left(s_t^1, a_t, \xi_t\right)+\gamma V^*_{t+1}\left(g_t(s_t^1,{a}_t,\xi_t), \mu^{\prime}\right)] - \rho_{\mu_t} \circ \rho_{P_{\theta_t}}[\mathcal{C}_t\left(s^2_t, {a}_t, \xi_t\right)-\gamma V^*_{t+1}\left(g_t(s_t^2,a_t,\xi_t), \mu^{\prime}\right)]\right|\\
	\leq &(L_{\mathcal{C}}+\gamma L_{V_{t+1}}L_{g})\left\|s_t^1-s_t^2\right\|.
\end{aligned}
$$
A similar conclusion can be derived by 
swapping the positions of $V_{t}^*(s_t^1,\mu_t)$ and $ V_{t}^*(s_t^2,\mu_t)$. Therefore, we have
\[|V_{t}^*(s_t^2,\mu_t)-V_{t}^*(s_t^1,\mu_t)|	\leq (L_{\mathcal{C}}+\gamma L_{V_{t+1}}L_{g})\left\|s_t^1-s_t^2\right\|.\]
This deduction confirms by mathematical induction that the Lipschitz continuity property holds for all episodes.

  Part (iii).  For any ${a}^1_t,a^2_t \in \mathcal{A}$, we have:
$$\begin{aligned}
	&|\mathcal{C}_t\left(s_t,a_t^1,\xi_t\right)+\gamma V^*_{t+1}\left(g_t(s_t,a_t^1,\xi_t), \mu^{\prime}\right) -\mathcal{C}_t\left(s_t,a_t^2,\xi_t\right)+\gamma V^*_{t+1}\left(g_t(s_t,a_t^2,\xi_t), \mu^{\prime}\right)|\\
	\leq& L_{\mathcal{C}}\left\|a_t^1-a_t^2\right\|+\gamma L_{V_{t+1}}||g_t(s_t,a_t^1,\xi_t)-g_t(s_t,a_t^2,\xi_t)||\\
	\leq &(L_{\mathcal{C}}+\gamma L_{V_{t+1}}L_{g})\left\|a_t^1-a_t^2\right\|
\end{aligned}
$$
for $t=1,\cdots,T-1$.
\hfill $\Box$

\section{Infinite-horizon BCR-SOC/MDP}

In this section, we move on to 
discuss the BCR-SOC/MDP problem 
(\ref{eq:MDP-BCR})
with 
infinite-horizon, i.e., $T=\infty$.
We start with the following assumption.

\begin{assumption}\label{ass5.1} Consider problem (\ref{eq:MDP-BCR}).
    (a) 
    The transition dynamics are time-invariant, i.e., $g_t(s,a,\xi) \equiv g(s,a,\xi)$ and  $\mathcal{C}_t(s,a,\xi) \equiv \mathcal{C}(s,a,\xi)$ for all $t$. 
(b) For all $(s, a,\xi) \in \mathcal{S} \times \mathcal{A}\times\Xi$, the 
absolute value of the 
cost function $|\mathcal{C}(s, a, \xi)|$ is bounded by a constant $\bar{\mathcal{C}}$.
(c)     The discount factor $\gamma\in (0,1)$.
\end{assumption}

Under 
Assumption \ref{ass5.1},  
we can formulate the BCR-SOC/MDP model 
(\ref{eq:MDP-BCR})
as follows:
\begin{subequations}
\label{eq:BCR-SOC/MDP-infinite}
\begin{eqnarray} 
\min _\pi && \mathcal{J}_\infty^\gamma(\pi,s_1,\mu_1):=\rho_{\mu_1}\circ\rho_{P_{\theta_{1}}}\left[\mathcal{C}\left(s_1, a_1, \xi_1\right)+\cdots+\gamma\rho_{ \mu_{t}}\circ \rho_{P_{\theta_{t}}}\left[\mathcal{C}\left(s_{t}, a_{t}, \xi_{t}\right)+\cdots\right]\right]
\label{eq:BCR-SOC/MDP-infinite-a} 
\\
		\text {s.t.} && s_{t+1}=g\left(s_t, a_t, \xi_t\right),\ a_t=\pi(s_t,\mu_t),\  t=1,2, \cdots ;
        \label{eq:BCR-SOC/MDP-infinite-b} 
        \\
		&& \mu_{t+1}(\theta)=\frac{p\left(\xi_t |\theta\right)\mu_t(\theta) }{\int_\Theta p\left(\xi_t |\theta\right)\mu_t(\theta)  d \theta},\ t=1,2, \cdots.
        \label{eq:BCR-SOC/MDP-infinite-c} 
    \end{eqnarray} 
    \end{subequations}
We use $\mathcal{J}_\infty^{\gamma,*}$ to denote the optimal value of problem (\ref{eq:BCR-SOC/MDP-infinite}).  
Given the boundness of the cost function $\mathcal{C}(s, a, \xi)$, it follows that the value function $\mathcal{J}_\infty^{\gamma,*}$ is also bounded, i.e., $$||\mathcal{J}_\infty^{\gamma,*}||_\infty:=\sup_{(s,\mu)\in\mathcal{S}\times\mathcal{D}^0_1}|\mathcal{J}_\infty^{\gamma,*}(s,\mu)|\leq\frac{\bar{\mathcal{C}}}{1-\gamma}.$$

\subsection{Bellman equation and optimality}
It is challenging to solve the infinite-horizon BCR-SOC/MDP problem (\ref{eq:BCR-SOC/MDP-infinite}). A standard approach is to use the Bellman equation to develop an iterative scheme
analogous to those used in the existing SOC/MDP 
models (e.g., \cite{guigues2023risk,shapiro2023episodic,xu2010distributionally}).
We begin by defining the 
Bellman equation 
of (\ref{eq:BCR-SOC/MDP-infinite}).
\begin{definition}[Deterministic Bellman operator] \label{def-beo-1}
	 Let $\mathfrak{B}(\mathcal{S},
     \mathcal{D}_1^0)$ 
     denote the space of real-valued bounded measurable functions on $\mathcal{S} \times\mathcal{D}_1^0$. 
     For any value function $V \in \mathfrak{B}(\mathcal{S}, \mathcal{D}_1^0)$,
  define 
  the deterministic operator $\mathcal{T}:  \mathfrak{B}(\mathcal{S}, \mathcal{D}_1^0)\rightarrow \mathfrak{B}(\mathcal{S}, \mathcal{D}_1^0)$ as follows:
	\begin{equation}
	(\mathcal{T} V)(s, \mu)=\min _{a \in \mathcal{A}} \rho_{\mu}\circ\rho_{P_{\theta}}
    \left(\mathcal{C}(s, a, \xi)+\gamma V\left(s^{\prime}, \mu^{\prime}\right)\right)  \text{ for all} \;\; s,\mu
 \label{eq:Operator-T}
	\end{equation}
and operator $\mathcal{T}^\pi: \mathfrak{B}(\mathcal{S}, \mathcal{D}_1^0)\rightarrow \mathfrak{B}(\mathcal{S}, \mathcal{D}_1^0)$ for a given $\pi$:
	\begin{eqnarray} 
     \label{eq:Operator-T^pi}
	\left(\mathcal{T}^\pi V\right)(s, \mu)=\rho_{\mu}
    \circ\rho_{P_{\theta}}\left(\mathcal{C}(s, \pi(s, \mu), \xi)+\gamma V\left(s^{\prime}, \mu^{\prime}\right)\right).
	\end{eqnarray} 
\end{definition}

The dynamic programming formulation integrated into the Bellman operator ensures that risk factors are effectively incorporated into the decision-making process. 
This approach accounts systematically for both epistemic and aleatoric uncertainties, which are crucial when determining the optimal policy at each decision point. In the Bellman equation presented above, the random variable $\xi$ serves as a general representation of the stochastic input in each decision episode. This simplification helps to maintain the recursive structure of the model without the need to explicitly track time indices for the random variables.

Model (\ref{eq:BCR-SOC/MDP-infinite}) and the Bellman equation \eqref{eq:Operator-T} assume deterministic policies, but the 
rationale 
behind this 
set-up has not yet been explained fully. 
The next lemma states that optimizing over randomized policies is essentially equivalent to optimizing over deterministic ones.
The latter will simplify the decision problem and improve computational efficiency. 
We write $\Delta(\mathcal A)$ for the set of all Borel probability measures on the action set $\mathcal A$.
A {randomized policy} is a measurable mapping  
$\pi:\; \mathcal S\times\mathcal D_1^{p}\;\longrightarrow\; \Delta(\mathcal A)$
such that for each 
pair
of state and belief 
$(s,\mu)$, 
$\pi(\cdot\mid s,\mu)$ stipulates 
the probability distribution actions
to be taken over $\mathcal{A}$.  
We write $\Pi^{R}$ for the set of all such randomized policies.

\begin{lemma}
	\label{lem-beo-1}The deterministic Bellman operator is equivalent to the random Bellman operator, i.e.,
	$$
	\begin{aligned}
		(\mathcal{T} V)(s, \mu)&=\min _{a \in \mathcal{A}} \rho_{\mu}\circ\rho_{P_{\theta}}\left(\mathcal{C}(s, a, \xi)+\gamma V\left(s^{\prime}, \mu^{\prime}\right)\right)\\
		& =\inf _{\pi \in \Pi^{R}} \mathbb{E}_{a\sim\pi(\cdot\mid s,\mu)}\left[\rho_{\mu}\circ\rho_{P_{\theta}}\left(\mathcal{C}(s, a, \xi)+\gamma V\left(s^{\prime}, \mu^{\prime}\right)\right)\right].
	\end{aligned}
	$$
\end{lemma}
The proof is somewhat standard in the literature of MDP, see details in the online version of this paper \cite{ma2024bayesian}. This lemma leverages an arbitrary randomized policy $\pi$ to demonstrate that the deterministic Bellman operator indeed corresponds to the random Bellman operator, thereby justifying our focus on deterministic policies. This simplification, however, does not preclude the extension of our analysis to randomized policies.
In what follows,
we demonstrate that
the 
optimal value function $\mathcal{J}_\infty^{\gamma,*}$ of  BCR-SOC/MDP problem  (\ref{eq:BCR-SOC/MDP-infinite})
is the unique fixed point of operator $\mathcal{T}$.
Specifically,
\begin{eqnarray} 
\label{eq:bellman-infinite-MDP}
\mathcal{T}\mathcal{J}_\infty^{\gamma,*}(s,\mu)=\mathcal{J}_\infty^{\gamma,*}(s,\mu),\ \forall(s,\mu)\in {\cal S}\times {\cal D}_1^0.
\end{eqnarray}
To this end, we derive an intermediate result 
to ensure that the Bellman operators
defined in Definition~\ref{def-beo-1} are monotonic and contractive.

\begin{lemma}[Monotonicity and contraction]
    \label{lem-beo-2} 
    Let $\mathcal{T}$ and $\mathcal{T}^\pi$ be defined as in 
    (\ref{eq:Operator-T}) and (\ref{eq:Operator-T^pi}). The following assertions hold.

    \begin{itemize}
        \item[(i)] Both $\mathcal{T}$ and $\mathcal{T}^\pi$  
   are monotonic, 
    i.e.,
    $V \leq V^{\prime}$ implies that $\mathcal{T}V \leq \mathcal{T}V^{\prime}$ and $\mathcal{T}^\pi V \leq \mathcal{T}^\pi V^{\prime}$.

\item[(ii)]
     For 
    any measurable functions $f_i:\Theta\to\mathbb{R}$ and  $g_i:\Xi\to\mathbb{R}$, $i=1,2$, 
    \begin{eqnarray}
    \label{eq:out-rm-No-expan}
    |\rho_{ \mu}(f_1(\theta))-\rho_{ \mu}(f_2(\theta))|\leq\|f_1(\theta)-f_2(\theta)\|_{\theta,\infty}
    \end{eqnarray} 
	and
	\begin{eqnarray} 
      \label{eq:inn-rm-No-expan}
    |\rho_{ P_{\theta}}(g_1(\xi))-\rho_{P_{\theta}}(g_2(\xi))|\leq\|g_1(\xi)-g_2(\xi)\|_{\xi,\infty},
    \end{eqnarray} 
where $\|\cdot\|_{\theta,\infty}$ is the sup-norm under probability measure $\mu$ and $\|\cdot\|_{\xi,\infty}$ is the sup-norm under probability measure $P_{\theta}$.

\item[(iii)] The operators $\mathcal{T}$ and $\mathcal{T}^\pi$ are $\gamma$ contractive with respect to the $\|\cdot\|_{\infty}$ norm. That is, for any bounded value functions $V, V^{\prime} \in \mathfrak{B}(\mathcal{S}, \mathcal{D}_1^0)$, we have
$$
\left\|\mathcal{T} V-\mathcal{T} V^{\prime}\right\|_{\infty} \leq \gamma\left\|V-V^{\prime}\right\|_{\infty}\text{ and }\left\|\mathcal{T}^\pi V-\mathcal{T} ^\pi V^{\prime}\right\|_{\infty} \leq \gamma\left\|V-V^{\prime}\right\|_{\infty}.
$$
    \end{itemize}
 
\end{lemma}

The proof for this lemma is standard in the literature of MDP, see details in the online version of this paper \cite{ma2024bayesian}.
We are now ready to present 
the main result 
of this section which states the existence of a unique solution to Bellman equation
(\ref{eq:bellman-infinite-MDP}) and the solution coinciding with the optimal value of 
the infinite BCR-SOC/MDP (\ref{eq:BCR-SOC/MDP-infinite}).

\begin{theorem}\label{thm-beo-1}
Consider the unique solution for Bellman equation (\ref{eq:bellman-infinite-MDP}). The following assertions hold.

\begin{itemize}
  
\item[(i)] 
Equations (\ref{eq:Operator-T}) and \eqref{eq:Operator-T^pi}
have unique value functions $V^*$ and $V^{\pi}$, respectively.
Moreover, 
for any value function $V\in \mathfrak{B}(\mathcal{S}, \mathcal{D}_1^0)$,
if 
$V\geq\mathcal{T}  V$, 
then $V\geq V^*$, and, if $V\leq\mathcal{T}  V$, 
then $V\leq V^*$.

\item[(ii)] 
The optimal value of problem (\ref{eq:BCR-SOC/MDP-infinite}),
$ \mathcal{J}_\infty^{\gamma,*}$,
is the unique solution to that
of (\ref{eq:Operator-T}),
i.e.,
    \[\mathcal{J}_\infty^{\gamma,*}(s,\mu)\equiv V^*(s,\mu),\ \forall (s,\mu)\in\mathcal{S}\times\mathcal{D}_1^0.\] 
\end{itemize}

\end{theorem}

\noindent \textbf{Proof.}
     Part (i).
     Since $\mathcal{T}$ is 
     contractive,
the existence and uniqueness of the optimal value function $V^*$ 
     follows directly from the Banach fixed-point theorem. We prove the other two arguments below.
Consider any policy $\pi$ such that
	\begin{equation}\label{eq1}
		V \geq \mathcal{T}^\pi V.
	\end{equation}
	Such a policy $\pi$ always exists since we can choose $\pi$ such that $\mathcal{T}^\pi V=\mathcal{T} V$. By applying $\mathcal{T}^\pi$  iteratively to both sides of (\ref{eq1}) and invoking the monotonicity property from Lemma \ref{lem-beo-2}, we deduce that $V \geq\left(\mathcal{T}^\pi\right)^n V$ for all $ n\in\mathbb{N}^+$. Here, $\left(\mathcal{T}^\pi\right)^n V$ represents the accumulated risk value over a finite horizon with the stationary policy $\pi$ and the terminal value function $V$.  Specifically, 
	$$
	\begin{aligned}
		&\left(\mathcal{T}^\pi\right)^n V \\
        & =\rho_{\mu_1}\circ\rho_{P_{\theta_{1}}}\left[\mathcal{C}\left(s_1, a_1, \xi_1\right)+\gamma\rho_{ \mu_{2}} \circ\rho_{P_{\theta_{2}}}\left[\mathcal{C}\left(s_{2}, a_{2}, \xi_{2}\right)+\cdots+\gamma \rho_{\mu_n}\circ\rho_{P_{\theta_{n}}}\left[\mathcal{C}\left(s_n, a_n, \xi_n\right)+\gamma V\left(s_{n+1},\mu_{n+1}\right)\right]\right]\right]\\
		& \geq \rho_{\mu_1}\circ\rho_{P_{\theta_{1}}}\left[\mathcal{C}\left(s_1, a_1, \xi_1\right)+\gamma\rho_{ \mu_{2}} \circ\rho_{P_{\theta_{2}}}\left[\mathcal{C}\left(s_{2}, a_{2}, \xi_{2}\right)+\cdots+\gamma \rho_{\mu_n}\circ\rho_{P_{\theta_{n}}}\left[\mathcal{C}\left(s_n, a_n, \xi_n\right)\right]\right]\right] .
	\end{aligned}
	$$
As $n$ tends to infinity, the sequence at right-hand side converges to the value function under the policy $\pi$. 
Since $V^*$ is defined as the minimum over all such policies, 
then $V \geq V^*$.

    Next,for the other case, consider a finite horizon BCR-SOC/MDP with the risk value at last episode defined as $ V\left(s_{t+1}, \mu_{t+1}\right)$. It follows that: $$\rho_{\mu_t}\circ\rho_{P_{\theta_{t}}}\left[{\cal C}\left(s_t, a_t, \xi_t\right)+\gamma V\left(s_{t+1}, \mu_{t+1}\right)\right] \geq (\mathcal{T} V)\left(s_t, \mu_t\right) \geq V\left(s_t, \mu_t\right).$$ 
	Therefore, we can deduce
	\begin{align*}
		&\rho_{\mu_1}\circ\rho_{P_{\theta_{1}}}\left[\mathcal{C}\left(s_1, a_1, \xi_1\right)+\gamma\rho_{ \mu_{2}} \circ\rho_{P_{\theta_{2}}}\left[\mathcal{C}\left(s_{2}, a_{2}, \xi_{2}\right)+\cdots+\gamma V\left(s_{t+1},\mu_{t+1}\right)\right]\right] \\
		\geq& \rho_{\mu_1}\circ\rho_{P_{\theta_{1}}}\left[\mathcal{C}\left(s_1, a_1, \xi_1\right)+\gamma\rho_{ \mu_{2}} \circ\rho_{P_{\theta_{2}}}\left[\mathcal{C}\left(s_{2}, a_{2}, \xi_{2}\right)+\cdots+\gamma V\left(s_{t},\mu_{t}\right)\right]\right].
	\end{align*}
	By setting $s_1=s$ and $\mu_1=\mu$, we have
	$$
	\rho_{\mu_1}\circ\rho_{P_{\theta_{1}}}\left[\mathcal{C}\left(s_1, a_1, \xi_1\right)+\gamma\rho_{ \mu_{2}} \circ\rho_{P_{\theta_{2}}}\left[\mathcal{C}\left(s_{2}, a_{2}, \xi_{2}\right)+\cdots+\gamma V\left(s_{t+1},\mu_{t+1}\right)\right]\right] \geq V(s, \mu) .
	$$
	Let $V_{\max}$ be an upper bound on $|V(s, \mu)|$, $\forall s \in \mathcal{S}, \mu \in \mathcal{D}_1^0$, we have
	$$
	\rho_{\mu_1}\circ\rho_{P_{\theta_{1}}}\left[\mathcal{C}\left(s_1, a_1, \xi_1\right)+\gamma\rho_{ \mu_{2}} \circ\rho_{P_{\theta_{2}}}\left[\mathcal{C}\left(s_{2}, a_{2}, \xi_{2}\right)+\cdots+\gamma \rho_{\mu_t}\circ\rho_{P_{\theta_{t}}}\mathcal{C}\left(s_t, a_t, \xi_t\right)\right]\right] \geq V(s, \mu)-V_{\max} \gamma^t .
	$$
	As $t \rightarrow \infty$, the left-hand side of the inequality converges to the value function under the policy $\pi$. The infimum over all policy $\pi$ is bounded from below by $V(s, \mu)$, leading to the conclusion that  $V^* \geq V$.

\vspace{0.3cm}
Part (ii).     To begin, with the definition of $\mathcal{J}_T^\gamma(\pi,s,\mu)$, for any horizon $T$, state $s$, $\mu$ and policy $\pi$, we have that 
    \[\mathcal{J}_T^{\gamma,*}(s,\mu)\leq\mathcal{J}_T^\gamma(\pi,s,\mu)\leq\mathcal{J}_\infty^\gamma(\pi,s,\mu). \]
    Therefore, 
    \[\mathcal{J}_T^{\gamma,*}(s,\mu)\leq\mathcal{J}_\infty^{\gamma,*}(s,\mu). \]
    Given the monotone property of Bellman operator $\mathcal{T}$, $\mathcal{J}_T^{\gamma,*}(s,\mu)$ forms a nondecreasing sequence which implies that $\lim_{T\to\infty}\mathcal{J}_T^{\gamma,*}=V^*.$

    Hence, to complete the proof, it only remains to show that $V^*\geq\mathcal{J}_\infty^{\gamma,*}$.
   Consider any policy \( \pi \) such that \( V^* \geq (\mathcal{T}^\pi)V^* \). By iterating this inequality for all \( n \geq 1 \) and  $(s,\mu)\in\mathcal{S}\times\mathcal{D}_1^0$, we obtain:
    \begin{align*}
        V^*(s,\mu)\geq\rho_{\mu_1}\circ\rho_{P_{\theta_{1}}}\left[\mathcal{C}\left(s_1, a_1, \xi_1\right)+\gamma\rho_{ \mu_{2}} \circ\rho_{P_{\theta_{2}}}\left[\mathcal{C}\left(s_{2}, a_{2}, \xi_{2}\right)+\cdots+\gamma \rho_{\mu_n}\circ\rho_{P_{\theta_{n}}}\left[\mathcal{C}\left(s_n, a_n, \xi_n\right)\right]\right]\right].
    \end{align*}
 Letting \( n \to \infty \), it follows that for all \( (s,\mu) \in \mathcal{S} \times \mathcal{D}_1^0 \),
  \[V^*(s,\mu)\geq\mathcal{J}_\infty^{\gamma}(\pi,s,\mu)\geq\mathcal{J}_\infty^{\gamma,*}(s,\mu).\]
 Therefore, combining this result with Part (i) of Theorem \ref{thm-beo-1}, we conclude that \( \mathcal{J}_\infty^{\gamma,*} \equiv V^* \) satisfies the Bellman equation \( \mathcal{J}_\infty^{\gamma,*} = \mathcal{T}\mathcal{J}_\infty^{\gamma,*} \).
\hfill $\Box$

\subsection{Asymptotic convergence}

Let $V^*$ be defined as in Theorem~\ref{thm-beo-1}
and $P_{\theta^c}$ be defined in Section 2.
Define
\begin{equation}
\label{eq:V^*-MDP-infinite}
	V^*(s,\delta_{\theta^c}) =\min _{a \in \mathcal{A}} \rho_{P_{\theta^c}}\left[\mathcal{C}(s, a, \xi)+\gamma V^{*}\left(s^{\prime}, \delta_{\theta^c}\right)\right]
\end{equation} 
and 
\begin{equation}
\label{eq:V^*-MDP-infinite-sln}
	\pi^*(s,\delta_{\theta^c})\in\arg\min _{a \in \mathcal{A}} \rho_{P_{\theta^c}}\left[\mathcal{C}(s, a, \xi)+\gamma V^{*}\left(s^{\prime}, \delta_{\theta^c}\right)\right].
\end{equation}
For $t=1,2,\cdots$, let $\mu_t$ be defined as
in (\ref{eq:BCR-SOC/MDP-infinite-c}). Define
\begin{equation}
\label{eq:V^*-mu-MDP-infinite}
	V^*(s,\mu_t) =\min _{a \in \mathcal{A}} \rho_{ \mu_t}\circ\rho_{P_{\theta}}\left[\mathcal{C}(s, a, \xi)+\gamma V^{*}\left(s^{\prime}, \mu_{t+1}\right)\right]
\end{equation} 
and 
\begin{equation}
\label{eq:V^*-mu-MDP-infinite-sln}
	\pi^*(s,\mu_t)\in\arg\min _{a \in \mathcal{A}} \rho_{\mu_t}\circ\rho_{P_{\theta}}\left[\mathcal{C}(s, a, \xi)+\gamma V^{*}\left(s^{\prime}, \mu_{t+1}\right)\right].
\end{equation}
In this section, we investigate 
the convergence 
of $V^*(s,\mu_t)$ to $V^*(s,\delta_{\theta^c})$
and $\pi^*(s,\mu_t) $ to $\pi^*(s,\delta_{\theta^c})$
as $t\to\infty$.
The rationale behind the convergence analysis
is that in the infinite-horizon Markov decision making process, 
the DM and the environment interact continuously, allowing the DM to acquire gradually an infinite amount of information about the true environment (represented by the true probability distribution of $\xi$).
As information or data accumulates, 
the DM's understanding of the environment improves continuously, 
diminishing the uncertainty 
of $\theta$. 
This raises a question as to whether
the optimal value $V^*(s,\mu_t)$ and the optimal policy $\pi^*(s,\mu_t)$, 
obtained 
from solving the BCR-SOC/MDP model (\ref{eq:BCR-SOC/MDP-infinite}), 
converge to their respective true optimal counterparts, $	V^*(s,\delta_{\theta^c}) $ and $\pi^*(s,\delta_{\theta^c})$
as the data set expands. 
The first step is to derive weak convergence of 
$\mu_t$ to $\delta_{\theta^c}$. To this end, we need the following assumption.

\begin{assumption}
	 \label{ass-ca-1} 
     Let $m_t$ and $v_t$ denote the posterior mean and variance with respect to $\mu_t$. 
     As $t$ tends to infinity, 
$v_t\to0$ and $m_t\to\theta^c$  almost surely (a.s.). 
\end{assumption}

This assumption indicates that as the sequence $\boldsymbol{\xi}^{t}$  expands with the increase of $t$, more observations of $\xi$ are gathered. Naturally, the Bayesian posterior distribution will increasingly concentrate around $\theta^c$ \cite[Bernstein - von Mises theorem]{gelman2013bayesian}, eventually reducing to a Dirac function centered at $\theta^c$. The following example illustrates that this assumption is reasonable.

\begin{example}\label{example-1}
	 Consider the inventory control problem, as proposed in \cite{lin2022bayesian}. Based on experience and relevant expertise, a warehouse manager believes that the customer demand follows a Poisson distribution with parameter $\theta^c$. The prior distribution $\mu_0$ is modeled as a $\operatorname{Gamma}(a, b)$ distribution, which, because of its conjugate nature with respect to the Poisson distribution, ensures that the posterior distribution $\mu_t$ of $\theta$ remains within the Gamma distribution family with updated parameters $\left(a_t, b_t\right)$. 	 
Here, $a_t=a+t \bar{\xi}$, $\bar{\xi}=\frac{\sum_{i=1}^t\xi_i}{t}$ and $b_t=b+t$.
	Therefore, the posterior mean and variance can easily be calculated as
	$$
	m_t=\frac{a+t\bar{\xi}}{b+t} \quad \text { and } \quad v_t=\frac{a+t \bar{\xi}}{(b+t)^2} .
	$$
	By the law of large numbers, $\bar{\xi} \rightarrow \theta^{c}$ a.s. under the true distribution $P_{\theta^{c}}$. From this, it follows immediately that $v_t \rightarrow 0$ and $m_t \rightarrow \theta^{c}$ almost surely, thereby satisfying Assumption \ref{ass-ca-1}. Further, since the convergence rate of $\bar{\xi}$ is $O(t^{-1/2})$, we have \[ \lim_{t\to\infty}{t}\max \{(m_t-\theta^c)^2,v_t\}=\theta^c,\]
which verifies the conditions required in Theorem \ref{thm-ca-2} presented below.
\end{example}

\begin{proposition}
    Under Assumption \ref{ass-ca-1}, for any $\varepsilon>0$,  we have that
	\begin{equation}\label{eq:mu_t-converge}
		\lim_{t\to\infty}\int_{\theta\in \mathbb{B}(\theta^c,\varepsilon)}\mu_t(\theta)d\theta=1,\ \text{a.s.}
	\end{equation}
	where $\mathbb{B}(\theta^c,\varepsilon):=\{\theta:|\theta-\theta^c|\leq\varepsilon\}$ denotes the neighborhood around $\theta^c$.
\end{proposition}
In \cite{shapiro2023episodic}, this result 
is called Bayesian consistency, see Assumption 3.2 there.  
As shown in \cite{hong2017review}, it follows 
by Markov inequality that
\begin{equation}\label{bayescon}
	\int_{\theta\in \mathbb{B}(\theta^c,\varepsilon)}\mu_t(\theta)d\theta\geq1-\frac{v_t+(m_t-\theta^c)^2}{\varepsilon^2}.
\end{equation}
 Therefore,  the Bayesian consistency as stated in (\ref{eq:mu_t-converge}) is achieved under Assumption \ref{ass-ca-1}.

\begin{lemma}
\label{lem-weak-converge}
    Under Assumption \ref{ass-ca-1},
    $\mu_t$ converges weakly to the Dirac measure $\delta_{\theta^c}$, i.e.,
\begin{eqnarray} 
\label{eq:Lemma5.3-pf-0}
\lim_{t\to\infty}\int_\Theta h(\theta)\mu_t(\theta)d\theta=  h({\theta^c})
\end{eqnarray} 
for 
any bounded and continuous function $h$. 
Moreover, $\rho_{\mu_t}(h(\theta))$ converges to $\rho_{\delta_{\theta^c}}\left(h({\theta})\right)=h({\theta^c})$ as $t$ tends to infinity.
\end{lemma}

\noindent \textbf{Proof.}
The result is similar to \cite[Lemma 3.2]{shapiro2023bayesian}, here we give a proof for completeness. For any $\varepsilon > 0$, decompose the integral as follows:
\begin{eqnarray} 
\label{eq:Lemma5.3-pf-1}
\int_{\Theta } h(\theta) \mu_t(\theta)d\theta = \int_{\mathbb{B}(\theta^c, \varepsilon)} h(\theta) \mu_t(\theta)d\theta + \int_{\Theta  \setminus \mathbb{B}(\theta^c, \varepsilon)} h(\theta) \mu_t(\theta)d\theta,
\end{eqnarray} 
Since $h$ is continuous, for any $\epsilon > 0$, there exists an $\varepsilon_0 > 0$ such that for all $\varepsilon \leq \varepsilon_0$ and $\theta \in \mathbb{B}(\theta^c, \varepsilon)$, $
|h(\theta) - h(\theta^c)| < \epsilon$.
Consequently, we have
\[
\left| \int_{\mathbb{B}(\theta^c, \varepsilon)} h(\theta) \mu_t(\theta)d\theta -\int_{\mathbb{B}(\theta^c, \varepsilon)}  h(\theta^c) \mu_t(\theta)d\theta \right| < \epsilon \int_{\mathbb{B}(\theta^c, \varepsilon)} \mu_t(\theta)d\theta< \epsilon.
\]
On the other hand, since
$h$ is bounded, 
then there exists a constant 
$M > 0$ such that 
\begin{eqnarray} 
\label{eq:Lemma5.3-pf-3}
\left| \int_{\Theta  \setminus \mathbb{B}(\theta^c, \varepsilon)} h(\theta) \mu_t(\theta)d\theta- \int_{\Theta  \setminus \mathbb{B}(\theta^c, \varepsilon)} h(\theta^c) \mu_t(\theta)d\theta \right| \leq 2M \int_{\Theta  \setminus \mathbb{B}(\theta^c, \varepsilon)} \mu_t(\theta)d\theta\to 0
\end{eqnarray} 
as $t \to \infty$.
Combining 
(\ref{eq:Lemma5.3-pf-1})-(\ref{eq:Lemma5.3-pf-3}), 
since $\epsilon$ can be arbitrarily small, 
we establish (\ref{eq:Lemma5.3-pf-0}) as desired. 

Moreover, since 
$\mu_t$ converges to $\delta_{\theta^c}$ weakly and by 
Lemma \ref{Prop:continuity}, 
$\rho_{\mu_t}(h(\theta))$ is continuous 
in $Q_\mu$,
we conclude that $\rho_{\mu_t}(h(\theta))$ converges to $h({\theta^c})$ as $t$ goes to infinity.
\hfill $\Box$

\subsubsection{Qualitative convergence}
We now focus on addressing qualitative asymptotic convergence of optimal values and optimal policies defined in (\ref{eq:V^*-mu-MDP-infinite})
and (\ref{eq:V^*-mu-MDP-infinite-sln}). To this end, we need the following generic
technical condition.

\begin{assumption}\label{ass-lip-rho}
Define $Z(\xi)=\mathcal{C}(s, a, \xi) + \gamma V^{*}\left(s^{\prime}, \mu_{t+1}\right)$ and
$
\psi(\theta):=\rho_{P_{\theta}}(Z(\xi))-\rho_{P_{\theta^c}}(Z(\xi)).
$ There exist constants $L_\rho^{\rm in}$ and $k$ such that 
    \begin{eqnarray} 
    \label{eq:inner-rm-Lip}
    |\psi(\theta)|
    \leq L_\rho^{\rm in}|\theta-\theta^c|^k,
    \ \forall \theta\in \mathbb{B}(\theta^c,\varepsilon).
    \end{eqnarray} 
\end{assumption}
The following propositions state how the assumption may be satisfied.

\begin{proposition}
\label{lemma-lip-rho}
    Suppose the total variation (TV) distance \cite{strasser1985mathematical} between 
$P_{\theta}$  and $P_{\theta^c}$ 
  satisfies the following condition:  
	\begin{eqnarray} 
    {\mathsf {d\kern -0.07em l}_{TV}}(P_{\theta},P_{\theta^c})=\frac{1}{2}\int_{\Xi}|p(\xi|\theta)-p(\xi|\theta^c)|d\xi\leq L_p|\theta-\theta^c|    ,\; 
\forall \theta\in \mathbb{B}(\theta^c,\varepsilon),
\end{eqnarray} 
where   $L_p$ is a constant.
Then, there exists a constant $L_\rho^{\rm in}$ such that
    \begin{eqnarray} 
    |\psi(\theta)|\leq L_\rho^{\rm in}|\theta-\theta^c|,\ \forall \theta\in \mathbb{B}(\theta^c,\varepsilon),
    \end{eqnarray} 
    if one of the following conditions hold:
(i)
The inner risk measure 
takes a parametric form as: 
		\begin{equation}\label{3.14}
      	\rho_{P_\theta}(Z(\xi))=\inf_{\phi \in \Phi}\mathbb{E}_{P_\theta}[\Psi(Z(\xi),\phi)],
		\end{equation}
		where $\Phi$ is a non-empty closed convex subset of $\mathbb{R}^m$, and  $\Psi:\mathbb{R}\times\Phi\to\mathbb{R}$ is uniformly bounded and jointly convex with respect to $(Z(\xi),\phi)\in\mathbb{R}\times\Phi$.
(ii) The inner risk measure is a coherent risk measure or a robust SRM as $\rho_{P_\theta}(Z(\xi))=\sup_{\sigma\in \mathfrak{A}} M_\sigma(Z(\xi))$.
\end{proposition}

\noindent \textbf{Proof.}
Part (i). By definition
$$
\begin{aligned}
   |\psi(\theta)|=&\left|\rho_{P_{\theta}}\left[\mathcal{C}(s, a, \xi) + \gamma V^{*}\left(s^{\prime}, \mu_{t+1}\right)\right]-\rho_{P_{\theta^c}}\left[\mathcal{C}(s, a, \xi) + \gamma V^{*}\left(s^{\prime}, \mu_{t+1}\right)\right]\right|\\
=&\left|\inf_{\phi \in \Phi}\mathbb{E}_{P_\theta}[\Psi(\mathcal{C}(s, a, \xi) + \gamma V^{*}\left(s^{\prime}, \mu_{t+1}\right),\phi)]-\inf_{\phi \in \Phi}\mathbb{E}_{P_{\theta^c}}[\Psi(\mathcal{C}(s, a, \xi) + \gamma V^{*}\left(s^{\prime}, \mu_{t+1}\right),\phi)] \right|\\
\leq & \sup_{\phi\in\Phi}\left|\mathbb{E}_{P_\theta}[\Psi(\mathcal{C}(s, a, \xi) + \gamma V^{*}\left(s^{\prime}, \mu_{t+1}\right),\phi)]-\mathbb{E}_{P_{\theta^c}}[\Psi(\mathcal{C}(s, a, \xi) + \gamma V^{*}\left(s^{\prime}, \mu_{t+1}\right),\phi)]\right|\\
\leq&  K \int|p(\xi|\theta)-p(\xi|\theta^c)|d\xi
\leq 2 K L_p|\theta-\theta^c|,
\end{aligned}
$$
where the second last inequality follows from the assumption on $\rho_{P_{\theta}}$,
which guarantees the existence of a constant $K$ such that $\Psi(\mathcal{C}(s, a, \xi) + \gamma V^{*}\left(s^{\prime}, \mu_{t+1}\right),\phi)\leq K$.    By setting $L_\rho^{\rm in}=2KL_p$ and $k=1$, we can draw the conclusion.

Part (ii). Denote $\nu_{P_\theta}=P_\theta\circ Z^{-1}$ and $\nu_{P_{\theta^c}}=P_{\theta^c}\circ Z^{-1}$. Since $Z(\xi)$ is bounded by $\frac{\bar{\mathcal{C}}}{1-\gamma}$, we have  
\[\mathsf {d\kern -0.07em l}_{W}^p(\nu_{P_\theta},\nu_{P_{\theta^c}})\leq \frac{\bar{\mathcal{C}}}{1-\gamma}[\mathsf {d\kern -0.07em l}_{TV}(\nu_{P_\theta},\nu_{P_{\theta^c}})]^{\frac{1}{p}}\leq \frac{\bar{\mathcal{C}}}{1-\gamma} [{\mathsf {d\kern -0.07em l}_{TV}}(P_{\theta},P_{\theta^c})]^{\frac{1}{p}}\leq L_pM|\theta-\theta^c|^{\frac{1}{p}}.\]
In the case that the inner risk measure is coherent, by setting $p=1$, we have from \eqref{eq:Lip-RM-rho} that
\[|\rho_{P_\theta}(Z(\xi))-\rho_{P_{\theta^c}}(Z(\xi))|\leq L\mathsf {d\kern -0.07em l}_{K}(\nu_{P_\theta},\nu_{P_{\theta^c}})\leq \frac{L_p\bar{\mathcal{C}}}{1-\gamma}|\theta-\theta^c|.\]
Further, by setting $L_\rho^{\rm in}=\frac{L_p\bar{\mathcal{C}}}{1-\gamma}\sup_{\sigma\in \mathfrak{A}}\left[\int_0^1\sigma(\tau)^qd\tau\right]^{\frac{1}{q}}$ and $k=1/p$, we can draw the conclusion of robust SRM using part (ii) of Lemma \ref{Prop:continuity}.
\hfill $\Box$

 The  assumption about TV distance in
 Proposition \ref{lemma-lip-rho} is satisfied by a number of important parametric probability distributions
    including exponential, 
    gamma and 
    mixed normal distributions under some moderate conditions.
The condition of part (i) in Proposition
\ref{lemma-lip-rho} 
is commonly used in the literature, 
such as in ~\cite{guigues2023risk,shapiro2021lectures,jiang2018risk}. 
It encompasses a variety of risk measures including
AVaR, $\phi$-divergence risk measure, $L_1$-norm risk measure and $L_\infty$-norm risk measure.

{\color{black}
\begin{proposition}\label{prop-k=2}
Suppose that  $Y(\theta):=\rho_{P_\theta}(Z(\xi))$ is twice {continuously} differentiable  
          on $\mathbb B(\theta^c,\varepsilon)$,  $ \nabla_\theta Y(\theta^c)=0$ and there exists an 
          $M$ such that for all $\theta\in\mathbb B(\theta^c,\varepsilon)$,
          \[
              \sup_{\theta\in\mathbb B(\theta^c,\varepsilon)}
              \bigl\|\nabla_{\theta}^2Y(\theta)\bigr\|
              \;\le\;\;M<\infty.
          \]
Then there exists a constant $L_\rho^{\rm in}$ such that
    \begin{eqnarray} 
    |\psi(\theta)|\leq L_\rho^{\rm in}|\theta-\theta^c|^2,\ \forall \theta\in \mathbb{B}(\theta^c,\varepsilon).
    \end{eqnarray} 
\end{proposition}

\begin{proof}
Let $\Delta=\theta-\theta^c$.  The integral form of Taylor’s 
expansion gives rise to
\[
  Y(\theta)=Y(\theta^c)+
             {\nabla Y(\theta^c)}^\top\Delta
               +\int_0^1(1-t)\,
                 \Delta^\top
                 \nabla_{\theta}^2Y(\theta^c+t\Delta)
                 \,\Delta\,dt.
\]
Using the uniform bound $\|\nabla_{\theta}^2Y\|\le M$, we have
\[
  \bigl|Y(\theta)-Y(\theta^c)\bigr|
  \;\le\;
  \|\Delta\|^2\int_0^1(1-t)M\,dt
  =\frac{M}{2}\,\|\theta-\theta^c\|^2 .
\]
Replacing $Y(\theta)$ and $Y(\theta^c)$ 
with 
$\rho_{P_\theta}(Z)$ and $\rho_{P_{\theta^c}}(Z)$ yields the
desired inequality with $L_\rho^{\rm in}=M/2$ and $k=2$.
\end{proof}

Note that if $p(\xi|\theta)$ is twice {continuously} differentiable  
          on $\mathbb B(\theta^c,\varepsilon)$ and there exists a non-negative integrable  function 
          $r(\xi)$ such that for all $\theta$ in this ball, we have
          \[
                \bigl\|\nabla_\theta\log p(\xi|\theta)\bigr\|,\;
                \bigl\|\nabla^2_{\theta}\log p(\xi|\theta)\bigr\|
                \;\le\; r(\xi),
                \quad
                \sup_{\theta\in\mathbb B(\theta^c,\varepsilon)}
                \mathbb{E}_{P_\theta}\bigl[r(\xi)^2\bigr]<\infty,
          \]
then the conditions in Proposition \ref{prop-k=2} 
are satisfied by various of risk measures 
such as expectation and entropic risk.
}

	\begin{theorem}
    \label{thm-canew-1}
	Let $V^*(s,\mu_t)$ and $V^*(s,\delta_{\theta^c})$ be defined as in 
    (\ref{eq:V^*-mu-MDP-infinite}) and
    (\ref{eq:V^*-MDP-infinite}), let
    $\pi^*(s,\mu_t) $ and  $\pi^*(s,\delta_{\theta^c})$
    be the associated optimal policies defined as in 
    (\ref{eq:V^*-mu-MDP-infinite-sln}) and
    (\ref{eq:V^*-MDP-infinite-sln}).
    Under Assumptions \ref{ass-bcr-1}, \ref{ass5.1}, \ref{ass-ca-1} and \ref{ass-lip-rho}, the following assertions 
   hold.
   \begin{itemize}

\item[(i)]    $V^{*}(s,\mu_t)$ converges 
to $V^{*}(s,\delta_{\theta^c})$ 
uniformly w.r.t.~$s\in {\cal S}$
 as $t\to\infty$, i.e., 
 for any $\varepsilon>0$, there exists a $t^*(\varepsilon)$ such that for all $t\geq t^*(\varepsilon)$,
		\begin{equation}
			\label{3.2}
			\sup_{s\in\mathcal{S}}\left|V^{*}(s,\mu_t)-V^{*}(s,\delta_{\theta^c})\right|\leq\frac{\varepsilon}{1-\gamma}, \text { a.s.}
		\end{equation}

\item[(ii)]
For each fixed $\bar{s}\in {\cal S}$,
$\pi^{*}(\bar{s},\mu_t)$ converges to the set $\mathcal{A}^{*}(\bar{s},\delta_{\theta^c})$ 
a.s.~as $t \rightarrow \infty$. 
    Moreover, 
    if $\mathcal{A}^{*}(\bar{s},\delta_{\theta^c})=\left\{\pi^{*}(\bar{s},\delta_{\theta^c})\right\}$ is a singleton, then $\pi^{*}(\bar{s},\mu_t)$ converges to $\pi^{*}(\bar{s},\delta_{\theta^c})$ almost surely.               
   \end{itemize}
	\end{theorem}

    \noindent \textbf{Proof.}
           Part (i).  It suffices to show the 
            uniform convergence of the right-hand side of (\ref{eq:V^*-mu-MDP-infinite}). 
	  By definition, 
		\begin{equation}
			\begin{aligned}\label{4.18}
					&\sup_{s\in\mathcal{S}}\left|V^{*}(s,\mu_t)-V^{*}(s,\delta_{\theta^c})\right|\\
				=&\sup_{s\in\mathcal{S}}\left|\min _{a \in \mathcal{A}} \rho_{ \mu_t}\circ\rho_{P_{\theta}}\left[\mathcal{C}(s, a, \xi)+\gamma V^{*}\left(s^{\prime}, \mu_{t+1}\right)\right] -\min _{a \in \mathcal{A}} \rho_{P_{\theta^c}}\left[\mathcal{C}(s, a, \xi)+\gamma V^{*}\left(s^{\prime}, \delta_{\theta^c}\right)\right]\right| \\
				 \leq& \sup_{(s,a)\in\mathcal{S}\times\mathcal{A}}\left|\rho_{ \mu_t}\circ\rho_{P_{\theta}}\left[\mathcal{C}(s, a, \xi)+\gamma V^{*}\left(s^{\prime}, \mu_{t+1}\right)\right] -\rho_{\mu_t}\circ\rho_{P_{\theta^c}}\left[\mathcal{C}(s, a, \xi)+\gamma V^{*}\left(s^{\prime}, \delta_{\theta^c}\right)\right]\right|\\
				 \leq& \sup_{(s,a)\in\mathcal{S}\times\mathcal{A}}\left|\rho_{ \mu_t}\circ\rho_{P_{\theta}}\left[\mathcal{C}(s, a, \xi)+\gamma V^{*}\left(s^{\prime}, \mu_{t+1}\right)\right] -\rho_{\mu_t}\circ\rho_{P_{\theta^c}}\left[\mathcal{C}(s, a, \xi)+\gamma V^{*}\left(s^{\prime}, \mu_{t+1}\right)\right] \right|\\
				 +& \sup_{(s,a)\in\mathcal{S}\times\mathcal{A}}\left|\rho_{P_{\theta^c}}\left[\mathcal{C}(s, a, \xi)+\gamma V^{*}\left(s^{\prime}, \mu_{t+1}\right)\right] -\rho_{P_{\theta^c}}\left[\mathcal{C}(s, a, \xi)+\gamma V^{*}\left(s^{\prime}, \delta_{\theta^c}\right)\right]\right|\\
				 \leq&| \rho_{\mu_t}[h(\theta)]|+ \gamma\rho_{P_{\theta^c}}\left[\sup_{s\in\mathcal{S}}\left| V^{*}\left(s, \mu_{t+1}\right)- V^{*}\left(s, \delta_{\theta^c}\right)\right|\right]\\
				  \leq& \varepsilon+ \gamma\varepsilon+ \gamma^2\rho_{P_{\theta^c}}\circ\rho_{P_{\theta^c}}\left[\sup_{s\in\mathcal{S}}\left| V^{*}\left(s, \mu_{t+2}\right)- V^{*}\left(s, \delta_{\theta^c}\right)\right|\right]\\
				 \leq& \sum_{i=1}^n\gamma^{i-1}\varepsilon^i+ \gamma^{n+1}\rho_{P_{\theta^c}}\cdots\rho_{P_{\theta^c}}\left[\sup_{s\in\mathcal{S}}\left| V^{*}\left(s, \mu_{t+n+1}\right)- V^{*}\left(s, \delta_{\theta^c}\right)\right|\right]
     		\leq\frac{\varepsilon}{1-\gamma}.
			\end{aligned}
		\end{equation}
The fourth inequality follows by defining $h(\theta):= L_\rho^{\rm in}|\theta-\theta^c|^k$ and applying Lemma \ref{lem-weak-converge} directly since $h(\theta^c)\equiv0$.
		Letting $n\to\infty$, we obtain that $V^{*}(s,\mu_t)$ converges uniformly to $V^{*}(s,\delta_{\theta^c})$.

Part (ii).  The result here is similar to that of \cite[Proposition~3.2]{shapiro2023episodic}, we provide a proof 
    for completeness due to the adoption of BCR
    in our model.
    By Part (i), 
    \begin{eqnarray*}
    \sup_{(s,a)\in\mathcal{S}\times\mathcal{A}}\left|\rho_{ \mu_t}\circ\rho_{P_{\theta}}\left[\mathcal{C}(s, a, \xi)+\gamma V^{*}\left(s^{\prime}, \mu_{t+1}\right)\right] -\rho_{P_{\theta^c}}\left[\mathcal{C}(s, a, \xi)+\gamma V^{*}\left(s^{\prime}, \delta_{\theta^c}\right)\right]\right|\to0,\ a.s.
    \end{eqnarray*}
Thus, for any $\varepsilon>0$, there exists a $t^*(\varepsilon)<\infty$ such that for all $t \geq t^*(\varepsilon)$,
	$$
	\begin{array}{r}
		\rho_{P_{\theta^c}}\left[\mathcal{C}(\bar{s},\pi^{*}(\bar{s},\mu_t) , \xi)+\gamma V^{*}\left(\bar{s}^\prime, \delta_{\theta^c}\right)\right]-\varepsilon / 2 \leq \rho_{\mu_t}\circ\rho_{P_{\theta}}\left[\mathcal{C}(\bar{s},\pi^{*}(\bar{s},\mu_t) , \xi)+\gamma V^{*}\left(\bar{s}^\prime, \mu_t\right)\right], \\
	\rho_{\mu_t}\circ\rho_{P_{\theta}}\left[\mathcal{C}(\bar{s}, \pi^{*}(\bar{s},\delta_{\theta^c}), \xi)+\gamma V^{*}\left(\bar{s}^\prime, \mu_t\right)\right]-\varepsilon / 2\leq \rho_{P_{\theta^c}}\left[\mathcal{C}(\bar{s}, \pi^{*}(\bar{s},\delta_{\theta^c}), \xi)+\gamma V^{*}\left(\bar{s}^\prime, \delta_{\theta^c}\right)\right] .
	\end{array}
	$$
By the definition of $\pi^{*}(\bar{s},\mu_t)$, 
	$$
\rho_{\mu_t}\circ\rho_{P_{\theta}}\left[\mathcal{C}(\bar{s},\pi^{*}(\bar{s},\mu_t) , \xi)+\gamma V^{*}\left(\bar{s}^\prime, \mu_t\right)\right] \leq 	\rho_{\mu_t}\circ\rho_{P_{\theta}}\left[\mathcal{C}(\bar{s}, \pi^{*}(\bar{s},\delta_{\theta^c}), \xi)+\gamma V^{*}\left(\bar{s}^\prime, \mu_t\right)\right] .
	$$
	Thus, we can derive
	$$
	\rho_{P_{\theta^c}}\left[\mathcal{C}(\bar{s},\pi^{*}(\bar{s},\mu_t) , \xi)+\gamma V^{*}\left(\bar{s}^\prime, \delta_{\theta^c}\right)\right] \leq \rho_{P_{\theta^c}}\left[\mathcal{C}(\bar{s}, \pi^{*}(\bar{s},\delta_{\theta^c}), \xi)+\gamma V^{*}\left(\bar{s}^\prime, \delta_{\theta^c}\right)\right]+\varepsilon.
	$$
	Define the lower level set:
	$$
	L_{\varepsilon}(\bar{s}):=\{a \in \mathcal{A}: 	\rho_{P_{\theta^c}}\left[\mathcal{C}(\bar{s},a , \xi)+\gamma V^{*}\left(\bar{s}^\prime, \delta_{\theta^c}\right)\right] \leq \rho_{P_{\theta^c}}\left[\mathcal{C}(\bar{s}, \pi^{*}(\bar{s},\delta_{\theta^c}), \xi)+\gamma V^{*}\left(\bar{s}^\prime, \delta_{\theta^c}\right)\right]+\varepsilon\}.
	$$
 Then 
 $\pi^{*}(\bar{s},\mu_t) \in L_{\varepsilon}(\bar{s})$.  
 By the convexity of the objective function 
 and the compactness and convexity of the 
 action set, we 
 deduce that for any $\bar{s} \in \mathcal{S}$, 
 the set of the 
 optimal actions,
 $\mathcal{A}^{*}(\bar{s},\delta_{\theta^c})$,
 is nonempty and bounded. 
 Consider a decreasing sequence $\varepsilon_k \to 0$, it is evident that $L_{\varepsilon_{k+1}}(\bar{s}) \subseteq L_{\varepsilon_k}(\bar{s})$. Thus the intersection $\cap_{k=1}^{\infty} L_{\varepsilon_k}(\bar{s})$ is contained in the topological closure of $\mathcal{A}^{*}(\bar{s},\delta_{\theta^c})$. 
 
 Assume for the sake of contradiction that 
 the convergence 
 does not hold. Then 
 we can find a bounded sequence $\{a_k\} \in L_{\varepsilon_k}(\bar{s})$ such that $\text{dist}\left(a_k, \mathcal{A}^{*}(\bar{s},\delta_{\theta^c})\right) \geq \delta>0$ for any $k$. 
 Let  $a^{*}$ be an 
 accumulation point of the sequence.
 Then 
 $$
 \rho_{P_{\theta^c}}\left[\mathcal{C}(\bar{s},a^{*}, \xi)+\gamma V^{*}\left(\bar{s}^\prime, \delta_{\theta^c}\right)\right]\leq  \rho_{P_{\theta^c}}\left[\mathcal{C}(\bar{s}, \pi^{*}(\bar{s},\delta_{\theta^c}), \xi)+\gamma V^{*}\left(\bar{s}^\prime, \delta_{\theta^c}\right)\right]+\varepsilon$$ for any $\varepsilon>0$, and hence $$\rho_{P_{\theta^c}}\left[\mathcal{C}(\bar{s},a^{*}, \xi)+\gamma V^{*}\left(\bar{s}^\prime, \delta_{\theta^c}\right)\right]= \rho_{P_{\theta^c}}\left[\mathcal{C}(\bar{s}, \pi^{*}(\bar{s},\delta_{\theta^c}), \xi)+\gamma V^{*}\left(\bar{s}^\prime, \delta_{\theta^c}\right)\right],
 $$ 
 i.e., $a^{*} \in \mathcal{A}^{*}(\bar{s},\delta_{\theta^c})$. However, this contradicts 
 the assumption that 
 $\operatorname{dist}\left(a^{*}, \mathcal{A}^{*}(\bar{s},\delta_{\theta^c})\right) \geq \delta$.	
 Consequently, the distance from 
 $\pi^{*}(\bar{s},\mu_t)$ 
 to the topological closure of 
 $\mathcal{A}^{*}(\bar{s},\delta_{\theta^c})$
 converges to zero almost surely, 
 and therefore the distance from $\pi^{*}(\bar{s},\mu_t)$ to $\mathcal{A}^{*}(\bar{s},\delta_{\theta^c})$ converges to zero as well. 
 The conclusion 
 is straightforward when $\mathcal{A}^{*}(\bar{s},\delta_{\theta^c})$ is a singleton.
\hfill $\Box$

The theorem shows uniform convergence
of $V^*(s,\mu_t)$ to $V^*(s,\delta_{\theta^c})$ and 
point-wise convergence of the optimal policies, which are similar to the convergence results in \cite[Propositions 3.1-3.2]{shapiro2023episodic}. 
We include a proof as 
the conventional SOC/MDPs and the episodic Bayesian SOC model \cite{shapiro2023episodic} focus on a risk-neutral objective function under fixed environment parameters, whereas the BCR-SOC/MDP model (\ref{eq:BCR-SOC/MDP-infinite}) 
introduces additional complexity by incorporating epistemic uncertainty about parameter estimation and its impact on both the policy and value function. 
Moreover, 
since the BCR does not usually preserve linearity, 
we need to make sure the specific 
details of the convergence of  
the Bellman optimal value function and policy
work through.
Note also that
the asymptotic convergence
is short of quantifying 
$t^*(\varepsilon)$ for the prescribed 
precision $\varepsilon$.
In the next section, 
we will present some 
quantitative convergence results when the BCR takes a specific form.

\subsubsection{Quantitative convergence} 

From Theorem~\ref{thm-canew-1}, we can see that
the convergence of the optimal values and optimal policies
depends heavily on the convergence
of $\mu_t$ to $\delta_{\theta^c}$.
Under some special circumstances, convergence of the latter
may be quantified by the mean and variance of $\mu_t$. 
This raises a question as to whether 
the convergence of the optimal values and optimal policies
can be quantified in terms of the convergence of 
the mean and variance of $\mu_t$. In this subsection, we address this. 

\begin{theorem}\label{thm-ca-2}
    Let Assumptions \ref{ass-ca-1} and \ref{ass-lip-rho} hold. If 
    there exist positive constants $M_1$ and $M_2$ such that 
    \begin{equation} \label{quant-1}
 \lim_{t\to\infty}{t^{M_1}}\max \{(m_t-\theta^c)^2,v_t\}={M_2},
    \end{equation}
    then the following assertions hold.
    \begin{itemize}
\item [(i)]  If 
the outer risk measure is chosen 
with $\mathbb{E}_{\mu_t}$, 
then 
\begin{eqnarray} 
\label{upbound-E-rho}
\sup_{s\in\mathcal{S}}\left|V^{*}(s,\mu_t)-V^{*}(s,\delta_{\theta^c})\right|\leq L_\rho^{\rm in}\left(v_t+\left(m_t-\theta^c\right)^2\right)^{\frac{k}{2}},\ a.s.
\end{eqnarray} 
for all $k\leq2.$
\item[(ii)]    If 
the outer risk measure is chosen 
with $\text{VaR}_{\mu_{t}}^\alpha$, 
then 
\begin{eqnarray}  \label{upbound-VaR-rho}
\sup_{s\in\mathcal{S}}\left|V^{*}(s,\mu_t)-V^{*}(s,\delta_{\theta^c})\right|\leq\frac{L_\rho^{\rm in} }{1-\gamma}\left({\frac{v_t+(m_t-\theta^c)^2}{\alpha}}\right)^{\frac{k}{2}},\ a.s.
\end{eqnarray} 
Furthermore, if the outer risk measure is chosen with robust SRM  $\sup_{\sigma\in \mathfrak{A}} M_\sigma$, then 
\begin{eqnarray}  \label{upbound-SRM-rho}
\sup_{s\in\mathcal{S}}\left|V^{*}(s,\mu_t)-V^{*}(s,\delta_{\theta^c})\right|\leq\frac{L_\rho^{\rm in} \left(v_t+(m_t-\theta^c)^2\right)^{\frac{k}{2}}}{1-\gamma }\sup_{\sigma\in \mathfrak{A}}\int_0^1(1-\tau)^{-\frac{k}{2}}\sigma(\tau)d\tau,\ a.s.
\end{eqnarray} 
\item[(iii)] If
    the outer risk measure is
    chosen with 
    $\text{AVaR}_{\mu_{t}}^\alpha$, then 
	\begin{eqnarray} 
\label{upbound-CVaR-rho}    \sup_{s\in\mathcal{S}}\left|V^{*}(s,\mu_t)-V^{*}(s,\delta_{\theta^c})\right|
    \leq \frac{L_\rho^{\rm in} t^{M_1k/4}}{(1-\gamma)\alpha}\left({\frac{v_t+(m_t-\theta^c)^2}{\alpha}}\right)^{\frac{k}{2}}+\frac{\bar{\mathcal{C}}}{(1-\gamma)^2t^{M_1/2}},\ a.s.
    \end{eqnarray}    
    \end{itemize}
\end{theorem}

Before providing a proof, it might be helpful to comment on the conditions and results of the theorem.
Unlike Theorem \ref{thm-canew-1}, 
Theorem \ref{thm-ca-2} 
explicitly quantifies 
the convergence by providing 
a worst-case upper bound
for the optimal value functions 
uniformly 
with respect to 
risk measures.
This is achieved by 
strengthening 
Assumption \ref{ass-ca-1}
with 
an additional condition \eqref{quant-1} 
about the rate of 
convergence of $v_t$ and $(m_t-\theta^c)^2$,
which 
means that as 
episode $t$ increases, both the posterior variance and bias decrease at a rate of $t^{-M_1}$. 
As 
shown in \cite{gelman2013bayesian}, under standard statistical assumptions, both $v_t$ and $(m_t-\theta^c)^2$ typically converge at a rate of $O(\frac{1}{t})$, corresponding to $M_1=1$,
see 
Example \ref{example-1}.
With the convergence rate of $v_t$ and $(m_t-\theta^c)^2$, 
we can establish the rate of  convergence of the optimal value function 
from Theorem \ref{thm-ca-2}
under some specific BCRs.
\begin{itemize}
    \item If 
the outer risk measure is chosen 
with $\mathbb{E}_{\mu_t}$, $\text{VaR}_{\mu_{t}}^\alpha$, 
and robust SRM, the convergence rate of $V^{*}(s,\mu_t)$ to $V^{*}(s,\delta_{\theta^c})$ is $O(t^{-\frac{k}{2M_1}})$.
   \item If 
the outer risk measure is chosen 
with $\text{AVaR}_{\mu_{t}}^\alpha$, the convergence rate of $V^{*}(s,\mu_t)$ to $V^{*}(s,\delta_{\theta^c})$ is $O(t^{-\frac{k}{4M_1}})$.
\end{itemize}
Shapiro et al.~\cite{shapiro2023episodic} 
demonstrate that the episodic Bayesian SOC model converges at a rate of $O(\frac{1}{\sqrt{t}})$. Under some additional mild conditions, we have demonstrated that the convergence rate can be $O(\frac{1}{t})$ ($M_1=1$ and $k=2$).
Moreover, Part (i) 
of the theorem provides 
a supplement to Part (iii). 
Specifically, when $k \leq 2$, we have the inequality $\text{AVaR}_{\mu_t}^\alpha(\psi(\theta))\leq\frac{1}{\alpha}\mathbb{E}_{\mu_t}(\psi(\theta))$, demonstrating that the convergence rate for $\text{AVaR}_{\mu_t}^\alpha$ can be as fast as that of $\mathbb{E}_{\mu_t}$.

\noindent \textbf{Proof.}
{\color{black}Part (i). By 
the definitions of $v_t$ and $m_t$ in Assumption~\ref{ass-ca-1}, we have
\begin{eqnarray} 
\mathbb{E}_{\mu_t}\left[\left(\theta-\theta^c\right)^2\right]=v_t+\left(m_t-\theta^c\right)^2.
\label{eq:thm5.3-pf-1-var}
\end{eqnarray} 
Since $\rho_{\mu_t}=\mathbb{E}_{\mu_t}$, then by 
(\ref{eq:inner-rm-Lip}), we have
\begin{eqnarray}
\sup_{s\in\mathcal{S}}\left|V^*(s, \mu_t)-V^*\left(s, \delta_{\theta^c}\right)\right| 
\leq \sup_{(s,a)\in\mathcal{S}\times\mathcal{A}}
\mathbb{E}_{\mu_t}\left[|\psi_{s, a}(\theta)|\right] 
\leq L_\rho^{\rm in}\mathbb{E}_{\mu_t}[|\theta-\theta^c|^k].
\label{eq:thm5.3-pf-2}
\end{eqnarray} 
By H\"older inequality,
\begin{eqnarray}
\mathbb{E}_{\mu_t}[|\theta-\theta^c|^k]\leq
\left(\mathbb{E}_{\mu_t}\left[\left(\theta-\theta^c\right)^2\right]\right)^{\frac{k}{2}}\leq \left(v_t+\left(m_t-\theta^c\right)^2\right)^{\frac{k}{2}}.
\label{eq:thm5.3-pf-3}
\end{eqnarray} 
A combination of (\ref{eq:thm5.3-pf-1-var})-
(\ref{eq:thm5.3-pf-3}) yields \eqref{upbound-E-rho}.

\vspace{0.2cm}

Part (ii) Under Assumption \ref{ass-lip-rho}, it follows from (\ref{bayescon}) that
	\begin{equation*}
		Q_{\mu_t}\left(|\psi(\theta)|\leq  L_\rho^{\rm in} \varepsilon\right)
        \geq Q_{\mu_t}(|\theta-\theta^c|^k\leq\varepsilon)\geq1-\frac{v_t+(m_t-\theta^c)^2}{\varepsilon^{2/k}}.
	\end{equation*}
For the specified $\alpha$ in the definition
of $\text{VaR}^\alpha_{\mu_t}$, let 
$\varepsilon :=\left({\frac{v_t+(m_t-\theta^c)^2}{\alpha}}\right)^{\frac{k}{2}}$.
Then the inequality above can equivalently be 
written as
	\begin{equation}\label{eqvar}
		Q_{\mu_t}\left(|\psi(\theta)|\leq L_\rho^{\rm in} \left({\frac{v_t+(m_t-\theta^c)^2}{\alpha}}\right)^{\frac{k}{2}}\right)\geq1-\alpha.
	\end{equation}
Let $q(\theta) = |\psi(\theta)|
$ and $\nu_q=Q_{\mu_t}\circ q^{-1}$. Then \eqref{eqvar} implies
	\begin{equation}\label{eqvar-2}
		{\nu_q}\left(q(\theta)\leq L_\rho^{\rm in} \left({\frac{v_t+(m_t-\theta^c)^2}{\alpha}}\right)^{\frac{k}{2}}\right)\geq1-\alpha.
	\end{equation}
    By the definition of VaR, this implies
$$
\text{VaR}_{\nu_q}^\alpha(q(\theta))\leq
L_\rho^{\rm in} \left({\frac{v_t+(m_t-\theta^c)^2}{\alpha}}\right)^{\frac{k}{2}}.
$$
Thus, we obtain that
    \begin{align*}
		&
\left|\text{VaR}_{\mu_t}^\alpha\circ\rho_{P_{\theta}}\left[\mathcal{C}(s, a, \xi)+\gamma V^{*}\left(s^{\prime}, \mu_{t+1}\right)\right]-\rho_{P_{\theta^c}}\left[\mathcal{C}(s, a, \xi)+\gamma V^{*}\left(s^{\prime}, \mu_{t+1}\right)\right]\right|\\
\leq& \text{VaR}_{\mu_t}^\alpha(|\psi(\theta)|)=\text{VaR}_{\nu_q}^\alpha(q(\theta))\leq
L_\rho^{\rm in} \left({\frac{v_t+(m_t-\theta^c)^2}{\alpha}}\right)^{\frac{k}{2}}.
    \end{align*}           	
    Similar to the proof of (\ref{4.18}), we can conclude that
\begin{equation*}
			\begin{aligned}
					&\sup_{s\in\mathcal{S}}\left|V^{*}(s,\mu_t)-V^{*}(s,\delta_{\theta^c})\right|\\
				 \leq& \sup_{(s,a)\in\mathcal{S}\times\mathcal{A}}\left|\rho_{ \mu_t}\circ\rho_{P_{\theta}}\left[\mathcal{C}(s, a, \xi)+\gamma V^{*}\left(s^{\prime}, \mu_{t+1}\right)\right] -\rho_{\delta_{\theta^c}}\circ\rho_{P_{\theta}}\left[\mathcal{C}(s, a, \xi)+\gamma V^{*}\left(s^{\prime}, \mu_{t+1}\right)\right] \right|\\
				 +& \sup_{(s,a)\in\mathcal{S}\times\mathcal{A}}\left|\rho_{\delta_{\theta^c}}\circ\rho_{P_{\theta}}\left[\mathcal{C}(s, a, \xi)+\gamma V^{*}\left(s^{\prime}, \mu_{t+1}\right)\right] -\rho_{\delta_{\theta^c}}\circ\rho_{P_{\theta}}\left[\mathcal{C}(s, a, \xi)+\gamma V^{*}\left(s^{\prime}, \delta_{\theta^c}\right)\right]\right|\\
				 \leq&L_\rho^{\rm in} \left({\frac{v_t+(m_t-\theta^c)^2}{\alpha}}\right)^{\frac{k}{2}}+ \gamma\rho_{P_{\theta^c}}\left[\sup_{s\in\mathcal{S}}\left| V^{*}\left(s, \mu_{t+1}\right)- V^{*}\left(s, \delta_{\theta^c}\right)\right|\right]\\
     		\leq&\frac{L_\rho^{\rm in} }{1-\gamma}\left({\frac{v_t+(m_t-\theta^c)^2}{\alpha}}\right)^{\frac{k}{2}}.
			\end{aligned}
		\end{equation*}
Thus the proof for $\text{VaR}_{\mu_{t}}^\alpha$ is completed. Similar result can be deduced by the definition of $\sup_{\sigma\in \mathfrak{A}} M_\sigma$.

    Part (iii). Similar to the proof 
    of Part (ii),
    we have 
	\begin{equation*}
		Q_{\mu_t}\left(|\psi(\theta)|\leq {L_\rho ^{\rm in}}\varepsilon\right)\geq Q_{\mu_t}(|\theta-\theta^c|^k\leq\varepsilon)\geq1-\frac{v_t+(m_t-\theta^c)^2}{\varepsilon^{2/k}}.
	\end{equation*}
By setting $\varepsilon=t^{M_1k/4}\left({\frac{v_t+(m_t-\theta^c)^2}{\alpha}}\right)^{\frac{k}{2}}$, we can 
obtain
	\begin{equation}\label{eqcvar}
	Q_{\mu_t}\left(|\psi(\theta)|\leq {L_\rho^{\rm in} }{t^{M_1k/4}}\left({\frac{v_t+(m_t-\theta^c)^2}{\alpha}}\right)^{\frac{k}{2}}\right)\geq1-\frac{\alpha}{t^{M_1/2}}\geq1-\alpha.
\end{equation}
Let 
$$
\Theta_{\varepsilon}=\left\{\theta:|\psi(\theta)|\leq  {L_\rho^{\rm in} }{t^{M_1k/4}}\left({\frac{v_t+(m_t-\theta^c)^2}{\alpha}}\right)^{\frac{k}{2}}\right\}, \ \bar{\Theta}_\varepsilon=\Theta-\Theta_\varepsilon,
$$ 
and 
$$
\Theta_\alpha=\left\{\theta:\rho_{P_{\theta}}\left[\mathcal{C}(s, a, \xi)+\gamma V^{*}\left(s^{\prime}, \mu_{t+1}\right)\right]\geq\text{VaR}_{\mu_t}^\alpha\circ\rho_{P_{\theta}}\left[\mathcal{C}(s, a, \xi)+\gamma V^{*}\left(s^{\prime}, \mu_{t+1}\right)\right]\right\}.
$$ 
Then, for all $s\in\mathcal{S}$, we have
	\begin{align*}
		&\left|\text{AVaR}_{\mu_t}^\alpha\circ\rho_{P_{\theta}}\left[\mathcal{C}(s, a, \xi)+\gamma V^{*}\left(s^{\prime}, \mu_{t+1}\right)\right]-\rho_{P_{\theta^c}}\left[\mathcal{C}(s, a, \xi)+\gamma V^{*}\left(s^{\prime}, \mu_{t+1}\right)\right]\right|\\
		=&\frac{1}{\alpha}\left|\int_{\Theta_\alpha}\left(\rho_{P_{\theta}}\left[\mathcal{C}(s, a, \xi)+\gamma V^{*}\left(s^{\prime}, \mu_{t+1}\right)\right]-\rho_{P_{\theta^c}}\left[\mathcal{C}(s, a, \xi)+\gamma V^{*}\left(s^{\prime}, \mu_{t+1}\right)\right]\right)\mu_t(\theta)d\theta\right|\\
		\leq&\frac{1}{\alpha}\int_{\Theta_\alpha\cap\Theta_\varepsilon}\left|\rho_{P_{\theta}}\left[\mathcal{C}(s, a, \xi)+\gamma V^{*}\left(s^{\prime}, \mu_{t+1}\right)\right]-\rho_{P_{\theta^c}}\left[\mathcal{C}(s, a, \xi)+\gamma V^{*}\left(s^{\prime}, \mu_{t+1}\right)\right]\right|\mu_t(\theta)d\theta\\
		+&\frac{1}{\alpha}\int_{\Theta_\alpha\cap\bar{\Theta}_\varepsilon}\left|\rho_{P_{\theta}}\left[\mathcal{C}(s, a, \xi)+\gamma V^{*}\left(s^{\prime}, \mu_{t+1}\right)\right]-\rho_{P_{\theta^c}}\left[\mathcal{C}(s, a, \xi)+\gamma V^{*}\left(s^{\prime}, \mu_{t+1}\right)\right]\right|\mu_t(\theta)d\theta\\
  		\leq&\frac{1}{\alpha}\sup_{\theta\in\Theta_\varepsilon}\left|\rho_{P_{\theta}}\left[\mathcal{C}(s, a, \xi)+\gamma V^{*}\left(s^{\prime}, \mu_{t+1}\right)\right]-\rho_{P_{\theta^c}}\left[\mathcal{C}(s, a, \xi)+\gamma V^{*}\left(s^{\prime}, \mu_{t+1}\right)\right]\right|\\
		+&\frac{1}{\alpha}\int_{\bar{\Theta}_\varepsilon}\left|\rho_{P_{\theta}}\left[\mathcal{C}(s, a, \xi)+\gamma V^{*}\left(s^{\prime}, \mu_{t+1}\right)\right]-\rho_{P_{\theta^c}}\left[\mathcal{C}(s, a, \xi)+\gamma V^{*}\left(s^{\prime}, \mu_{t+1}\right)\right]\right|\mu_t(\theta)d\theta\\
  		\leq&\frac{1}{\alpha}\sup_{\theta\in\Theta_\varepsilon}\left|\psi(\theta)\right|+\frac{\bar{\mathcal{C}}}{(1-\gamma)\alpha}\int_{\bar{\Theta}_\varepsilon}\mu_t(\theta)d\theta\\
		\leq&\frac{{L_\rho^{\rm in} }t^{M_1k/4}}{\alpha}\left({\frac{v_t+(m_t-\theta^c)^2}{\alpha}}\right)^{\frac{k}{2}}+\frac{\bar{\mathcal{C}}}{(1-\gamma)t^{M_1/2}}.
	\end{align*}    
	The second last inequality results from the boundedness of the cost function. In combination with (\ref{4.18}), we can conclude that
	\[\sup_{s\in\mathcal{S}}\left|V^{*}(s,\mu_t)-V^{*}(s,\delta_{\theta^c})\right|\leq\frac{L_\rho^{\rm in} t^{M_1k/4}}{(1-\gamma)\alpha}\left({\frac{v_t+(m_t-\theta^c)^2}{\alpha}}\right)^{\frac{k}{2}}+\frac{\bar{\mathcal{C}}}{(1-\gamma)^2t^{M_1/2}},\]
    which completes the proof for $\text{AVaR}_{\mu_{t}}^\alpha$.
\hfill $\Box$

\subsection{Value iterations algorithm for BCR-SOC/MDP}

Having examined the theoretical properties of the proposed infinite-horizon BCR-SOC/MDP problem, we now consider its solution. This subsection introduces an algorithm for computing the unique value function $V^*$ within the infinite-horizon BCR-SOC/MDP framework. 
The algorithm is inspired by the dynamic programming equation $(\mathcal{T}  V^{*})(s, \mu)= V^*(s, \mu)$ embodying a risk-averse adaptation of the usual value iteration method applied in standard SOC/MDPs (see \cite{bellman2015applied,howard1960dynamic} for details). 

In the BCR-SOC/MDP model, the introduction of risk-averse characteristics requires adjustments to the conventional value iteration algorithm to accommodate the new dynamic programming equations. The core of these adjustments lies in integrating the impact of risk measures on decision-making during the value function update process. Building on the framework of the conventional value iteration algorithm, we propose a value iteration algorithm tailored for BCR-SOC/MDP, as detailed in Algorithm~\ref{alg:B}, which aims to approximate $V^*$ with $\hat{V}^*$ and meanwhile identify the corresponding $\varepsilon$-optimal policy $\pi^{\varepsilon,*}$.

\begin{algorithm}
	\caption{Value Iteration Algorithm for Infinite-horizon BCR-SOC/MDP (\ref{eq:BCR-SOC/MDP-infinite}).}	
    \label{alg:B}  
	\begin{algorithmic}[1]
		\STATE Initialization: the precision parameter $\varepsilon$, the iteration number $I$
        and the        
        value functions 
        $V_0=0$ and 
        $V_1={\cal T}V_0$, 
set $i :=0$;
		\WHILE{$\left\|V_{i+1}-V_{i}\right\|_\infty\geq\varepsilon(1-\gamma)/2\gamma$
        \textbf{ or }
        $i\leq I$}
        
		\FORALL{$\left(s, \mu\right) \in \mathcal{S} \times \mathcal{D}_1^0$}
		\STATE 
        \bgeqn 
         \label{eq:BCR-Alg2}
        V_{i+1}(s,\mu)\gets \min _{a \in \mathcal{A}} \rho_{\mu}\circ\rho_{P_{\theta}}\left[\mathcal{C}(s, a, \xi)+\gamma V_i\left(s^{\prime}, \mu^{\prime}\right)\right],\ s'=g\left(s, a, \xi\right),\  \mu'(\theta)=\frac{p\left(\xi |\theta\right)\mu(\theta) }{\int_\Theta p\left(\xi |\theta\right)\mu(\theta)  d \theta}.\quad
        \edeqn 
		\ENDFOR
		\STATE Set $i :=i+1$.
		\ENDWHILE
		\FORALL{$\left(s, \mu\right) \in \mathcal{S} \times \mathcal{D}_1^0$}
		\STATE 
        Set   $\hat{V}^*(s,\mu) :=V_{i}(s,\mu)$ 
        and determine $\pi^{\varepsilon,*}(s,\mu)$ by
        \begin{eqnarray} \label{eq-algorithm2}
        \pi^{\varepsilon,*}(s,\mu)=\arg\min _{a \in \mathcal{A}} \rho_{\mu}\circ\rho_{P_{\theta}}\left[\mathcal{C}(s, a, \xi)+\gamma \hat{V}^*\left(s^{\prime}, \mu^{\prime}\right)\right].
        \end{eqnarray} 
		\ENDFOR
		\RETURN The optimal value function $\hat{V}^*$ and optimal policy $\pi^{\varepsilon,*}$.
	\end{algorithmic}
\end{algorithm}

This kind of algorithm is well known in the literature of infinite horizon MDP, see \cite{puterman2014markov,ruszczynski2010risk}. 
There are some specific issues in the algorithm 
which need to be addressed.

\begin{itemize}
\item[(i)] Well-definedness of problems \eqref{eq:BCR-Alg2} and
\eqref{eq-algorithm2}. We need to ensure that the minimum in the two problems are attainable and that the second problem has a unique optimal solution.

\item[(ii)] Convergence of sequence $\{V^i\}$. The convergence of such algorithms is 
well documented and will provide a theoretical
guarantee for the convergence of the algorithm. Specifically,
we will need to explain the 
stopping criterion 
for obtaining 
an approximate optimal value function and 
a corresponding $\varepsilon$-optimal policy of infinite-horizon BCR-SOC/MDPs.

\item[(iii)] Discreteness of the set $\mathcal{S}\times\mathcal{D}^0_1$.
Steps 2-7 are designed to 
obtain an approximate optimal value function of problem (\ref{eq:BCR-SOC/MDP-infinite}). As in Algorithm~\ref{alg:A}, 
this requires the set $\mathcal{S}\times\mathcal{D}^0_1$ to be a finite discrete set. 
The main challenge (computational complexity) lies
in Steps 3-5.
In the case that the physical state space ${\cal S}$ and/or
the belief space $\mathcal{D}_1^0$ is infinite,
the algorithm cannot complete the steps in a finite number of iterations. On the other hand, 
in many practical applications, the belief space $\mathcal{D}_1^0$ is often infinite dimensional 
(the posterior distributions $\mu_t$, $t=1,\cdots$,
are continuously distributed). 
We address the issue in 
the Section 6. 
\end{itemize}
In what follows, we address 
the first two issues. We begin with convergence of $\{V^i\}$. The next theorem states the uniform convergence of the value functions.

\begin{theorem}\label{thm-viqi-1}
	Let the sequence 
    $\{V_{i}\}$ be generated by Algorithm \ref{alg:B}. Then 
    \begin{itemize}
    
\item[(i)]     $\{V_{i}\}$  converges uniformly
to the unique value function $V^*$ at a rate of $\gamma$. Further, the initialization $V_0=0$ ensures a non-decreasing sequence $\{V_{i}\}$.
Conversely, if Algorithm \ref{alg:B} is initialized with  $V_0\geq\frac{\bar{\mathcal{C}}}{1-\gamma}$, the resulting sequence $\{V_i\}$ is non-increasing.

\item[(ii)] The policy $\pi^{\varepsilon,*}$  derived in Algorithm  \ref{alg:B} is $\varepsilon$-optimal, i.e., the corresponding value function ${V}^{\pi^{\varepsilon,*}}$ satisfies
\begin{eqnarray}\label{eq-vipi-result}
\left\|
{V}^{\pi^{\varepsilon,*}}
-V^*\right\|_\infty\leq\varepsilon.
\end{eqnarray}

\end{itemize}
\end{theorem}

The proof is standard in the literature of MDP, see details in the online version of this paper \cite{ma2024bayesian}.

Next, we discuss the well-definedness of problems \eqref{eq:BCR-Alg2} and
\eqref{eq-algorithm2}. To this end, we need to make the following assumption.

\begin{assumption}\label{ass:cont}
Suppose (a)
$\mathcal C(s,a,\xi)$ is jointly continuous and convex in $(s,a)$
for every $\xi\in\Xi$ and strongly convex in~$a$
with a constant $\lambda>0$.
(b)
$g(s,a,\xi)$ is jointly continuous in $(s,a)$
for every $\xi$ and affine in $a$.
(c)
Both inner and outer risk measures
are coherent risk measures.
\end{assumption}

\begin{proposition}\label{thm:V-cont}
Under Assumptions~\ref{ass5.1} and~\ref{ass:cont}, both the  approximation \(\hat V^*\) produced by Algorithm~\ref{alg:B} and the unique fixed point \(V^*\) are continuous in \((s,\mu)\).  Moreover, the policy \(\pi^{\varepsilon,*}(s,\mu)\) defined in \eqref{eq-algorithm2} is unique and continuous.
\end{proposition}

\begin{proof}
Let
$$
\Psi_V(s,\mu,a)
  :=\rho_\mu\!\circ\rho_{P_\theta}\!
     \bigl[\mathcal C(s,a,\xi)+\gamma V\bigl(g(s,a,\xi),\mu'\bigr)\bigr].
$$
We proceed the proof by induction.  
For $i=0$, 
the initial value function \(V_0\equiv 0\) which  is trivially continuous in \((s,\mu)\) and convex in \(s\).  
Thus, \(\Psi_{V_0}\) is continuous in \((s,\mu)\) and strongly convex in \(a\) given that $\mathcal{C}(s,a,\xi)$ is strongly convex in $a$. 
Since the set of 
feasible actions
\(\mathcal A\) is compact, then by Berge’s maximum theorem \cite[Berge’s maximum theorem 17.31]{aliprantis2006infinite}, problem
\[
  (\mathcal T V_0)(s,\mu)
  = \min_{a\in\mathcal A}\Psi_{V_0}(s,\mu,a)
\]
has a unique attainable minimizer and 
the optimal value function is continuous in 
\((s,\mu)\).
Moreover, since $\mathcal C(s,a,\xi)$ is jointly convex in $(s,a)$, then $(\mathcal T V_0)(s,\mu)$ is convex in $s$.
Assume now the conclusion holds for $i=i_0$. That is, 
$V_{i_0}(s,\mu)$ is continuous in $(s,\mu)$ and strongly convex in $s$. Then
$$
V_{i_0+1}(s,\mu) = \min_{a\in\mathcal A}\Psi_{V_{i_0}}(s,\mu,a)=\rho_\mu\!\circ\rho_{P_\theta}\!
     \bigl[ \mathcal{C}(s,a,\xi)+\gamma V_{i_0}\bigl(g(s,a,\xi),\mu'\bigr)\bigr].
$$
Since $g$ is linear in $a$, then $V_{i_0}\bigl(g(s,a,\xi),\mu'\bigr)$ is also convex in $a$ and $C(s,a,\xi)+\gamma V_{i_0}\bigl(g(s,a,\xi),\mu'\bigr)$ is strongly convex in $a$. Repeating the argument in $i=0$ case, we conclude that $V_{i_0+1}(s,\mu)$ is continuous in $(s,\mu)$ and convex in $s$.
On the other hand, it follows from the algorithm that
\(
\hat V^*  := 
V_I\) and 
that by Theorem~\ref{thm-viqi-1},
$V_I$
converges uniformly to $V^*$
as $I\to \infty$
since $\|V_I-V^*\|_\infty\leq \frac{\gamma^I}{1-\gamma}\|V_1-V_0\|_\infty$.
Thus $V^*$ is continuous in $(s,\mu)$.

Finally, since \(\Psi_{\hat V^*}(s,\mu,\cdot)\) is \(\lambda\)-strongly convex according to Assumption \ref{ass:cont}, the \(\varepsilon\)-optimal action \(\pi^{\varepsilon,*}(s,\mu)\) is unique.  
By \cite[Berge’s maximum theorem 17.31]{aliprantis2006infinite}, the
\(\varepsilon\)-optimal action \(\pi^{\varepsilon,*}(s,\mu)\) is continuous in $(s,\mu)$.
\end{proof}

\begin{proposition}
\label{prop-lip-infinite-horizon}
Assume: (a) Assumptions~\ref{ass5.1} and~\ref{ass:cont} hold, and (b)
$\mathcal{C}(s,a,\xi)$ and $g(s,a,\xi)$ are Lipschitz continuous in $s\in\mathcal S$
uniformly for all $a,\xi$ with moduli $L_{\mathcal C}$ and $L_g$, respectively.
Then the folloowing assertions hold.
\begin{itemize}
    
\item[(i)] For any $\varepsilon>0$, let
$
I\ \ge\ \Big\lceil \frac{\ln\!\big(\tfrac{\bar{\cal C}}{(1-\gamma)\varepsilon}\big)}{\ln(1/\gamma)} \Big\rceil$, 
then
$\hat V^*:=V_I$ is an approximation of $V^*$ such that
$\|\hat V^*-V^*\|_\infty\le \varepsilon$ 
and
$\hat V^*(s,\mu)$ is $L_{\hat V^*}$-Lipschitz in $s$ with
\begin{equation}\label{eq:Lip-V_infinite}
L_{\hat V^*}=\frac{L_{\mathcal C}\,\bigl(1-(\gamma L_g)^{I}\bigr)}{1-\gamma L_g}.
\end{equation}

\item[(ii)] 
If, in addition,  $\gamma L_g<1$, then the optimal value function  $V^*$ is $L_{V^*}$-Lipschitz in $s$ with
$L_{V^*}\le \frac{L_{\mathcal C}}{1-\gamma L_g}$.

\end{itemize}

\end{proposition}

\begin{proof}
Part (i). Consider the value iteration sequence $V_{i+1}:=\mathcal T V_i$ with $V_0\equiv 0$. 
From Theorem~\ref{thm-viqi-1},
we know that 
\begin{equation}\label{eq:geom-error}
\|V_I-V^*\|_\infty \;\le\; \frac{\gamma^I}{1-\gamma}\,\|V_1-V_0\|_\infty .
\end{equation}
Since the cost function $\cal C$ is bounded by $\bar{\cal C}$ and 
coherent 
risk measures 
are $1$-Lipschitz in $\|\cdot\|_\infty$, 
we have
$\|V_1\|_\infty=\|\mathcal T V_0\|_\infty\le \bar{\cal C}$. Combining with \eqref{eq:geom-error} gives
$\|V_I-V^*\|_\infty\le \frac{\gamma^I}{1-\gamma}\bar{\cal C}$. Therefore, if
\[
I\ \ge\ \Big\lceil \frac{\ln\!\big(\tfrac{\bar{\cal C}}{(1-\gamma)\varepsilon}\big)}{\ln(1/\gamma)} \Big\rceil,
\]
then $\|V_I-V^*\|_\infty\le \varepsilon$. The conclusion follows by setting $\hat V^*:=V_I$.

Suppose for the moment that $V$ is $L_V$-Lipschitz in $s$ (uniformly with respect to $a,\xi$).
Under the Lipschitzness of 
$\mathcal C$ and $g$,
we can derive for any $s_1,s_2\in\mathcal S$ and $a\in\mathcal A$, 
\[
\begin{aligned}
\big| \rho_\mu\!\circ\rho_{P_\theta}\!\big[\mathcal C(s_1,a,\xi)+\gamma V(g(s_1,a,\xi),\mu')\big]
      - \rho_\mu\!\circ\rho_{P_\theta}\!\big[\mathcal C(s_2,a,\xi)+\gamma V(g(s_2,a,\xi),\mu')\big]\big| &\le  \big(L_{\mathcal C}+\gamma L_g L_V\big)\,\|s_1-s_2\|,
\end{aligned}
\]
which implies that $\mathcal T V$ is $\big(L_{\mathcal C}+\gamma L_g L_V\big)$-Lipschitz in $s$.
Thus, if $V_i$ is Lipschitz continuous with modulus $L_i$, then
$V_{i+1}$ is also Lipschitz continuous with modulus
\[
L_{i+1}\ \le\ L_{\mathcal C}+\gamma L_g\,L_i,\;\; \text{for all}\; i\ge 0.
\]
Starting from $i=0$, we observe that 
$V_0\equiv 0$, which  is trivially Lipschitz continuous in $s$ with modulus
$L_0=0$. The 
recursion above ensures that $V_I$ is Lipschitz continuous in $s$ with
\[
L_I\ \le\ L_{\mathcal C}\sum_{j=0}^{I-1}(\gamma L_g)^j
= \frac{L_{\mathcal C}\,\bigl(1-(\gamma L_g)^I\bigr)}{1-\gamma L_g}.
\]
Since $\hat V^*=V_I$, we obtain 
\(
L_{\hat V^*}
= \frac{L_{\mathcal C}\left(1-(\gamma L_g)^I\right)}{1-\gamma L_g}.
\)
%

Part (ii). 
In the case when
$\gamma L_g<1$, the rhs of the equation is strictly increasing in $I$. Thus,
by letting 
$I\to\infty$,
we obtain an upper  bound
$L_{V^*}\le \frac{L_{\mathcal C}}{1-\gamma L_g}$ for 
the Lipschitz modulus of the 
optimal value function $V^*$.
\end{proof}

\begin{remark} In the literature,
the transition mapping $g$ is  nonexpansive in the state variable $s$, i.e., 
$\|g(s_1,a,\xi)-g(s_2,a,\xi)\|\le \|s_1-s_2\|$ for all $s_1,s_2\in\mathcal S$, $a\in\mathcal A$, and $\xi\in\Xi$, see, e.g.,\cite{shapiro2023episodic,powell2012approximate,yang2020wasserstein,zipkin2008structure}. 
The condition is particularly sensible 
when ${\cal S}$ is bounded.
Under the circumstance,  the condition 
$L_g\le 1$ is trivially satisfied.
\end{remark}

\section{Hyper–parameter approach for discretization of posterior distribution space 
and SAA}

As we commented earlier, 
Algorithms~\ref{alg:A}–\ref{alg:B}
can be implemented in practice only when 
$\mathcal{S} \times \mathcal{D}_1^p$
is a discrete set. 
Since the belief space $\mathcal{D}_1^p$ is often continuous, we propose an approach 
for discretizing the belief space in this section. To this end, we confine our discussions to the case when the prior distributions can be represented 
as a family of 
parametric distributions parameterized by a finite dimensional vector of parameters. 
The hyperparameterization technique 
is widely used in the literature of 
machine learning.
The next assumption specifies this.

\begin{assumption}
    Every posterior belief $\mu\in\mathcal{D}_1^p$ 
    can be uniquely determined by
        a finite–dimensional vector $h\in\mathcal{H}\subseteq\mathbb{R}^k$, i.e.,
        $\mu(\theta)$ can be written as a parametric distribution with pdf $\mu_h(\theta)$.
 The Bayesian 
 update \eqref{eq:MDP-BCR-c} can be equivalently 
represented by update of parameters 
$h' :=h+H(\xi)$,
 where $H(\xi)$ denotes the sufficient statistic with observed $\xi \in \Xi$. 
\end{assumption}

The update \(h' := h + H(\xi)\) 
is common in  conjugate families such as
exponential‐family models with conjugate priors \cite{gelman2013bayesian}.
By Bayes’ rule, the posterior pdf \(\mu'(\theta)\) is a function of the prior pdf  \(\mu(\theta)\) and the new observation \(\xi\). 
If \(\mu_h\) is uniquely parameterized by \(h \in \mathbb{R}^k\), then the posterior’s parameter \(h'\) must be of the form \(h' = F(h,\xi)\), see Example \ref{example-2}. 
In the conjugate exponential‐family case, \(F\) 
can be simplified to adding the sufficient statistic \(H(\xi)\) to  \(h\).
In many cases, 
the hyper-parameter set $\mathcal{H}$ is a continuous set, which 
requires 
to solve a continuum  of BCR 
problems  \eqref{eq:BCR-Alg1} and 
 \eqref{eq:BCR-Alg2}
in Algorithms~\ref{alg:A}–\ref{alg:B}.
This prompts us to adopt the scenario reduction technique from MSP and clustering technique to further discretize $\mathcal{H}$.
Subsequently, we only need to solve a finite number of BCR minimization problems in Algorithms~\ref{alg:A}–\ref{alg:B}.
Furthermore, since the random variables $\xi$ and $\theta$ are usually continuously distributed, we further use 
the well-known Sample Average Approximation (SAA) 
to discretize these continuous distributions.

To concretize the parametrization process for both the finite and infinite time  horizon cases, we introduce a new index 
\[
  \tau =
  \begin{cases}
     t   & \text{in dynamic programming (DP)
     for finite  horizon}, \\
     I-i & \text{in value iteration (VI) for infinite horizon},
  \end{cases}
\]
which unifies 
 the \textbf{stage index} $t=1,2,\cdots,T$ for the finite horizon and the \textbf{iteration index} $i=1,\cdots,I$ for the infinite horizon.  
In the rest of the section, we use $\tau$ unless specified otherwise.
For each $\tau$, the update
$
  h_{\tau+1}=h_{\tau} + H(\xi)
$
incrementally moves the hyper-parameter along its support. Consequently, at step 
$\tau$ the only reachable hyper-parameters belong to
\bgeqn 
\label{eq:H_tau-sec6}
  \mathcal{H}_\tau \;=\;\left\{\,h_1 + \sum_{j=1}^{\tau-1}H(\xi_j)\colon \xi_j\in\Xi,\, \; \text{for}\; j=1,\cdots,\tau-1\right\},
\edeqn 
which is 
a subset of 
the full space
\(\mathcal{H}\). 
The next example illustrates this.

\begin{example}[Dirichlet–Categorical model]
  Let \(\theta=(\theta_1,\dots,\theta_M)\) be the unknown probability vector on categories \(1,\dots,M\),
  with prior
  \[
    \theta\sim\mathrm{Dirichlet}(\alpha_{0,1},\dots,\alpha_{0,M}).
  \]
  At each step \(\tau\), we observe a \(\xi_\tau\in\{1,\dots,M\}\), and update the posterior hyper-parameters
  \[
    \alpha_{\tau+1,i}
    =\alpha_{\tau,i} + \mathds{1}_{\{\xi_\tau = i\}},
    \quad i=1,\dots,M.
  \]
  Hence
  \[
    \mathcal{H}_\tau
    = \bigl\{\alpha_0 + n: n=(n_1,\dots,n_M),\, n_i\in\mathbb{N},\,\sum_i n_i = \tau-1\bigr\},
  \]
  and indeed the whole hyper-parameter space \(\mathcal{H} = \bigcup_\tau\mathcal{H}_\tau\), while
  \(\mathcal{H}_\tau\subseteq\mathcal{H}\) with
  \(\mathcal{H}_\tau\cap\mathcal{H}_{\tau'}=\varnothing\) for \(\tau\neq \tau'\).  
\end{example}

 Based on the discussions above, 
 we propose 
 a step‑wise adaptive grid generation algorithm as follows.

\begin{algorithm}[ht]
\caption{Step-Wise Adaptive Grid for Hyper-Parameters}
\label{alg:adaptive-grid}
\begin{algorithmic}[1]
    \STATE Initialize candidate set $\Upsilon \leftarrow \emptyset$, support size \(B\), 
      max grid size \(M_{\max}\),
      projection tolerance \(\epsilon\), representative set $\mathcal{H}_1^{\mathrm{rep}} \leftarrow \{h_1\}$ and projection map $\operatorname{Rep}_1(h_1) \leftarrow h_1$.
\FOR{$\tau = 1,2,\cdots,T-1$ (DP) \textbf{or} $\tau=1,2,\cdots,I$ (VI)}
    \STATE \textbf{(1) Expansion Step:} 
     \FORALL{$h \in \mathcal{H}_\tau^{\mathrm{rep}}$}
        \STATE Calculate $\{\xi^{(b)}\}_{b=1}^B \subseteq \Xi$, which forms an $\epsilon$‑net of $\mathbb{B}(0,R):=\{\xi\in\Xi:\|\xi\|\leq R\}$.
        \FOR{$b = 1,\cdots,B$}
            \STATE Compute exact hyper-parameter $h^{\mathrm{exact}} \leftarrow h + H(\xi^{(b)})$.
            \STATE Calculate weight $w_b \leftarrow \int_{\Theta}p(\xi^{(b)}|\theta)\mu_h(\theta)d\theta$ (use numerical integration or closed-form, if possible).
            \STATE Add 
            $(h^{\mathrm{exact}},w_b)$ 
            to set $\Upsilon$.
        \ENDFOR
    \ENDFOR
    
    \STATE \textbf{(2) Clustering Step:}
    \WHILE{$|\Upsilon| > M_{\max}$}
        \STATE Identify closest pair $(h_i,h_j)$ in $\Upsilon$ by Euclidean distance.
        \STATE Merge pair into a new center:
        \[
          h^{\mathrm{new}}:=\frac{h_iw_i+h_jw_j}{w_i+w_j}, \quad w^{\mathrm{new}}:=w_i+w_j.
        \]
        \STATE Replace $(h_i,w_i),(h_j,w_j)$ with $(h^{\mathrm{new}},w^{\mathrm{new}})$ {\color{black} and relabel the elements in $\Upsilon$}.
    \ENDWHILE
    \STATE Update $\mathcal{H}_{\tau+1}^{\mathrm{rep}} \leftarrow \{h : (h,\cdot) \in \Upsilon\}$.
    
    \STATE \textbf{(3) Projection Guarantee:}
    \FORALL{$(h^{\mathrm{exact}},\cdot) \in \Upsilon$}
        \IF{$\min_{h\in \mathcal{H}_{\tau+1}^{\mathrm{rep}}}\|h-h^{\mathrm{exact}}\|_2 > \epsilon$}
            \STATE Add $h^{\mathrm{exact}}$ explicitly to $\mathcal{H}_{\tau+1}^{\mathrm{rep}}$.
        \ENDIF
    \ENDFOR
    
    \STATE \textbf{(Optional Re-merge)} if $|\mathcal{H}_{\tau+1}^{\mathrm{rep}}| > M_{\max}$.
    
    \STATE \textbf{(4) Define Projection Map:}
    \STATE Set projection $\operatorname{Rep}_{\tau+1}(h) \leftarrow \arg\min_{h'\in \mathcal{H}_{\tau+1}^{\mathrm{rep}}} \|h-h'\|_2$.
\ENDFOR

\RETURN $\{\mathcal{H}_\tau^{\mathrm{rep}}\}_{\tau}$ and projection maps $\{\operatorname{Rep}_\tau\}_{\tau}$.

\end{algorithmic}
\end{algorithm}

 Algorithm \ref{alg:adaptive-grid} step‑wise constructs a finite, adaptive grid
$\Upsilon$
over the posterior hyper-parameter space. At each step, we first expand the previous grid 
$\Upsilon$
by sampling a fixed number of support points $\{\xi^{(b)}\}_{b=1}^B \subseteq \Xi$
and computing the corresponding hyper-parameter 
and likelihood weights, which captures all newly reachable posterior directions. 
To control the overall grid size, we then cluster the expanded grid set by repeatedly merging the two closest hyper-parameters into a single weighted centroid whenever the grid size exceeds a prescribed maximum. This helps us preserve the  representativeness of the grid while controlling its 
size. To 
control the resulting approximation error, a projection‑guarantee step re‑inserts any original candidate grid that lies beyond a predefined distance from all grid points, which ensures all significant posterior modes are included. Finally, we project any hyper-parameter to its nearest grid point with respect to Euclidean distance, which enables efficient belief updates in later DP or VI loops. The combination of above expansion, weighted clustering, and projection operations properly balances between accurately approximating an infinite‑dimensional belief space and maintaining a tractable, finite grid for efficiently solution of BCR-SOC/MDP problems.

Recall that $Q_{\mu}$ denotes the probability 
distribution of $\theta$
associated with 
pdf $\mu$. To ease the notation, we introduce the notation $Q_{h}:=Q_{\mu_h}$ to represent the probability measure corresponding to $h \in {\cal H}$.
Next, we move on to 
quantify the errors of approximation in Algorithm \ref{alg:adaptive-grid}. 
To this end, we need to make the following generic assumption.


\begin{assumption}\label{ass:holder}
  For all $h\in\mathcal{H}$ and some $\varepsilon>0$, there exist constants $C>0$ and $k>0$ such that
  \[
    \mathsf {d\kern -0.07em l}_{\rm K}\bigl(Q_{h},Q_{{h'}}\bigr)\le C_h\bigl\| h - h'\bigr\|^{k},\ \forall h'\in\mathbb{B}(h,\varepsilon)\cap {\cal H}.
  \]
\end{assumption}
This assumption
is satisfied by virtually all conjugate posteriors commonly used in practice, including Gaussian, Beta, Gamma, Dirichlet and other exponential--family posteriors.
Define
\bgeqn 
  {\cal M}_{\cal H_\tau}
  :=\bigl\{ Q_{h}:h\in{\cal H}_\tau\bigr\}\subset {\cal M}_1^p,
  \quad
  {\cal M}_{{\cal H}_\tau^{\mathrm{rep}}}
  :=\bigl\{Q_{h}:h\in{\cal H}_\tau^{\mathrm{rep}}\bigr\}\subset {\cal M}_1^p.
\edeqn
Suppose there exists a $p>0$ such that the $p$-th moment of $\xi$ is finite, thus for all $R>0$, there exists $M_p$ such that
$
  P_{\theta^c}(\|\xi\|>R)
\le \frac{M_p}{R^p}.
$
Based on this mild assumption, we have the following result.

\begin{proposition}
Suppose Assumption \ref{ass:holder} holds and $H\colon \Xi\to\mathbb{R}^{k}$ is Lipschitz continuous with modulus $L_{H}$.
  Let $h_\tau=h_{1} + \sum_{j=1}^{\tau-1}H(\xi_j)$ be the exact hyper-parameter at step $\tau$ generated by the sequence $\boldsymbol\xi^{\tau}:=(\xi_1,\dots,\xi_{\tau-1})$, and $\mathcal{H}_{\tau}^{\mathrm{rep}}$ be the discrete grid
  produced by Algorithm~\ref{alg:adaptive-grid}. 
  For any small positive number $\varepsilon$, if   
$\epsilon\leq\frac{\varepsilon}{(\tau_{\max} - 1)(1 + L_{H})}$, where 
 $\epsilon$ is the precision in 
 Algorithm~\ref{alg:adaptive-grid}, 
  then 
  \bgeqn  \label{eq-th-error-bound}
P^{\otimes (\tau-1)}_{\theta^c} \left( 
    \delta_{\tau} :=  \mathrm{dist}\Bigl(h_\tau,\mathcal{H}_{\tau}^{\mathrm{rep}}\Bigr)=\inf_{h\in\mathcal{H}_{\tau}^{\mathrm{rep}}}\|h_\tau-h\| \le 
    \varepsilon\right) \geq \left(1 - \frac{M_p}{R^p}\right)^{\tau-1},
    \qquad \forall \tau\geq1
  \edeqn 
 and 
 \bgeqn 
P^{\otimes (\tau-1)}_{\theta^c} \left( 
\mathsf {d\kern -0.07em l}_{\rm K}\Bigl(Q_{{h_\tau}},\mathcal{M}_{\mathcal{H}_{\tau}^{\mathrm{rep}}}\Bigr)=\inf_{Q_{h}\in\mathcal{M}_{\mathcal{H}_{\tau}^{\mathrm{rep}}}}\mathsf {d\kern -0.07em l}_{\rm K}\Bigl(Q_{{h_\tau}},Q_h\Bigr)
    \le C_{h_\tau}\varepsilon^k\right)\geq \left(1 - \frac{M_p}{R^p}\right)^{\tau-1},
  \edeqn 
  where $P^{\otimes (\tau-1)}_{\theta^c}$ is the probability measure over product space 
$\Xi^{\otimes (\tau-1)} = \underbrace{\Xi\times \cdots\times \Xi}_{\tau-1}$. 
\end{proposition}

\begin{proof}
The proof proceeds by induction on $\tau$.  Since $\mathcal{H}_{1} = \mathcal{H}_{1}^{\mathrm{rep}} = \{h_{1}\}$ exactly, we have $\delta_{1} = 0$.
  In the expansion step, for each representative $h_{1}\in
  \mathcal{H}_{1}^{\mathrm{rep}}$, we draw 
  \ samples
  $\{\xi^{(b)}\}_{b=1}^{B}\subseteq \mathbb{B}(0,R)$.  By the algorithm, the 
  samples 
  form an $\epsilon$‑net of $\mathbb{B}(0,R)$,
  i.e., for every \(\xi_1\in \mathbb{B}(0,R)\),
  there exists some \(\xi^{(b)}\) with
  \(\| \xi_1 - \xi^{(b)}\| \le \epsilon\).
   Using Lipschitz continuity of $H$, we obtain
  \[
    \| H(\xi_1) - H(\xi^{(b)})\| 
    \le L_{H}\| \xi_1 - \xi^{(b)}\| 
    \le L_{H}\epsilon.
  \]
 Algorithm~\ref{alg:adaptive-grid} ensures projection to the nearest grid point within $\epsilon$, thus for any $\xi_1 \in \mathbb{B}(0,R)$,
  \begin{align*}
    \delta_2=\mathrm{dist}\left(h_{1} + H(\xi_1),\mathcal{H}_{2}^{\mathrm{rep}}\right)
    &\le \| h_{1} + H(\xi_1) - (h_{1} + H(\xi^{(b)}))\| 
           + \mathrm{dist}\bigl(h_{1} + H(\xi^{(b)}),\mathcal{H}_{2}^{\mathrm{rep}}\bigr) \\
    &\le L_{H}\epsilon + \epsilon 
    =(1 + L_{H})\epsilon.
  \end{align*}
  Therefore, when $\xi_1\in \mathbb{B}(0,R)$ (which itself occurs with
  probability at least $1 - \frac{M_p}{R^p}$), every new node
  $h_2=h_{1}+H(\xi_1)\in\mathcal{H}_2$ lies within distance $(1+L_{H})\epsilon$ of
  $\mathcal{H}_{2}^{\mathrm{rep}}$. 
  Thus we have
    \[
    P_{\theta^c}\left(\mathsf {d\kern -0.07em l}_{\rm K}\bigl(Q_{{h_2}},\mathcal{M}_{\mathcal{H}_{2}^{\mathrm{rep}}}\bigr)
    \le C_{h_2}\delta_{2}^{k}
   \right)\geq1 - \frac{M_p}{R^p}.
  \]
  This proves the statement for $\tau=2$.
  Now fix any $\tau\ge 2$ and suppose
   \[
    \delta_{\tau} 
    = \mathrm{dist}\left(h_{1} + \sum_{j=1}^{\tau-1}H(\xi_j),\mathcal{H}_{\tau}^{\mathrm{rep}}\right)
    \le (\tau - 1)(1 + L_{H})\epsilon\leq\varepsilon,
  \]
  with probability at least
  $(1 - \frac{M_p}{R^p})^{\tau-1}$, i.e.,
  \bgeqn
  \label{eq:Prob-H-H-rep}
P^{\otimes (\tau-1)}_{\theta^c} \left( \delta_{\tau} 
    = \mathrm{dist}\left(h_{1} + \sum_{j=1}^{\tau-1}H(\xi_j),\mathcal{H}_{\tau}^{\mathrm{rep}}\right)
    \le (\tau - 1)(1 + L_{H})\epsilon\right)\geq \left(1 - \frac{M_p}{R^p}\right)^{\tau-1}.
  \edeqn 
For each $h_\tau^{\mathrm{rep}}\in \mathcal{H}_{\tau}^{\mathrm{rep}}$, we draw another
  batch of $B$ i.i.d.\ samples $\{\xi^{(b)}\}_{b=1}^{B}$ which form an $\epsilon$‑net of
  $\mathbb{B}(0,R)$. 
  Now consider an arbitrary member of
  \(\mathcal{H}_{\tau+1}\). 
  If
  $\xi_\tau\in \mathbb{B}(0,R)$, pick $\xi^{(b)}$ with $\| \xi_\tau - \xi^{(b)}\| \le
  \epsilon$.  
  Then define
  $
    h_{\tau+1}^{\mathrm{exact}} 
    = h_\tau^{\mathrm{rep}} + H(\xi^{(b)})
  $  and we have
  \[
    \begin{aligned}
      \bigl\| h_{\tau+1} - h_{\tau+1}^{\mathrm{exact}}\bigr\| 
      &=\bigl\| h_\tau + H(\xi_\tau) - \bigl(h_\tau^{\mathrm{rep}} + H(\xi^{(b)})\bigr)\bigr\| \\
      &\le{\| h_\tau - h_\tau^{\mathrm{rep}}\|}
             + {\| H(\xi_\tau) - H(\xi^{(b)})\|}
      \le \delta_{\tau} + L_{H}\epsilon.
    \end{aligned}
  \]
  Next, the algorithm’s projection‐guarantee step ensures that once
  $h_{\tau+1}^{\mathrm{exact}}$ is inserted into the candidate set
  \(\Upsilon\), any final projection to
  \(\mathcal{H}_{\tau+1}^{\mathrm{rep}}\) is within $\epsilon$.  Hence
  \[
    \mathrm{dist}\bigl(h_{\tau+1},\mathcal{H}_{\tau+1}^{\mathrm{rep}}\bigr)
    \le
    \| h_{\tau+1} - h_{\tau+1}^{\mathrm{exact}}\| +  \mathrm{dist}\bigl(h_{\tau+1}^{\mathrm{exact}},\mathcal{H}_{\tau+1}^{\mathrm{rep}}\bigr)
    \le
     \delta_{\tau} +  L_{H}\epsilon+\epsilon
    = \tau(1 + L_{H})\epsilon\leq\varepsilon.
  \]
  Finally, we have
  \[
    \mathsf {d\kern -0.07em l}_{\rm K}\bigl(Q_{{h_{\tau+1}}},\mathcal{M}_{\mathcal{H}_{\tau+1}^{\mathrm{rep}}}\bigr)
    \le C_{{h_{\tau+1}}}\delta_{\tau+1}^{k}
    \le C_{{h_{\tau+1}}}\varepsilon^k.
  \]
  The  probability over step $\tau+1$ is bounded
  by
 $(1 - \frac{M_p}{R^p})^{\tau}$.  This completes the induction.
\end{proof}

\begin{assumption}\label{H-5}
Let $\rho_{P_\theta}$ and 
 $\rho_{\mu_h}$ 
 be risk measures
 over $L^p$.
 Let $Z:\Xi\to\R$ be 
 a real–valued function and 
 $Y(\theta):=\rho_{P_\theta}(Z(\xi))$. 
 For each fixed $h\in\mathcal{H}$,
 there exist constants $L_{h,\rho}>0$, $k>0$ and $\varepsilon>0$ such that  
\bgeqn
    |\rho_{\mu_h}\circ\rho_{P_\theta}(Z)-\rho_{\mu_{h'}}\circ\rho_{P_\theta}(Z)|\leq L_{h,\rho}\|h - h'\|^{k},\ \forall h'\in \mathbb{B}(h,\varepsilon)\cap \mathcal{H}. 
\edeqn 
\end{assumption}
The following proposition provides conditions under which Assumption \ref{H-5} is satisfied.
\begin{proposition}[Sufficient condition for Assumption \ref{H-5}]
\label{prop-out-lip}
    Suppose Assumption~\ref{ass:holder} holds, $Y(\theta)$ is Lipschitz continuous in $\theta$ with modulus $L_{\rho}^{\rm in}$ and
   the outer risk measure $\rho_{\mu_h}$
   is a coherent risk measure for any $h\in {\cal H}$.
    Then there exists a constant $L_{h,\rho}$ such that
    \bgeqn 
    |\rho_{\mu_h}\circ\rho_{P_\theta}(Z)-\rho_{\mu_{h'}}\circ\rho_{P_\theta}(Z)|\leq L_{h,\rho}\|h - h'\|^{k},\ \forall h'\in\mathbb{B}(h,\varepsilon)\cap {\cal H}.
    \edeqn 
\end{proposition}
\begin{proof}
Let $\nu_{h} :=Q_{h}\circ Y^{-1}$ and $\nu_{{h'}}: =Q_{{h'}}\circ Y^{-1}$
be probability measures on $\R$ induced by 
$Y$.
Given that $Y$ is Lipschitz continuous with constant $L_{\rho}^{\mathrm{in}}$ and Assumption \ref{ass:holder}, we have
\bgeq \dd_{K}\bigl(\nu_{h},\nu_{{h'}}\bigr) &=& \sup_{g\in {\cal G}_1}\left|
\int_\R g(x) Q_h\circ Y^{-1}(dx)-
\int_\R g(x) Q_{h'}\circ Y^{-1}(dx)\right|\\
&=&
\sup_{g\in {\cal G}_1}\left|
\int_\Theta g(Y(\theta)) Q_h(d\theta)
-
\int_\Theta g(Y(\theta)) Q_{h'}(d\theta)
\right|\\
&\le&
  L_{\rho}^{\rm in}\;\dd_{K}\bigl(Q_h,Q_{h'}\bigr)
  \;\le\;
  L_{\rho}^{\rm in}\,C_h\;\| h-h'\|^{k},
\edeq
which gives
\[
\begin{aligned}
  \Bigl|\,
    \rho_{\mu_h}\circ\rho_{P_\theta}(Z)
    -\rho_{\mu_{h'}}\circ\rho_{P_\theta}(Z)
  \Bigr|
  =\bigl|\rho_{{\mu_h}}(Y)-\rho_{{\mu_{h'}}}(Y)\bigr|
  \le L^{\rm out}_\rho
 \dd_{K}\bigl(\nu_{h},\nu_{{h'}}\bigr)
  \le
  L^{\rm out}_\rho L_{\rho}^{\rm in}C_h\;\| h-h'\|^{k},
\end{aligned}
\]
where the first inequality follows from the Lipschitz continuity of the outer risk measure in \eqref{eq:Lip-RM-rho} with constant $L^{\rm out}_\rho$. Setting $L_{h,\rho}=L^{\rm out}_\rho L_{\rho}^{\rm in}C_h$ completes the proof.
\end{proof}
Proposition \ref{prop-out-lip} remains valid for a broad class of non-coherent outer risk measures satisfying \eqref{eq:Lip-RM-rho} such as ERM and SRM \cite{bhat2019concentration}. For a detailed discussion of these risk measures and their associated Lipschitz modulus $L_\rho^{\rm out}$, we refer to \cite[Appendix B]{liang2024regret}. 
In the following, we consider the DP method for finite horizon problems for simplicity. It is worth noting that similar results can be obtained for VI in infinite-horizon problems by simply modifying the initial value function.

\begin{theorem}[Local Lipschitz continuity of value function in $h$]
\label{thm:V-h-Lipschitz}
Suppose Assumption \ref{H-5} holds, define $
  L_{h,T} = L_{h,\rho}\bar{\mathcal{C}}, $
and for \(t=T-1,T-2,\cdots,1\) set
\begin{equation}\label{eq:Lh-recursion}
  L_{h,t}
  :=L_{h,\rho}+
 \gamma L_{h,t+1}.
\end{equation}
Then for each \(t\)
and fixed \(h\in\mathcal H\), 
\bgeqn 
\label{eq:thm5-V-Lip}
  \sup_{s\in \mathcal{S}}\; \bigl|
    V_t^*(s,\mu_h)-V_t^*(s,\mu_{h'})\bigr|
  \le
  L_{h,t}\|h - h'\|^{k},\ \forall h'\in\mathbb{B}(h,\varepsilon)\cap {\cal H}.
\edeqn 
\end{theorem}

\begin{proof}
We prove by backward induction. At $t=T$, 
since \(\mathcal C_T(s)\) is independent 
of \(h\), then $V_T^*(s,\mu_{h})$ is also independent of $h$. Thus
\[
  \bigl|V_T^*(s,\mu_h) - V_T^*(s,\mu_{h'})\bigr| = 0 
  \le L_{h,T}\|h-h'\|^{k}
\]
for any \(s\in\mathcal S\), \(h\in\mathcal H\) and $ h'\in\mathbb{B}(h,\varepsilon)\cap {\cal H}$.
Next,  assume for some \(t+1 \le T\) that
\bgeqn 
  \bigl|V_{t+1}^*(s,\mu_h)-V_{t+1}^*(s,\mu_{h'})\bigr|
  \le
  L_{h,t+1}\|h - h'\|^{k},
 \ \forall h'\in\mathbb{B}(h,\varepsilon)\cap {\cal H},
\label{eq:indct-V^*-t+1}
\edeqn 
where \(L_{h,t+1}\) is defined as in \eqref{eq:Lh-recursion}.
For fixed  \(s\in\mathcal S\),  \(h\in\mathcal H\), and \(a\in\mathcal A\), define 
\[
  Z_h(\xi)
  :=
  \mathcal C_t\bigl(s,a,\xi\bigr)
  +
  \gamma V_{t+1}^*\!\Bigl(g_t(s,a,\xi),\mu_{h + H(\xi)}\Bigr).
\]
Under Assumption \ref{H-5}, 
\begin{equation}\label{eq:outer-diff}
  \bigl|
     \rho_{\mu_{h}}\circ\rho_{P_\theta}(Z_h)
    -  \rho_{\mu_{h'}}\circ\rho_{P_\theta}(Z_h)
  \bigr|
  \le
  L_{h,\rho}\|h - h'\|^{k},\ \forall h'\in\mathbb{B}(h,\varepsilon)\cap {\cal H}.
\end{equation}
By Lemma \ref{lem-beo-2}, for each \(\theta\), we have
\[
\begin{aligned}
 \bigl|
    \rho_{P_\theta}\bigl(Z_h\bigr)
    -
    \rho_{P_\theta}\bigl(Z_{h'}\bigr)
  \bigr|
  &\le
  \|
    Z_h(\xi) - Z_{h'}(\xi)
  \bigr\|_{\xi,\infty}\\
  &= 
  \bigl\|
    \mathcal C_t(s,a,\xi) + \gamma V_{t+1}^*(g_t(s,a,\xi),\mu_{h + H(\xi)})
    - \mathcal C_t(s,a,\xi) - \gamma V_{t+1}^*(g_t(s,a,\xi),\mu_{h' + H(\xi)})
  \bigr\|_{\xi,\infty}\\
  &= 
  \gamma
  \bigl\|
    V_{t+1}^*(g_t(s,a,\xi),\mu_{h + H(\xi)})
    - V_{t+1}^*(g_t(s,a,\xi),\mu_{h' + H(\xi)})
  \bigr\|_{\xi,\infty}\\
  &\le \gamma L_{h,t+1}
    \bigl\|(h + H(\xi)) - (h' + H(\xi))\bigr\|^{k} \quad \text{(by (\ref{eq:indct-V^*-t+1}))}\\
    &  =\gamma L_{h,t+1}\|h - h'\|^{k}.
  \end{aligned}
\]
The inequality above and  translation invariance of $\rho_{\mu_{h'}}$ imply
\begin{equation}\label{eq:inner-diff}
  \bigl|
     \rho_{\mu_{h'}}\circ\rho_{P_\theta}(Z_h)
    -  \rho_{\mu_{h'}}\circ\rho_{P_\theta}(Z_{h'})
  \bigr|
  \le
  \gamma L_{h,t+1}\|h - h'\|^{k}.
\end{equation}
From \eqref{eq:outer-diff} and \eqref{eq:inner-diff}, for each fixed \(a\in\mathcal A\) we get
\[
\begin{aligned}
 \bigl|
     \rho_{\mu_{h}}\circ\rho_{P_\theta}(Z_h)
    -  \rho_{\mu_{h'}}\circ\rho_{P_\theta}(Z_{h'})
  \bigr|\le
  (L_{h,\rho}+
 \gamma L_{h,t+1})\|h - h'\|^{k}.
\end{aligned}
\]
Since 
$
  V_t^*(s,\mu_h) =\min_{a\in\mathcal A} \rho_{\mu_{h}}\circ\rho_{P_\theta}(Z_h(\xi))
$
and similarly for \(h'\), we have
\[
  \bigl|
    V_t^*(s,\mu_h) - V_t^*(s,\mu_{h'})
  \bigr|
  \le
  \sup_{a\in\mathcal A}\left| \rho_{\mu_{h}}\circ\rho_{P_\theta}
  (Z_h(\xi))-
   \rho_{\mu_{h'}}\circ\rho_{P_\theta}
  (Z_{h'}(\xi))\right|
  \le
  (L_{h,\rho}+
 \gamma L_{h,t+1})\|h - h'\|^{k},
\]
which leads to \eqref{eq:thm5-V-Lip}.
\end{proof}

We are now ready to state the main result in this section.

\begin{theorem}[Error bound on approximation of value functions due to discretization]
    For any sequence $\boldsymbol \xi^t=\{\xi_1,\cdots,\xi_t\}$, define \(h_t:=h_1+\sum_{j=1}^{t-1}H(\xi_j)\in\mathcal H_t\), and \(h_t^{\mathrm{rep}} = \mathrm{Rep}_t(h_t)\).
        For any small positive number $\varepsilon$, 
  if      
        $\epsilon \le \frac{\varepsilon}{(T - 1)(1 + L_H)}$, where 
$\epsilon$ is the precision in 
Algorithm~\ref{alg:adaptive-grid},
    then
    for $t=1,2,\cdots,T$,
\bgeqn 
 P^{\otimes (t-1)}_{\theta^c} \left(\sup_{s\in \mathcal{S}} \bigl|\,
    V_t^*(s,\mu_{h_t}) - V_t^*(s,\mu_{h_t^{\mathrm{rep}}})\bigr|
\le
  L_{h_t,t}\,\bigl[(t-1)(1 + L_H)\epsilon\bigr]^{k}\right)\geq (1 - \frac{M_p}{R^p})^{t-1}.
\edeqn
\end{theorem}
\begin{proof}
By \eqref{eq:thm5-V-Lip} and \eqref{eq:Prob-H-H-rep},
\bgeq 
&&P^{\otimes (t-1)}_{\theta^c} \left(
\sup_{s\in \mathcal{S}} \bigl|\,
    V_t^*(s,\mu_{h_t}) - V_t^*(s,\mu_{h_t^{\mathrm{rep}}})\bigr|
\le
  L_{h_t,t}\,\bigl[(t-1)(1 + L_H)\epsilon\bigr]^{k}
\right)\\
&\geq& 
P^{\otimes (t-1)}_{\theta^c} \left(
  L_{h_t,t}\|h_t - h_t^{\mathrm{rep}}\|^{k}
\le
  L_{h_t,t}\,\bigl[(t-1)(1 + L_H)\epsilon\bigr]^{k}
\right)\\
&\geq& 
P^{\otimes (t-1)}_{\theta^c} \left(
 \|h_t - h_t^{\mathrm{rep}}\|
\le
(t-1)(1 + L_H)\epsilon
\right)\geq  
    \left(1 - \frac{M_p}{R^p}\right)^{t-1}.
\edeq 
The conclusion follows.
\end{proof}

}

With the adaptively discretized hyper-parameter 
set generated by Algorithm \ref{alg:adaptive-grid}, 
the BCR problems  \eqref{eq:BCR-Alg1} and 
 \eqref{eq:BCR-Alg2}
can 
 then be  approximated by the following problem:
\begin{subequations}
\label{eq-SAA-1}
	\begin{eqnarray} 
	\min _{a \in \mathcal{A}} &&\rho_{\mu_{h_\tau}}\circ\rho_{P_{\theta}}\left[\mathcal{C}(s, a, \xi)+\gamma V\left(s^{\prime}, \mu_{h^{\prime}}\right)\right]\\   
			\text{s.t.} && s'=g\left(s, a, \xi\right),\  h'={\operatorname{Rep}_{\tau+1}}(h_\tau+H(\xi)),
	\end{eqnarray} 
	\end{subequations}
where \(V\) is 
the value function defined 
over \(\mathcal{S}\times \mathcal{H}^{\mathrm{rep}}_{\tau+1}\).
In the rest of this section, we discuss how an appropriate SAA method can be implemented when the inner and outer risk measures take specific risk measures such as VaR and AVaR.

\subsection{The VaR-Expectation case}
In Example \ref{bcr-2}, 
we explained how BCR-SOC/MDPs 
can be derived 
under the distributionally robust SOC/MDP framework
by setting 
 the outer risk measure
 as $\text{VaR}_{\mu}^{\alpha}$ with $\alpha=0$.
In this section, we revisit the problem but with a general $\alpha\in (0,1)$:
	\begin{subequations}\label{VE}
	\begin{eqnarray} 
            \min _{a \in \mathcal{A}} &&\text{VaR}_{\mu_{h_{\tau}}}^{\alpha} \circ\mathbb{E}_{P_{\theta}}\left[\mathcal{C}(s, a, \xi)+\gamma V\left(s^{\prime}, \mu_{h^{\prime}}\right)\right]\\
			\text{s.t.} && s'=g\left(s, a, \xi\right),\  h'={\operatorname{Rep}_{\tau+1}}({h_{\tau}}+H(\xi)).
	\end{eqnarray} 
	\end{subequations}
    In this formulation, the objective is to minimize 
    the $\alpha$-quantile of the expected costs
  $\mathbb{E}_{P_{\theta}}\left[\mathcal{C}(s, a, \xi)+\gamma V\left(s^{\prime}, \mu_{h^{\prime}}\right)\right]$. 
  This approach is less conservative than a DRO problem
  \begin{eqnarray} 
  \label{eq:Qian-VaR}
  \min_{a\in\mathcal{A}} \max_{\theta\in{\Theta}}\mathbb{E}_{P_{\theta}}\left[\mathcal{C}(s, a, \xi)+\gamma V\left(s^{\prime}, \mu_{h^{\prime}}\right)\right],
  \end{eqnarray} 
as demonstrated in \cite{qian2019composite}.
Further, problem \eqref{VE} can be reformulated as a chance-constrained minimization problem:
	\begin{equation}
		\begin{aligned}
		&\min _{\kappa,a \in \mathcal{A}}  \ \ \kappa\\
		&\text{ s.t.} \ \ \ \  Q_{{h_{\tau}}} (\mathbb{E}_{P_{\theta}}\left[\mathcal{C}\left(s,a,\xi\right)+\gamma V\left(s',\mu_{h^{\prime}}\right) \right]\geq \kappa)\leq \alpha.
		\end{aligned}
		\label{VE_CHANCE}\tag{$\text{P}_{\text{VE}}^{\alpha}$}
	\end{equation}
This reformulation provides a probabilistic interpretation of the $\alpha$-quantile constraint, aligning it with Bayesian decision-making. Specifically, by treating $\mathbb{E}_{P_{\theta}} \left[\mathcal{C}(s, a, \xi) + \gamma V(s', \mu_{h^{\prime}}) \right] - \kappa$ as a function of $\theta$, the chance constraint in \eqref{VE_CHANCE} can be understood as Bayesian posterior feasibility, as defined in \cite{gupta2019near}. 
Thus, problem \eqref{VE_CHANCE} can be equivalently reformulated as a DRO problem with a Bayesian ambiguity set, as proposed in \cite{gupta2019near}. 

Below, we discuss how to use SAA to discretize 
\eqref{VE}. To this end, let $\theta_i, i=1,2,\cdots,N$, 
be iid 
random variables having 
posterior distribution $\mu_{h_{\tau}}$.
Analogous to \cite{luedtke2008sample,qian2019composite}, we
can formulate the sample average approximation 
of \eqref{VE_CHANCE} as:  
	\begin{equation}
	    \begin{aligned} \min _{\kappa,a \in \mathcal{A},z_i \in\{0,1\}}  \ \  &\kappa\\
			\text{s.t.}\ \ \ \ \ \ \ \  &-Uz_i+\mathbb{E}_{P_{\theta_{i}}}\left[\mathcal{C}\left(s,a,\xi\right)+\gamma V\left(s',\mu_{h^{\prime}}\right) \right]\leq \kappa, \ i=1,2,\cdots,N,\\
			&\sum_{i=1}^Nz_i=\lfloor \alpha N \rfloor,
	    \end{aligned}\label{VE_MINLP} \tag{$\text{P}_{\text{VE}}^{N,\alpha}$}
	\end{equation}		
where $z_i$, $i=1,\cdots N$ are auxiliary 
variables taking integer values $\{0,1\}$,
$\lfloor b\rfloor$ stands for the largest integer value
below $b$ and 
$U$ is a large 
positive constant 
exceeding 
$$
\max_{a \in \mathcal{A},i=1,\cdots,N} \mathbb{E}_{P_{\theta_{i}}}\left[\mathcal{C}\left(s,a,\xi\right)+\gamma V\left(s',\mu_{h^{\prime}}\right) \right]- \min_{a \in \mathcal{A},i=1,\cdots,N} \mathbb{E}_{P_{\theta_{i}}}\left[\mathcal{C}\left(s,a,\xi\right)+\gamma V\left(s',\mu_{h^{\prime}}\right) \right].
$$
The next proposition states convergence of 
the problem \eqref{VE_MINLP} to its true counterpart \eqref{VE_CHANCE} in terms of the optimal value.
    
    \begin{proposition}\label{sam-gua}
Assume: (a)        
        $\mathbb{E}_{P_{\theta}}
	\left[\mathcal{C}\left(s,a,\xi\right)
		+\gamma V\left(s',\mu_{h^{\prime}}\right) \right]$ 
	is globally Lipschitz continuous 
    in $a$ with 
    modulus $L$ 
    uniformly 
    for all $s\in\mathcal{S}$ and $\theta\in \Theta$, 
    (b) 
the optimal values $\vartheta^{\alpha}_{VE}$ and
		$\vartheta^{N,\alpha}_{VE}$
of 
problems \eqref{VE_CHANCE} and \eqref{VE_MINLP} are finite.
        Let
$0 <\iota < \alpha, \epsilon \in(0,1), \eta > 0$,  $D\geq ||a||$ for all $a\in\mathcal{A}$ and
		\begin{equation}\label{eqsam-gua1}
			N_0=\frac{2}{\iota^2}\left[ \log(\frac{1}{\epsilon})+n\log \lceil\frac{2LD}{\eta}\rceil+n\log \lceil \frac{2}{\iota}\rceil\right].
		\end{equation}
        Then with probability at least $1-2\epsilon$, 
		\begin{equation}\label{eqsam-gua2}
			\vartheta^{\alpha+\iota}_{VE}-\eta \leq \vartheta^{N,\alpha}_{VE}\leq \vartheta^{\alpha-\iota}_{VE} 
		\end{equation}
    for all $N\geq N_0$.       
	\end{proposition}

	\noindent \textbf{Proof.}
    By Theorem 3 in \cite{luedtke2008sample}, we can 
    deduce that $\vartheta^{N,\alpha}_{VE}\leq \vartheta^{\alpha-\iota}_{VE}$  with probability at least $1-\epsilon$ 
    for all 
    $N \geq \frac{1}{2\iota^2} \log(\frac{1}{\epsilon})$.
    Moreover, 
    following a similar argument to 
    the proof of 
    \cite[Theorem 10]{luedtke2008sample},
    we can show that 
    a feasible solution to the problem
    \begin{equation}
			\begin{aligned}
			\min _{\kappa,a \in \mathcal{A},z_i \in\{0,1\}} & \kappa\\
			\text{s.t.} \quad	&-Uz_i+\mathbb{E}_{P_{\theta_{i}}}\left[\mathcal{C}\left(s,a,\xi\right)+\gamma V\left(s',\mu_{h^{\prime}}\right) \right]+\eta\leq \kappa, \ i=1,2,\cdots,N,\\
				&\sum_{i=1}^Nz_i=\lfloor \alpha N \rfloor,
			\end{aligned}
			\label{VE_feasi}
		\end{equation}
	is also a feasible
    solution to 
    problem $(\text{P}_{\text{VE}}^{\alpha+\iota})$ with probability at least $1-\epsilon$ 
    for all 
    $N\geq N_0$. 
    Thus, the optimal value of 
    problem \eqref{VE_feasi} 
    can be written as 
    $\vartheta^{N,\alpha}_{VE}+\eta$. This means that 
    when the solution 
    to problem \eqref{VE_feasi} is 
    feasible for problem $(\text{P}_{\text{VE}}^{\alpha+\iota})$, it holds that $\vartheta^{N,\alpha}_{VE}+\eta \geq\vartheta^{\alpha+\iota}_{VE}$. The conclusion follows as  $N_0 \geq \frac{1}{2\iota^2} \log(\frac{1}{\epsilon})$.
	\hfill $\Box$

	Note that in (\ref{eqsam-gua1}), the sample size $N_0$ 
    depends on the logarithm
    of $1/\epsilon$ and $1/\eta$,
    which means $N_0$ increases
    very slowly when $\epsilon$ and $\eta$
    are driven to $0$.
    Indeed, the 
    numerical experiments in \cite{luedtke2008sample} show that the bounds in Proposition \ref{sam-gua} are over conservative, potentially allowing for a smaller $N$ in practice.
	Note also that
	 when $\mathbb{E}_{P_{\theta}}\left[\mathcal{C}\left(s,a,\xi\right)+\gamma V\left(s',\mu_{h^{\prime}}\right) \right]$  does not have a closed-form and the dimension of $a$ is large,
          we need to 
          use sample average to 
          approximate $\mathbb{E}_{P_{\theta}}\left[\mathcal{C}\left(s,a,\xi\right)+\gamma V\left(s',\mu_{h^{\prime}}\right) \right]$. 
          Consequently, problem \eqref{VE_MINLP} can be approximated further by the following MINLP:
	 \begin{equation}
	 	\begin{aligned}
	 		\min _{\kappa,a \in \mathcal{A},z_i \in\{0,1\}} & \kappa\\
	 	\text{s.t.}\quad	&-Uz_i+\sum_{j=1}^M\left[\mathcal{C}\left(s,a,\xi^{i,j}\right)+\gamma V\left(s^{i,j},\mu_{h^{i,j}}\right) \right]\leq \kappa, \ i=1,2,\cdots,N,\\
	 		&\sum_{i=1}^Nz_i=\lfloor \alpha N \rfloor,
	 	\end{aligned}
	 	\label{VE_MINLP2}\tag{$\text{P}_{\text{VE}}^{N,M,\alpha}$}
	 \end{equation}
where 
$(\xi^{i,1},\cdots,\xi^{i,M})$ 
are
i.i.d. samples generated from $P_{\theta_i}$ for $i=1, \cdots, N$.
	Denote the optimal value of problem \eqref{VE_MINLP2} by $\vartheta^{N,M,\alpha}_{VE}$. The next
    proposition 
    states that $\vartheta^{N,M,\alpha}_{VE}$
    lies in a neighborhood of $\vartheta^{N,\alpha}_{VE}$ with a high probability when the sample size $M$ is sufficiently large.
    
	\begin{proposition}\label{prop-4}
		Suppose: (a) there exists a measurable function $\varphi:\Xi\to\mathbb{R}^+$ such that 
        \begin{eqnarray} 
        \left|\mathcal{C}\left(s,a_1,\xi\right)+\gamma V\left(s_1',\mu_{h^{\prime}}\right)-\mathcal{C}\left(s,a_2,\xi\right)+\gamma V\left(s_2',\mu_{h^{\prime}}\right)\right|\leq\varphi(\xi)\left\|a_1-a_2\right\|
        \end{eqnarray} 
  for all $s\in\mathcal{S}$, $a_1,a_2\in\mathcal{A}$ and $\xi\in\Xi$,
  (b) the moment generating function of $\varphi(\xi)$ is finite valued in a neighborhood of 0. 
  Define 
  \[
  M_0=\frac{8\varsigma^2}{\delta^2}\left[\log\left(1+\frac{D^n}{\upsilon^{n}}\right)+\log(\frac{1}{\varepsilon})\right],
  \]
		where $\varsigma$ and $\upsilon$ are positive 
        constants corresponding to the given objective function 
        $\mathcal{C}\left(s,a,\xi\right)+\gamma V\left(s',\mu_{h^{\prime}}\right)$ and distribution $P_{\theta}$. 
        Then 
        with probability at least $1-2\varepsilon$, we have
		\begin{equation}\label{eqsam-gua3}
			\vartheta^{N,\alpha}_{VE}-\delta\leq \vartheta^{N,M,\alpha}_{VE}\leq \vartheta^{N,\alpha}_{VE}+\delta\end{equation}
	for all $M\geq M_0$.
    \end{proposition}

\noindent \textbf{Proof.}
	Proposition 2 in \cite{wang2008sample} demonstrates that a feasible solution for problem (\ref{VE_MINLP2}) is also feasible for the problem: 
	\begin{equation}
		\begin{aligned}
			\min _{\kappa,a \in \mathcal{A},z_i \in\{0,1\}} &\kappa\\
			\text{s.t.}\quad &-Uz_i+\mathbb{E}_{P_{\theta_{i}}}\left[\mathcal{C}\left(s,a,\xi\right)+\gamma V\left(s',\mu_{h^{\prime}}\right) \right]\leq \kappa+\delta, \ i=1,2,\cdots,N,\\
			&\sum_{i=1}^Nz_i=\lfloor \alpha N \rfloor,
		\end{aligned}
		\label{VE_feasi2}
	\end{equation} with probability at least $1-\varepsilon$ 
    for all $M\geq M_0$. Observe that the optimal value of problem \eqref{VE_feasi2} is $\vartheta^{N,\alpha}_{VE}-\delta$. Therefore, if a feasible solution to problem (\ref{VE_MINLP2}) is also 
    a feasible solution to 
    problem $(\ref{VE_feasi2})$, then we have $\vartheta^{N,\alpha}_{VE}-\delta \leq\vartheta^{N,M,\alpha}_{VE}$.
Likewise a feasible solution to the following problem 
	\begin{equation}
		\begin{aligned}
			\min _{\kappa,a \in \mathcal{A},z_i \in\{0,1\}} & \kappa\\
			\text{s.t.}\quad &-Uz_i+\mathbb{E}_{P_{\theta_{i}}}\left[\mathcal{C}\left(s,a,\xi\right)+\gamma V\left(s',\mu_{h^{\prime}}\right) \right]\leq \kappa-\delta, \ i=1,2,\cdots,N,\\
			&\sum_{i=1}^Nz_i=\lfloor\alpha N \rfloor,
		\end{aligned}
		\label{VE_feasi3}
	\end{equation} 
	is also feasible for 
    problem (\ref{VE_MINLP2}) with probability at least $1-\varepsilon$  when $M\geq M_0$. The optimal value of problem \eqref{VE_feasi3} is given by $\vartheta^{N,\alpha}_{VE}+\delta$.  Therefore, when the feasible solution 
    to problem \eqref{VE_feasi3} is feasible 
    for problem (\ref{VE_MINLP2}), it holds that $\vartheta^{N,\alpha}_{VE}+\delta \geq\vartheta^{N,M,\alpha}_{VE}$.
\hfill $\Box$

The assumptions (a) and (b) in Proposition \ref{prop-4} can easily be satisfied, and its justification can be found in part (iii) of Proposition \ref{convex}.
From the numerical perspective, we can conveniently use packages such as BARON to solve moderately-sized problems \eqref{VE_MINLP} and \eqref{VE_MINLP2}.
By combining Propositions \ref{sam-gua} and \ref{prop-4},
we are ready to state the
main result of this subsection.

\begin{theorem}
    Assume the settings and 
    conditions 
    of Propositions \ref{sam-gua} and \ref{prop-4} are satisfied. Then
with probability at least $1-2\epsilon-2\varepsilon$, we have
    		$$
			\vartheta^{N,M,\alpha+\iota}_{VE}-\delta\leq \vartheta^{\alpha}_{VE}\leq \vartheta^{N,M,\alpha-\iota}_{VE}+\delta+\eta 
$$
for all
$N\geq N_0$ and $M\geq M_0$. 
\end{theorem}

\subsection{The VaR-AVaR case}

We now discuss the case when 
the outer risk measure and the inner risk measure
in the BCR minimization problem are set with VaR and 
AVaR respectively:
        \begin{equation}
        \label{VCV}
		\begin{aligned}
	\min _{a \in \mathcal{A}} \quad &\text{VaR}_{\mu_{h_{\tau}}}^{\alpha} \circ\text{AVaR}_{P_{\theta}}^{\beta}\left[\mathcal{C}\left(s,a,\xi\right)+\gamma V\left(s',\mu_{h^{\prime}}\right) \right]\\
			\text{s.t.} \quad& s'=g\left(s, a, \xi\right),\  h'={\operatorname{Rep}_{\tau+1}}({h_{\tau}}+H(\xi)).
		\end{aligned}
	\end{equation}
This kind of model  
is used to
minimize the highest risk value with  preference $\text{AVaR}_{P_\theta}^\beta$ that occurs with a probability less than or equal to a predefined threshold $\alpha$ and can also  be interpreted within the framework of  a Bayesian DRO problem. 
Since it
is often difficult to derive a closed-form expression for $\text{AVaR}_{P_{\theta}}^{\beta}\left[\mathcal{C}\left(s,a,\xi\right)+\gamma V\left(s',\mu_{h^{\prime}}\right) \right]$, we may consider the sample average approximation of the problem.
However, accurately evaluating SAA requires a large sample size when $\alpha$ and $\beta$ are close to $1$. Meanwhile, directly substituting $\mathbb{E}_{P_{\theta}}$ with $\text{AVaR}_{P_{\theta}}^\beta$ in model \eqref{VCV} results in a large mixed integer program, posing significant computational challenges.
To address these issues, we relax the outer risk measure VaR to its closest convex upper bound AVaR, resulting in a convex program which can be solved efficiently using a wide range of algorithms. Specifically, under this setting, we propose an algorithm to solve the following optimization problem within the BCR-SOC/MDP framework:
		\begin{equation}\label{CVCV}
	\begin{aligned}
		\min _{a \in \mathcal{A}} \text{AVaR}_{\mu_{h_{\tau}}}^{\alpha} \circ\text{AVaR}_{P_{\theta}}^{\beta}\left[\mathcal{C}\left(s,a,\xi\right)+\gamma V\left(s',\mu_{h^{\prime}}\right) \right]\\
		s'=g\left(s, a, \xi\right),\ h'={\operatorname{Rep}_{\tau+1}}({h_{\tau}}+H(\xi)).
	\end{aligned}
\end{equation}
By exploiting Rockafellar-Uryasev's reformulation of 
$\mathrm{AVaR}$ in \cite{rockafellar2000optimization} as an optimization problem, 
we can rewrite 
problem (\ref{CVCV}) 
as
	\begin{equation}
		\min _{\substack{a \in \mathcal{A} \\ u \in \mathbb{R}}}\left\{u+\frac{1}{\alpha} \mathbb{E}_{\mu_{h_{\tau}}}\left[\text{AVaR}^{\beta}_{P_{\theta}}\left(\mathcal{C}\left(s,a,\xi\right)+\gamma V\left(s',\mu_{h^{\prime}}\right)\right)-u\right]^{+} d \theta\right\}
		\label{CVCV2}\tag{$\text{P}_{\text{CC}}^{\alpha,\beta}$}
	\end{equation}
and applying SAA to the latter
	\begin{equation}
\mathbb{E}_{\mu_{h_{\tau}}}\left[(\text{AVaR}^{\beta}_{P_{\theta}}\left(\mathcal{C}\left(s,a,\xi\right)+\gamma V\left(s',\mu_{h^{\prime}}\right)\right)-u^{+})\right]\\
			\approx \frac{1}{N}\sum_{i=1}^N
			(\text{AVaR}^{\beta}_{P_{\theta_i}}\left(\mathcal{C}\left(s,a,\xi\right)+\gamma V\left(s',\mu_{h^{\prime}}\right)\right)-u^{+}),
		\label{CV_sam}
	\end{equation}
we can 
obtain 
sample average approximation of problem (\ref{CVCV2}):
	\begin{equation}
		\begin{aligned}
			\min_{u\in \mathbb{R},a \in \mathcal{A}, v\in \mathbb{R}^N}&u+\frac{1}{\alpha N}e^Tv\\
			\text{s.t. }\quad  &v_i\geq \text{AVaR}^{\beta}_{P_{\theta_i}}\left(\mathcal{C}\left(s,a,\xi\right)+\gamma V\left(s',\mu_{h^{\prime}}\right)\right)-u,\\
			&v_i\geq 0,\ i=1,2,\cdots,N,
		\end{aligned}
		\label{CVCV_sam}\tag{$\text{P}_{\text{CC}}^{N,\alpha,\beta}$}
	\end{equation}
	where $e$ represents the vector with all components being ones.  By regarding 
    problem \eqref{CVCV_sam} as a AVaR-constrained problem, we can use another SAA to approximate the AVaR in the constraints:
	\begin{equation}
		\begin{aligned}
			\min_{u,a,v,w}&u+\frac{1}{\alpha N}e^Tv\\
			\text{s.t. } &v_i\geq w_i+\frac{1}{\beta M}\sum_{j=1}^M\left[\mathcal{C}\left(s,a,\xi^{i,j}\right)+\gamma V\left(s^{i,j},\mu_{h^{i,j}}\right)-w_i\right]^+-u,\\
			&v_i\geq 0,\ i=1,2,\cdots,N,
		\end{aligned}
		\label{CVCV_sam2}\tag{$\text{P}_{\text{CC}}^{N,M,\alpha,\beta}$}
	\end{equation}
   where $w= (w_1, \cdots ,w_N )$ denotes an auxiliary variable and $(\xi^{i,1},\cdots,\xi^{i,M})$ are $M$ i.i.d. samples generated from $P_{\theta_i}$ for $i=1, \cdots, N$. 
Because of the convexity 
    as demonstrated in Proposition \ref{convex}, problem \eqref{CVCV_sam2} 
    is a convex optimization problem with a linear objective function, which can be solved efficiently even for large-scale problems.
    
    Next, we examine the necessary sample sizes $N$ and $M$ for a specified precision of the optimal value. 
    Denote the optimal values of \eqref{CVCV_sam} and \eqref{CVCV_sam2} by $\vartheta^{N,\alpha,\beta}_{CC}$ and $\vartheta^{N,M,\alpha,\beta}_{CC}$, respectively. 
	Regarding $N$, it has been proven in 
    \cite[Corollary 5.19]{shapiro2021lectures} that under certain conditions, the next proposition holds.
	
\begin{proposition}
\label{prop-5}
		Suppose that  for all $s\in\mathcal{S}$, $a\in\mathcal{A}$,   $\mathcal{C}\left(s,a,\xi\right)+\gamma V\left(s',\mu_{h^{\prime}}\right)$ is Lipschtiz continuous with respect to $a $ with a Lipschitz modulus $L$. Let $\epsilon,\eta>0$, and
		\[N_0=\left(\frac{O(1)LD}{\eta}\right)^2\left[n\log\frac{O(1)LD}{\eta}+\log (\frac{1}{\epsilon})\right],\]
		where $O(1)$ is a constant.
		Then  the optimal solution of problem \eqref{CVCV_sam} 
        lies in the $\eta$-neighborhood of that of problem \eqref{CVCV2} with probability at least $1-\epsilon$
        for all $N>N_0$. Moreover, the optimal value of problem \eqref{CVCV_sam} lies in the $L\eta$-neighborhood of that of problem \eqref{CVCV2} with probability at least $1-\epsilon$
        for all $N>N_0$.
\end{proposition}

	As for the selection of $M$, we give the following proposition.
	\begin{proposition}\label{prop-6}
		Suppose: (a) there exists a measurable function $\varphi:\Xi\to\mathbb{R}^+$ such that $$\left|\mathcal{C}\left(s,a_1,\xi\right)+\gamma V\left(s_1',\mu_{h^{\prime}}\right)-\mathcal{C}\left(s,a_2,\xi\right)+\gamma V\left(s_2',\mu_{h^{\prime}}\right)\right|\leq\varphi(\xi)\left\|a_1-a_2\right\|,$$ for all $s\in\mathcal{S}$, $a_1,a_2\in\mathcal{A}$ and $\xi\in\Xi$;
        (b)
        the moment generating function of $\varphi(\xi)$ is finite in a neighborhood of 0. Let $\delta>0$, $\varepsilon\in(0,1)$, \[M_0=\frac{8\varsigma^2}{\delta^2}\left[\log\left(1+\frac{D^n2\bar{\mathcal{C}}}{\upsilon^{n+1}(1-\gamma)}\right)+\log(\frac{1}{\varepsilon})\right],
        \]
		where $\varsigma$ and $\upsilon$  are constants associated with the given objective function $\mathcal{C}\left(s,a,\xi\right)+\gamma V\left(s',\mu_{h^{\prime}}\right)$ and distribution $P_{\theta}$. 
        Then 
    with probability at least $1-2\varepsilon$,
		\begin{equation}\label{eqsam-gua4}
			\vartheta^{N,\alpha,\beta}_{CC}-\delta\leq \vartheta^{N,M,\alpha,\beta}_{CC}\leq \vartheta^{N,\alpha,\beta}_{CC}+\delta\end{equation}
        for all
        $M\geq M_0$. 
\end{proposition}
        
\noindent \textbf{Proof.}
		Proposition 3 in \cite{wang2008sample} 
        states that with probability at least $1-\varepsilon$, a feasible solution 
        to problem (\ref{CVCV_sam2}) is also feasible 
        for the problem: 
\begin{equation}
	\begin{aligned}
		\min_{u\in \mathbb{R},a \in \mathcal{A}, v\in \mathbb{R}^N}&u+\frac{1}{\alpha N}e^Tv\\
		\text{s.t. }\quad  &v_i\geq \text{AVaR}^{\beta}_{P_{\theta_i}}\left(\mathcal{C}\left(s,a,\xi\right)+\gamma V\left(s',\mu_{h^{\prime}}\right)\right)-u-\delta,\\
		&v_i\geq 0,\ i=1,2,\cdots,N
	\end{aligned}
	\label{CVCV_sam_es}
\end{equation}
for all
$M\geq M_0$. 
Note that 
the optimal value of problem \eqref{CVCV_sam_es} equals $\vartheta^{N,\alpha,\beta}_{CC}-\delta$. Thus,  
if a feasible solution to
problem (\ref{CVCV_sam2}) is feasible for problem \eqref{CVCV_sam_es}, 
then 
$\vartheta^{N,\alpha,\beta}_{CC}-\delta \leq\vartheta^{N,M,\alpha,\beta}_{CC}$.
Likewise
a feasible solution
to the problem 
	\begin{equation}
		\begin{aligned}
		\min_{u\in \mathbb{R},a \in \mathcal{A}, v\in \mathbb{R}^N}&u+\frac{1}{\alpha N}e^Tv\\
		\text{s.t. } &v_i\geq \text{AVaR}^{\beta}_{P_{\theta_i}}\left(\mathcal{C}\left(s,a,\xi\right)+\gamma V\left(s',\mu_{h^{\prime}}\right)\right)-u+\delta,\\
		&v_i\geq 0,\ i=1,2,\cdots,N
	\end{aligned}
	\label{CVCV_sam_es2}
\end{equation}
	is also feasible for problem (\ref{CVCV_sam_es}) with probability  at least $1-\varepsilon$ 
    for all $M\geq M_0$. The optimal value of problem \eqref{CVCV_sam_es2} is $\vartheta^{N,\alpha,\beta}_{CC}+\delta$. 
    Thus, when the feasible solution 
    to problem \eqref{CVCV_sam_es2} is feasible 
    for problem (\ref{CVCV_sam_es}), we have $\vartheta^{N,\alpha,\beta}_{CC}+\delta \geq\vartheta^{N,M,\alpha,\beta}_{CC}$.
\hfill $\Box$

By combining Propositions \ref{prop-5} and \ref{prop-6},
we can generate the second main result of this subsection.
\begin{theorem}
    Assume the settings and 
    conditions of Propositions \ref{prop-5} and \ref{prop-6} are satisfied. Then
with probability at least $1-\epsilon-2\varepsilon$,
    		$$
			\vartheta^{N,M,\alpha,\beta}_{CC}-\delta-L\eta\leq \vartheta^{\alpha,\beta}_{CC}\leq \vartheta^{N,M,\alpha,\beta}_{CC}+\delta+L\eta 
$$
for all
$N\geq N_0$ and $M\geq M_0$. 
\end{theorem}

The sample sizes established 
in Propositions \ref{prop-5} and \ref{prop-6}
could be large.
In practice,
various strategies 
may be employed to accelerate the solution process. For example, \cite{xu2009smooth} 
shows that smoothing the AVaR can improve computational efficiency significantly, potentially speeding up the solution process by several multiples.

\section{Numerical results}

In this section,
we evaluate the performance of 
proposed BCR-SOC/MDP models 
and algorithms 
by applying them to 
a finite-horizon spread betting problem and an infinite-horizon inventory control problem. 
All the experiments 
are carried out 
on a 64-bit PC with 12 GB of RAM and a 3.20 GHz processor. 

We choose the following SOC/MDP models for comparison. 
\begin{itemize}

\item 
BCR-SOC/MDP: 
for the exact dynamic programming method and the value iteration algorithm, Algorithm \ref{alg:A} and Algorithm \ref{alg:B}, we need to {\color{black} 
figure out 
the hyper-parameter space $\mathcal{H}_\tau^\mathrm{rep}$ of the augmented state $\mu_\tau$ by Algorithm \ref{alg:adaptive-grid}.}
We consider two variants in our experiment: the VaR-Expectation BCR-SOC/MDP {\color{black} (which 
corresponds to a specific case in \cite{lin2022bayesian})} and the AVaR-AVaR BCR-SOC/MDP, as discussed in Section 6. {\color{black} Besides,
we also consider the Expectation-Expectation BCR-SOC/MDP and AVaR-Expectation BCR-SOC/MDP}.

\item Risk-averse SOC/MDP (RA-SOC/MDP) presented in \cite{carpin2016risk}: we derive a maximal likelihood estimation (MLE) $\hat{\theta}$ from the observations of $\xi$ and solve the corresponding risk-averse SOC/MDP using the estimated distribution $P_{\hat{\theta}}$. For comparability with the risk-averse Bayes-adaptive MDP  (AVaR-AVaR case), we adopt the AVaR as the risk measure.

\item Standard SOC/MDP: for comparability with the risk-averse Bayes-adaptive MDP  (VaR-Expectation case), we adopt the expectation operator as the objective function. 
Likewise, we use the MLE estimator to estimate the true distribution's parameter.

\item Bayes-adaptive MDP: here, we focus on the episodic Bayesian approach presented in \cite{shapiro2023episodic}, which {\color{black} does not contain 
the belief space and corresponds to a suboptimal variant of our model, reformulated as the Expectation–Expectation BCR‐SOC/MDP variant illustrated in Example~\ref{bcr-2}.}

\item Distributionally robust SOC/MDP (DR-SOC/MDP) presented in \cite{xu2010distributionally,lin2022bayesian}: since prior probabilistic information to construct the ambiguity set might not be available from the dataset directly, we simulate samples of $\theta$ from the posterior distribution to construct the ambiguity set, and derive the optimal policy that minimizes the expected cost under the worst-case scenario of $\theta$.
\end{itemize}

Each method will be used to determine the optimal policy from the same sample set, and the performance of the resulting policy will then be tested on a real system,
specifically, a SOC/MDP with the true parameter $\theta^c$.
Below are 
detailed 
descriptions of the problems and the tests.

\subsection{Finite horizon spread betting problem}

In this subsection, 
we consider a spread 
betting problem
varied from Birge et al.~\cite{birge2021dynamic}.
In spread betting markets, market makers (e.g., bookmakers) quote a spread line for the outcome of an uncertain event, and participants (bettors) 
bet on 
whether the outcome will exceed or fall 
below the spread line 
based on their own judgment and experience. 
A typical example is in sports betting, where the bookmaker sets a point spread for a specific match or game. Bettors then 
bet on 
whether the outcome will be greater 
(favorite wins by more than the spread) 
or smaller (favorite does not win by more than the spread) than the set spread line.
Birge et al.~\cite{birge2021dynamic} consider the market maker's
decision-making process on setting spread line
such that the total expected profit is maximized
over a specified time horizon.

Here, we focus on a risk-averse bettor's 
optimal decision-making problem where the bettor
aims to minimize the long-term risk
of random losses 
by dynamically  adjusting the bet size 
and learning about the 
market's movements over time. 
To model this dynamic decision-making process, we apply the BCR-SOC/MDP model to the informed bettor’s decision problem in \cite{birge2021dynamic}. The setup is specified as follows.
\begin{itemize}

\item [(i)] \textbf{Market movement $\xi_t$}.
The spread market is a betting market on a specific event, where the 
outcome of the event 
is represented by a continuous random variable $X_t: \Omega 
\to \mathbb{R}$ 
with cdf $F_{X_t}$ at episode $t$. 
The bookmaker sets a spread line $l_t$ 
for each bet at episode $t$. 
The corresponding spread market movement $\xi_t$ is a binary random variable, 
where: $\xi_t = 1$ 
represents 
the event that 
the outcome of $X_t$ exceeds
$l_t$ 
and $\xi_t = -1$ otherwise.
In practice,
the probability $P(\xi_t = 1)=1-F_{X_t}(l_t)$ 
is unknown to bettors 
and must be estimated over time.  
For simplicity, we assume that the market maker maintains a constant probability for the event $\xi_t = 1$, denoted as $\theta^c$, for all bets. 
This 
may be understood as 
the market maker striving to control the spread line $l_t$ such that
the probability of the event exceeding the spread remains constant over time.  
The assumption
means that the market maker would maintain a 
fixed probability for market movements by adjusting $l_t$ according to the distribution of $X_t$. In the case that $\theta^c=0.5$,
the market maker 
remains indifferent to the outcome of the bets in each run, 
avoiding significant exposure to either side of the bet.

\item [(ii)] \textbf{Wealth $s_t$}.
The bettor’s cumulative wealth at episode $t$, representing the bettor's account balance. Wealth is updated at each episode based on the trades and the spread market scenarios.

\item[(iii)] \textbf{Belief $\mu_t(\theta)$}.
The bettor’s belief (or estimate) of the spread market activity level $\theta^c$. The belief about spread market movements is updated based on the spread market feedback after the trade observation, reflecting the evolving understanding of the spread market state over time.  A typical prior $\mu_1$ for $\theta$ here is the Beta prior since the posterior $\mu_t$ is also a Beta distribution (see \cite{gelman2013bayesian}). {\color{black} Let $h_t=(m_t,d_t)$ denote the hyper-parameter of Beta distribution and we choose $h_1=(1,1)$ as the uniform prior. Under this setting, the Bayesian update of $\mu_t$ can be reformulated as the update of $h_t$ by $m_{t+1}=m_t+\mathds{1}_{\{\xi_t=1\}}$ and $d_{t+1}=d_t+\mathds{1}_{\{\xi_t=0\}}$. Here, $\mathds{1}_{\{\xi_t = 0\}}$ is an indicator function that takes 1 if 
$\xi_t = 0$ and 0 otherwise.}

\item[(iv)] \textbf{Position Size $a_t$}. 
At each episode $t$, the bettor chooses the position size $a_t$, which represents the amount of money to bet. The bettor can bet on the event $\{X_t > l_t\}$  (i.e., the outcome exceeds the spread line). The position size is determined by the bettor's policy $\pi_t(s_t,\mu_t)$, which depends on the bettor's current wealth $s_t$  and belief about the market state $\mu_t(\theta)$. In this model, the bettor cannot bet more than the bettor's available wealth, so the stake must satisfy the constraint $0 \leq a_t \leq s_t$, ensuring that the bettor does not wager more than they have.

\item [(v)] \textbf{Loss function ${\cal C}_t$ and state transition function $g_t$}.  
The loss at episode $t$ is defined as the negative profit from the trade. Specifically, the loss function ${\cal C}_t$ is defined as:
 $${\cal C}_t(s_t,a_t,\xi_t)=-(1-\tau\mathds{1}_{\{\xi_t=1\}})a_t\xi_t,
 $$
where $\tau$ denotes the commission rate (a value between 0 and 1) charged by the market maker. The 
net profit for the bettor is calculated as follows. 
If the bettor wins (i.e., $a_t \xi_t > 0$), the effective profit is reduced by the commission: $(1 - \tau) a_t \xi_t$.
If the bettor loses (i.e., $a_t \xi_t < 0$), then the 
net loss is $-a_t \xi_t$, and no commission is applied.
The cumulative 
wealth of the bettor at episode $t+1$, $s_{t+1}$, can then be determined as
$s_{t+1}=g(s_t,a_t,\xi_t):=s_t+(1-\tau\mathds{1}_{\{\xi_t=1\}})a_t\xi_t$.
\end{itemize}

The BCR-SOC/MDP formulation of
the above spread betting problem
can be formulated as:
	\begin{subequations}
	\begin{align*}
		\min _{\pi} \ &\rho_{\mu_{m_1,d_1}}\circ\rho_{P_{\theta_{1}}}\left[-(1-\tau\mathds{1}_{\{\xi_1=1\}})a_1\xi_1+\cdots+\rho_{ \mu_{m_{T-1},d_{T-1}}}\circ \rho_{P_{\theta_{T-1}}}\left[-(1-\tau\mathds{1}_{\{\xi_{T-1}=1\}})a_{T-1}\xi_{T-1}\right]\right] \\
		\text { s.t. } & s_{t+1}=s_t+(1-\tau\mathds{1}_{\{\xi_1=1\}})a_t\xi_t,\ a_t=\pi_t(s_t,\mu_{m_t,d_t}), \ 0\leq a_t\leq s_t,\ t=1, \cdots, T-1,\\
        &m_{t+1}=m_t+\mathds{1}_{\{\xi_t=1\}}, \ d_{t+1}=d_t+\mathds{1}_{\{\xi_t=0\}},\ t=1, \cdots, T-1.
	\end{align*}
\end{subequations}
The well-definedness of this problem can be verified using the expressions of loss function, state transition function as well as the boundedness of $\xi_t$.
In the numerical tests, 
the initial wealth is set with $s_1 = 80$. The true parameter is defined as $\theta^c = 0.6$, with the support set of the parameter being $\Theta = (0,1)$. The bettor engages in a series of $T = 8$ episodes, selecting the betting size $a_t$ from the action space $\mathcal{A} = \{0,2,4,6,8,10\}$ at each episode. We set $\alpha = 0.6$ and $\beta = 0.8$ for the outer and inner risk parameters, respectively. The commission rate $\tau=0.05$. The dataset for this problem comprises historical market movement records of size $N$. To reduce the effect of random sampling, we repeat each test with 100 simulations, each of which employs an independent dataset. Table~\ref{Table-BCR-SOC/MDP-1} displays the average CPU time (in seconds) required to determine the optimal betting policy, as well as the mean and variance of the actual performance of the optimal policy across the 100 simulations with different sample sizes $N = 5, 10, 50,150$.
Figure \ref{fig:3} visualizes the histogram of actual performances for the BCR-SOC/MDP with $\text{AVaR}_{\mu}^{\alpha}$-$\text{AVaR}_{P_\theta}^{\beta}$, incorporating different combinations of risk levels: $\alpha = 0.01, 0.5, 0.9$ and $\beta = 0.01, 0.5, 0.9$.

\begin{table}[ht]
	\centering
	\caption{CPU time, mean and variance of actual performances of different approaches for the spread betting problem.}
    \scalebox{0.75}{
	\begin{tabular}{lcccccccccc}
		\hline
		{Approach} & {time(s)} & \multicolumn{2}{c}{$N=5$} & \multicolumn{2}{c}{$N=10$} &\multicolumn{2}{c}{$N=50$} &\multicolumn{2}{c}{$N=150$}\\
		& & {mean} & {variance} & {mean} & {variance} & {mean} & {variance} & {mean} & {variance} \\
		\hline
   Bayes-adaptive MDP ($\mathbb{E}_{\mu}$-$\mathbb{E}_{P_{\theta}}$) \cite{shapiro2023episodic} & 1.14 & {-10.43}  & {58.19}  & {-12.23} & {46.95} &{-13.10}& {12.39} &{-14.92}& {0.87}\\
    
    BCR-SOC/MDP ($\mathbb{E}_{\mu}$-$\mathbb{E}_{P_{\theta}}$) & 4.11 & \textbf{-12.33}  & \textbf{16.79}  & \textbf{-13.31} & \textbf{14.24} &\textbf{-15.37}& \textbf{2.89} &\textbf{ -15.60}& \textbf{0.82}\\
    
		Standard SOC/MDP & 1.08 & -10.85  & 60.53 & -12.55 & 47.63  & -12.78 & 14.52 & -15.37 & 0.96\\
        
		BCR-SOC/MDP ($\text{VaR}_{\mu}^{\alpha}$-$\mathbb{E}_{P_{\theta}}$) & 5.46 & \textbf{-12.12}  &\textbf{14.52}  & \textbf{-14.86} & \textbf{11.02} & \textbf{-15.44} &\textbf{2.05}& \textbf{-15.55} &\textbf{0.64}\\
        
        BR-MDP ($\text{VaR}_{\mu}^{\alpha}$-$\mathbb{E}_{P_{\theta}}$) \cite{lin2022bayesian} & 834.66 & -11.88  & 16.41  & -13.40 & 11.56 & -15.19 & 2.38 & -15.54 & 0.77
        \\
        
        BCR-SOC/MDP ($\text{AVaR}_{\mu}^{\alpha}$-$\mathbb{E}_{P_{\theta}}$) & 8.30 & \textbf{-11.43}  &\textbf{13.04}  & \textbf{-13.66} & \textbf{10.08} & \textbf{-15.26} &\textbf{1.21}& \textbf{-15.89} &\textbf{0.63}\\
        
        BR-MDP ($\text{AVaR}_{\mu}^{\alpha}$-$\mathbb{E}_{P_{\theta}}$) \cite{lin2022bayesian} & 13.25 & -9.31  & 14.63  & -11.85 & 12.47 & -14.17 & 1.75 & -15.07 & 0.78
        \\
        
		RA-SOC/MDP  & 1.15 & -8.29 & 50.07 & -10.76 & 45.97& -12.02 & 13.57& -14.36 & 0.88 \\
        
		BCR-SOC/MDP ($\text{AVaR}_{\mu}^{\alpha}$-$\text{AVaR}_{P_\theta}^{\beta}$) &8.37 & \textbf{-9.09}  & \textbf{12.39} & \textbf{-11.63}	 & \textbf{9.74}	& \textbf{-14.79} &	\textbf{1.18}& \textbf{-15.39} &	\textbf{0.67} \\
        
		DR-SOC/MDP & 1.84 & 0.00 & 0.00 & 0.00 & 0.00  & 0.00 & 0.00& 0.00 & 0.00\\
		\hline
	\end{tabular}}\label{Table-BCR-SOC/MDP-1}
\end{table}

\begin{figure}[ht]
	\centering 
	\includegraphics[width=1\textwidth]{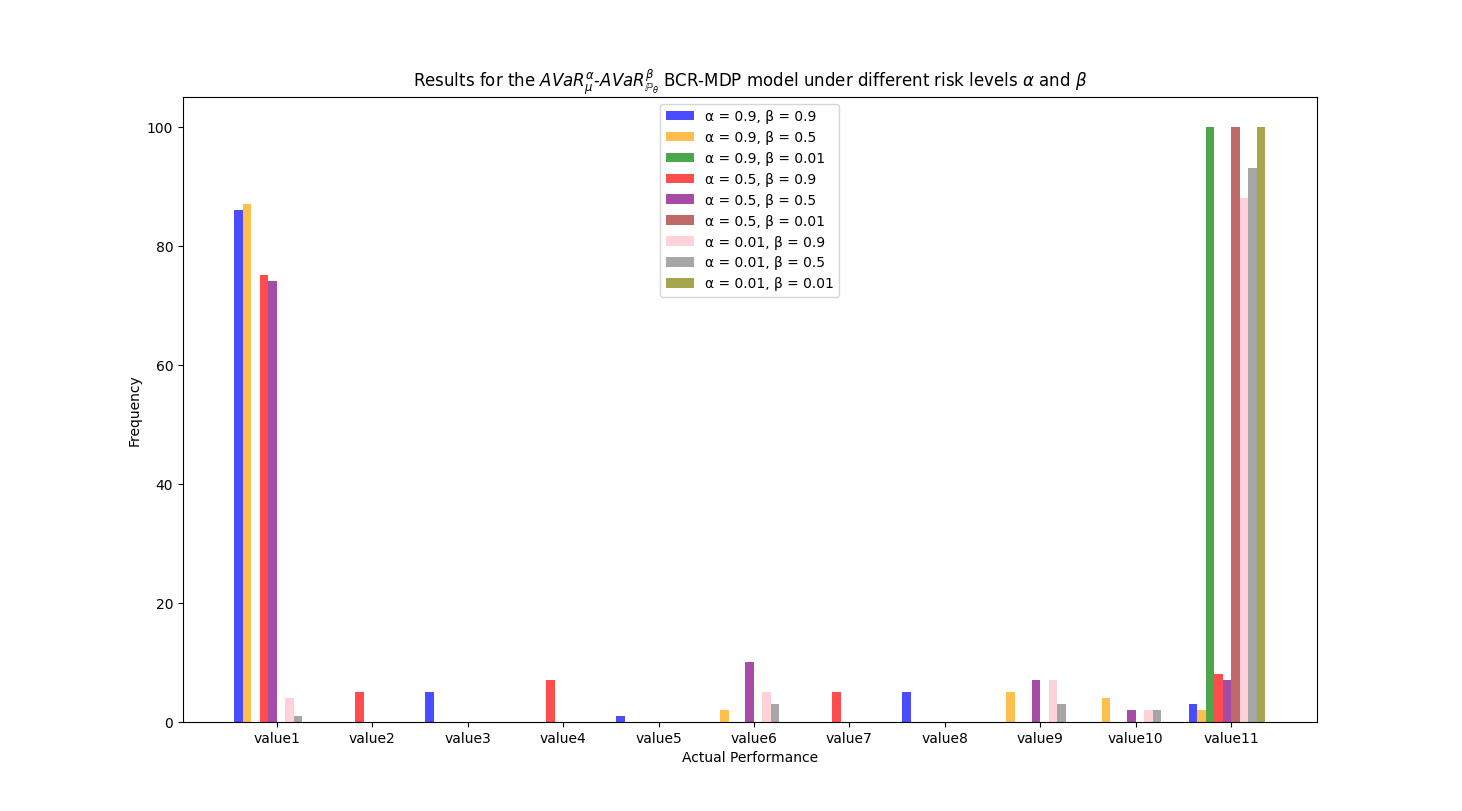}
	\includegraphics[width=1\textwidth]{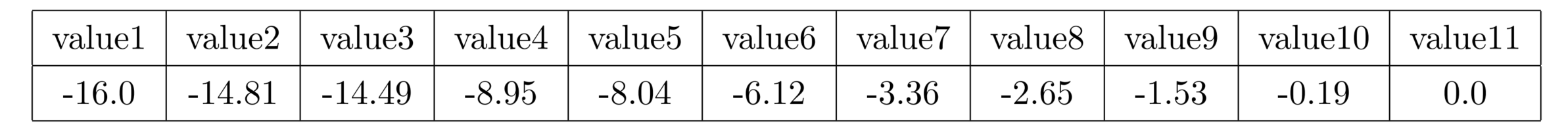} 
	\caption{Histogram of performance frequencies  over 100 replications for  the $\text{AVaR}_{\mu}^{\alpha}$-$\text{AVaR}^\beta_{P_{\theta}}$ BCR-SOC/MDP with different $\alpha$ and $\beta$.} 
	\label{fig:3} 
\end{figure}

The observations presented in Table~\ref{Table-BCR-SOC/MDP-1} and Figure~\ref{fig:3} lead to the following conclusions.

\begin{itemize}

\item Compared to the standard SOC/MDP and RA-SOC/MDP, 
the 
BCR-SOC/MDP model 
exhibits 
better 
robustness 
in terms of the mean and variance
across 
some specific BCR measures 
including
$\mathbb{E}_{\mu}$-$\mathbb{E}_{P_{\theta}}$, $\text{VaR}_{\mu}^{\alpha}$-$\mathbb{E}_{P_{\theta}}$ and $\text{AVaR}_{\mu}^{\alpha}$-$\text{AVaR}_{P_\theta}^{\beta}$.
{\color{black}
Since the episodic Bayesian approach  \cite{shapiro2023episodic} can be treated as an MLE estimation of $\theta^c$ with a prior, which has been illustrated in Example \ref{bcr-3}, it performs similarly to the standard SOC/MDP, 
highlighting its 
suboptimality in absence of
belief-dependent adjustments.}
Unlike these methods that depend solely on the point estimate $\hat{\theta}$, our approach utilizes the update of posterior $\mu$ to capture the uncertainty in $\theta$ effectively, significantly enhancing the stability of the optimal policy for both small and large sample sizes. 
Moreover, the DR-SOC/MDP tends to be over conservative, as it always considers the worst-case scenario from the set of all possible market uncertainties, resulting in zero variance in actual performance. 
In particular, even in situations where the probability of winning at $\theta^c$ is high, the DM under DR-SOC/MDP model refrains from betting because of its conservative nature.

\item As the sample size increases, the uncertainty corresponding to the model parameter diminishes. Consequently, the 
performances of all SOC/MDP 
models improve consistently, with variance monotonically
decreasing.
The tendency is underpinned by 
the convergence of both the posterior distribution $\mu$ and the MLE estimator $\hat{\theta}$ to the true parameter $\theta^c$ as the sample size increases. Eventually, every model converges to the SOC/MDP model with respect to the true parameter $\theta^c$ in terms of 
optimal values and 
optimal policies.

\item In 
the BR-MDP 
model~\cite{lin2022bayesian}, 
the authors 
use the 
pair of augmented state 
including 
the original physical state and 
the posterior distribution  
to obtain the 
value function ($\mathcal{J}^{\gamma,*}_T$ in their paper)
and 
the corresponding optimal policy. 
{\color{black}The size of the space of 
discretized state ($\mathcal{D}_1^p$) 
is excessively large ($10^{6}$),
which may result 
in prohibitively 
expensive 
computational cost in terms of CPU time for the proposed BR-MDP model ($\text{VaR}_{\mu}^{\alpha}$-$\mathbb{E}_{P_{\theta}}$). 
To address the issue, the authors
propose an
algorithm
to solve the BR-MDP ($\text{AVaR}_{\mu}^{\alpha}$-$\mathbb{E}_{P_{\theta}}$) which is an approximation of the BR-MDP model ($\text{VaR}_{\mu}^{\alpha}$-$\mathbb{E}_{P_{\theta}}$).
The resulting CPU time is significantly
lower than the one required for the 
discretized BR-MDP  ($\text{VaR}_{\mu}^{\alpha}$-$\mathbb{E}_{P_{\theta}}$) model.
Here, we use the hyper-parameter approach as detailed in Section 6
to solve the problem.  The CPU time (8.30 seconds) is comparable to theirs (13.25 seconds).  
The underlying reason is that in our hyper-parameter approach, the size of augmented state space $\mathcal{H}_t$ at stage $t$ is $t$, 
which 
is significantly 
lower than the size of 
the discretized 
augmented state space ($10^6$)
in \cite{lin2022bayesian} and 
subsequently the computational time.}


\item 
In Figure~\ref{fig:3}, the risk levels $\alpha$ and $\beta$ reflect the bettor's risk preferences towards the epistemic and aleatoric uncertainties, respectively. With lower values of $\alpha$ and $\beta$, the DM tends to adopt a more conservative policy, typically choosing not to bet. Consequently, this conservative stance causes the distribution of actual performance to shift toward the right. However, the risk preferences characterized by $\alpha$ and $\beta$ diverge in specific aspects.
A lower $\alpha$ indicates a bettor's increased concern over losses in extreme environment, signifying a preference for avoiding such outcomes. As $\alpha$ decreases, the optimal decisions under BCR-SOC/MDP model closely resemble those of the DR-SOC/MDP approach, resulting in more conservative choices. 
Conversely, a lower $\beta$ directly signifies a bettor's aversion to losses, demonstrating prudence and a preference for minimal potential losses. Thus, as $\beta$ decreases, the bettor focuses more on the losses with respect to worst-case $\xi$ irrespective of the environmental parameter $\theta$, steering the BCR-SOC/MDP model towards the robust SOC/MDP model. This invariably leads to more cautious decisions. Moreover, compared to the risk level $\alpha$, $\beta$ exerts a stronger influence on the conservativeness of the outcomes. This is because, while the outer conservatism ensures that losses maintain a degree of randomness even in the worst environment, the inner conservatism compels the bettor to directly consider the most adverse scenario.
 \end{itemize}

\subsection{Infinite horizon inventory control problem}
Having demonstrated the reasonability and efficiency of the proposed BCR-SOC/MDP model compared with typical SOC/MDP models in the literature, in this subsection, we mainly apply the proposed infinite horizon BCR-SOC/MDP model to the inventory control problem with a discount factor $\gamma=0.8$. 

Inventory models have numerous real-world applications, including healthcare systems, electric energy storage systems, and more (see \cite{saha2019modelling, weitzel2018energy} and the references therein). In inventory management, the problem typically involves deciding the order quantity required in each episode to fulfill customer demand while minimizing costs. The goal of the inventory manager is to manage the trade-offs between ordering costs, holding costs, and penalty costs associated with stockouts, while dynamically adjusting orders based on inventory levels and the uncertainty of future demand. A typical risk-averse inventory control problem aims to minimize long-term costs by dynamically adjusting order quantities and updating beliefs about future demand.

To model this dynamic decision-making process using the BCR-SOC/MDP framework, we introduce the following parameters and assumptions:

\begin{itemize}
    \item [(i)] \textbf{Customer demand $\xi_t$}.
The random variable $\xi_t$ represents the demand for goods at episode $t$. We assume that the demand $\xi_t$ follows a Poisson distribution with an unknown parameter $\theta^c$ \cite{shapiro2021lectures}. 

    \item [(ii)] \textbf{Inventory level $s_t$}.
The inventory level at episode $t$, denoted by $s_t$, represents the quantity of goods available in stock at that episode. Inventory is updated at each episode based on the decision made by the inventory manager (the order quantity) and the actual demand observed during that episode.

   \item [(iii)] \textbf{Belief $\mu_t(\theta)$}.
The inventory manager’s belief about the demand distribution at episode $t$, denoted by $\mu_t(\theta)$, represents the manager's knowledge of the likelihood of various demand amounts.  A typical prior $\mu_1$ for $\theta$ here is the Gamma prior since the posterior $\mu_t$ is also a Gamma distribution (see also \cite{gelman2013bayesian}). {\color{black} Here, \(h_t=(m_t,d_t)\) denotes the hyper-parameter pair of the Gamma distribution, and we choose \(h_1=(m_1,d_1)=(1,1)\) as a noninformative prior.  Under the usual Poisson demand model, observing a single demand \(\xi_t\) updates the hyper-parameters via
$m_{t+1}=m_t + \xi_t$, 
$d_{t+1}=d_t + 1.$}

\item [(iv)]\textbf{ Order quantity $a_t$}.
The order quantity at episode $t$, denoted by $a_t$, represents the amount of goods the inventory manager decides to order. 
This decision $a_t$ is determined by the manager's policy $\pi_t(s_t,\mu_t)$, which depends on the current inventory level $s_t$ and the inventory manager's belief $\mu_t(\theta)$ about the demand distribution. 

\item [(v)] \textbf{Cost function ${\cal C}$ and state transition function $g$}.  
The cost ${\cal C}$ at episode $t$ is defined as 
$$\mathcal{C}(s_t,a_t,\xi_t)=oa_{t}+c(s_{t}+a_{t},\xi_{t}),$$
where
\[c(y,z):=p\max\{z-y,0\}+h\max\{z-d,0\},\]
$o,p$ 
and $h$ are 
unit 
order cost, penalty cost 
for unfulfilling 
a unit demand,
and unit holding cost for excess inventory, respectively.
The inventory level $s_{t+1}$ can then be defined as $s_{t+1}=g(s_t,a_t,\xi_t):=s_t+a_t-\xi_t$.
\end{itemize}

Under the BCR-SOC/MDP framework, the infinite horizon inventory control problem can be formulated as: 
	\begin{subequations}
	\begin{align}\label{numer-2}
		\min_{\pi} \quad &\rho_{\mu_{m_1,d_1}}\circ\rho_{P_{\theta_{1}}}\left[oa_1+c(s_1+a_1,\xi_1)+\cdots+\gamma\rho_{ \mu_{m_t,d_t}} \circ\rho_{P_{\theta_{t}}}\left[oa_{t}+c(s_{t}+a_{t},\xi_{t})+\cdots\right]\right] \\
		\text { s.t. } \quad & s_{t+1}=s_t+a_t-\xi_t,\ a_t=\pi(s_t,\mu_{m_t,d_t}),\ a_t\geq0,\ t=1, 2,\cdots,\\
        &m_{t+1}=m_t+\xi_t, \ d_{t+1}=d_t+1,\ t=1, \cdots.
	\end{align}
\end{subequations}

The classical stationary inventory control problem corresponding to the true distribution $P_{\theta^c}$ can be described in BCR-SOC/MDP form as follows:
	\begin{equation}
	\begin{aligned}\label{true-inventory}
		\min_{\pi} \quad &\rho_{P_{\theta^c}}\left[oa_1+c(s_1+a_1,\xi_1)+\cdots+\gamma\rho_{P_{\theta^c}}\left[oa_{t}+c(s_{t}+a_{t},\xi_{t})+\cdots\right]\right] \\
		\text {s.t. } \quad  &s_{t+1}=s_t+a_t-\xi_t,\ a_t=\pi(s_t,\delta_{\theta^c}),\ a_t\geq0,\  t=1, 2,\cdots
	\end{aligned}
\end{equation}
As demonstrated in \cite{shapiro2023episodic}, the optimal policy for problem (\ref{true-inventory}) is the base-stock policy $\pi^*(s,\delta^c)=\max\{s^*-s,0\}$, where $s^*=H^{-1}(\kappa)$ with $H(\cdot)$ being the cdf of the random demand and $\kappa:=\frac{p-(1-\gamma)o}{p+h}$. Thus, for $s\leq s^*$, the optimal value of (\ref{true-inventory}) is given by:
\begin{equation}\label{optimalvalue}
	V^*(s,\delta_{\theta^c})=-os+os^*+(1-\gamma)^{-1}\rho_{P_{\theta^c}}[\gamma o\xi+c(s^*,\xi)].
\end{equation}

With this analytical expression, we can compute the true optimal value functions with $\rho_{P_{\theta^c}}=\mathbb{E}_{P_{\theta^c}}$ and $\rho_{P_{\theta^c}}=\text{AVaR}_{P_{\theta^c}}^{\beta}$, respectively, as benchmarks for standard SOC/MDP and risk-averse SOC/MDP to demonstrate the convergence of the BCR-SOC/MDP models  corresponding to $\text{VaR}_{\mu}^{\alpha}$-$\mathbb{E}_{P_{\theta}}$ and $\text{AVaR}_{\mu}^{\alpha}$-$\text{AVaR}_{P_\theta}^{\beta}$ with respect to the value functions.

In this inventory control problem, 
we set 
$\Theta=\mathbb{R}_+$. The unit cost for ordering
is 
$o=2$, the unit 
holding 
cost is $h = 4$, and the penalty for unmet demand is set at $p = 6$. 
For the BCR-SOC/MDP model, we set the outer risk parameter $\alpha = 0.6$ and the inner risk parameter $\beta=0.4$.
{\color{black} 
Since the support set $\Xi$ is $\mathbb{N}^+$, we first employ Algorithm \ref{alg:adaptive-grid} to construct a finite hyper-parameter approximation space. To demonstrate the efficiency of the proposed method, we analyze the histogram of projection errors $\delta_I=\mathrm{dist}\left(h_{1} + \sum_{j=1}^{I-1}H(\xi_j),\mathcal{H}_{I}^{\mathrm{rep}}\right)$ across Monte Carlo simulations.
The parameters are set with $I=10$, $R = 20$ and $\theta^c = 10$ for this experiment.
Figure \ref{fig:eps_M_grid} depict the empirical density of $\delta_I$ where the blue histogram represents the distribution of $\delta_I$ across all simulations. The red dashed lines signify the theoretical error bounds for $\delta_I$ (see \eqref{eq-th-error-bound} for definition) under ${P}_{\theta^c}^{\otimes(I-1)}(\|\xi\|\leq R)^{I-1}=0.98$ for $\epsilon \in \{1, 1.5,2\}$  (here $\epsilon \geq 1$ is chosen since $\xi$ takes integer values) and $M_{\max}=\{10,20,100\}$. 
\begin{figure}[ht]
  \centering
  \begin{subfigure}[b]{0.3\textwidth}
    \centering
    \includegraphics[width=\textwidth]{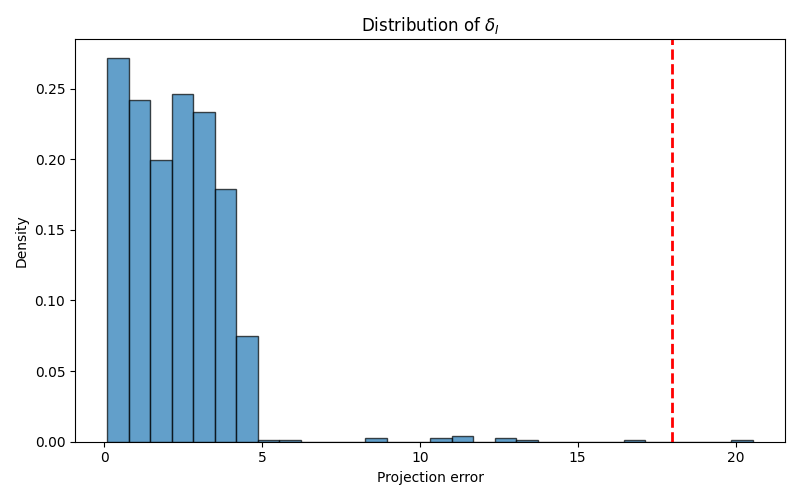}
    \caption{$\epsilon=1,\ M_{\max}=10$}
  \end{subfigure}%
  \hfill
  \begin{subfigure}[b]{0.3\textwidth}
    \centering
    \includegraphics[width=\textwidth]{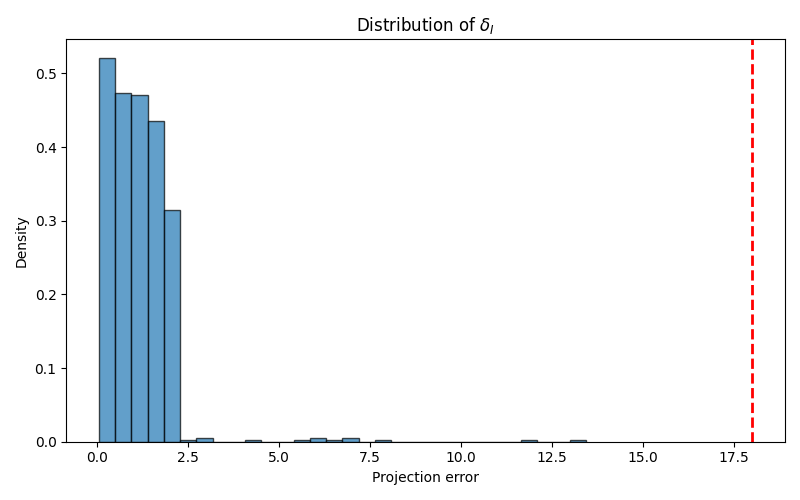}
    \caption{$\epsilon=1,\ M_{\max}=20$}
  \end{subfigure}%
  \hfill
  \begin{subfigure}[b]{0.3\textwidth}
    \centering
    \includegraphics[width=\textwidth]{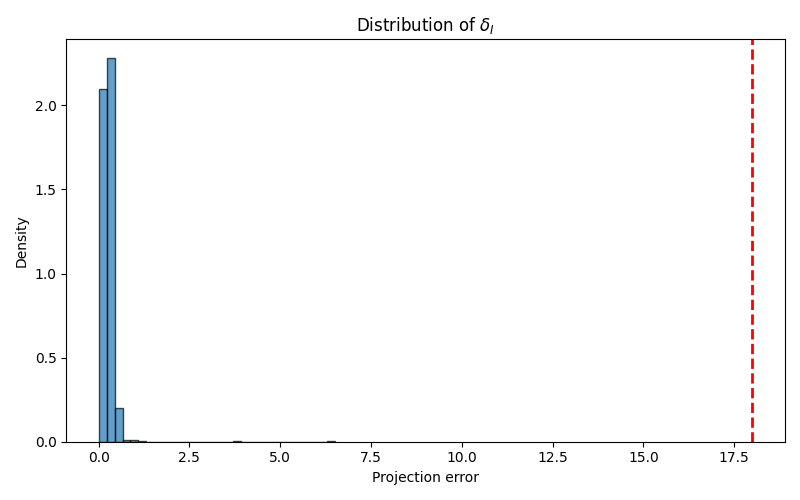}
    \caption{$\epsilon=1,\ M_{\max}=100$}
  \end{subfigure}
  \begin{subfigure}[b]{0.3\textwidth}
    \centering
    \includegraphics[width=\textwidth]{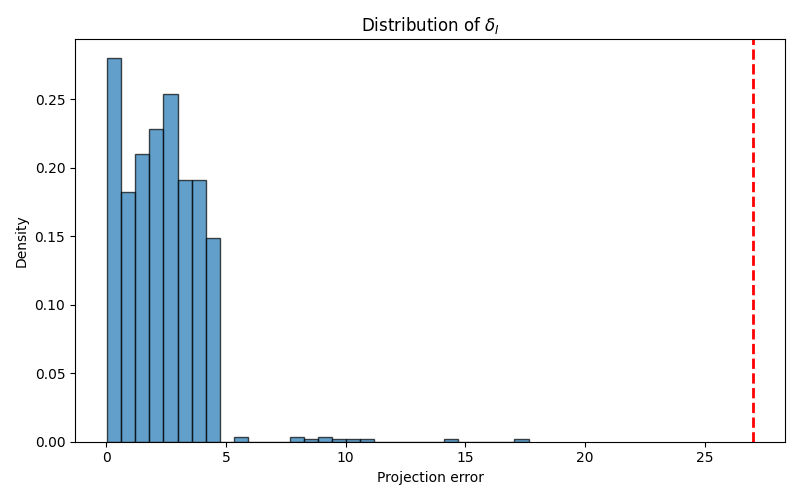}
    \caption{$\epsilon=1.5,\ M_{\max}=10$}
  \end{subfigure}%
  \hfill
  \begin{subfigure}[b]{0.3\textwidth}
    \centering
    \includegraphics[width=\textwidth]{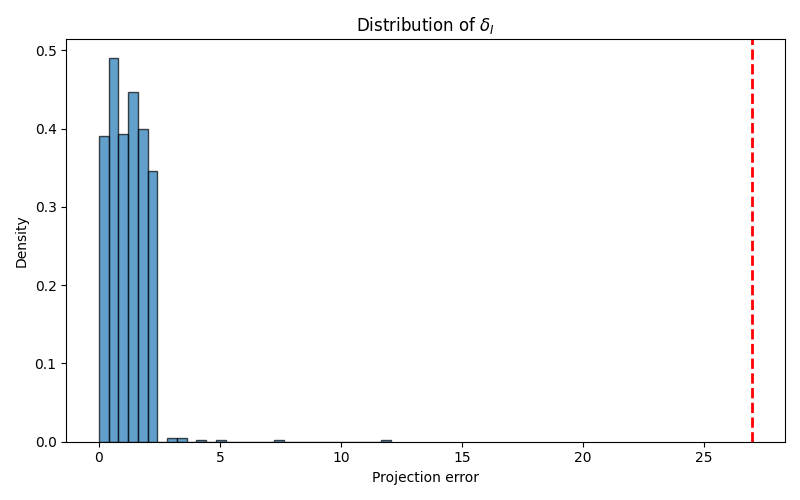}
    \caption{$\epsilon=1.5,\ M_{\max}=20$}
  \end{subfigure}%
  \hfill
  \begin{subfigure}[b]{0.3\textwidth}
    \centering
    \includegraphics[width=\textwidth]{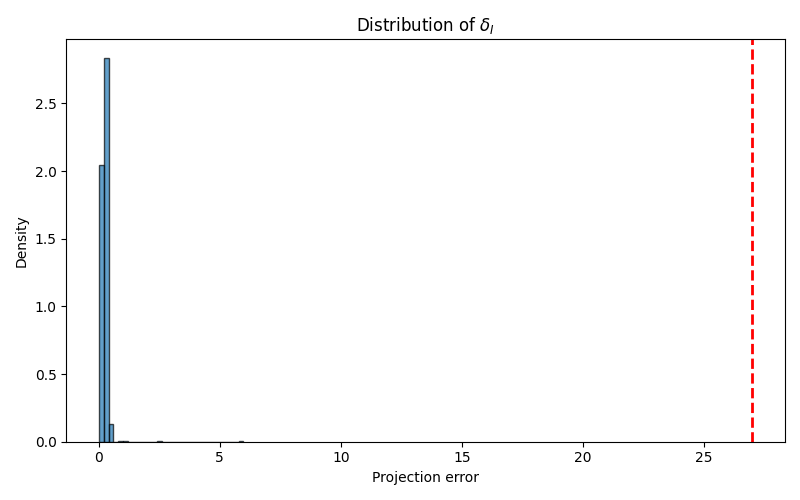}
    \caption{$\epsilon=1.5,\ M_{\max}=100$}
  \end{subfigure}
   \begin{subfigure}[b]{0.3\textwidth}
    \centering
    \includegraphics[width=\textwidth]{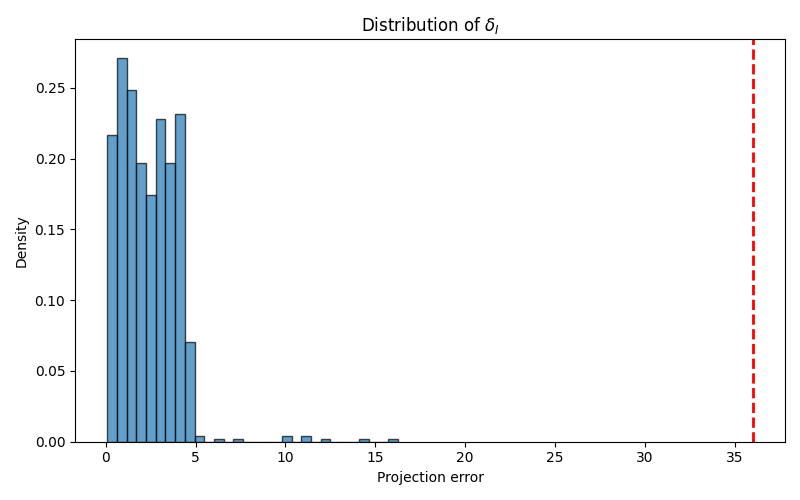}
    \caption{$\epsilon=2,\ M_{\max}=10$}
  \end{subfigure}%
  \hfill
  \begin{subfigure}[b]{0.3\textwidth}
    \centering
    \includegraphics[width=\textwidth]{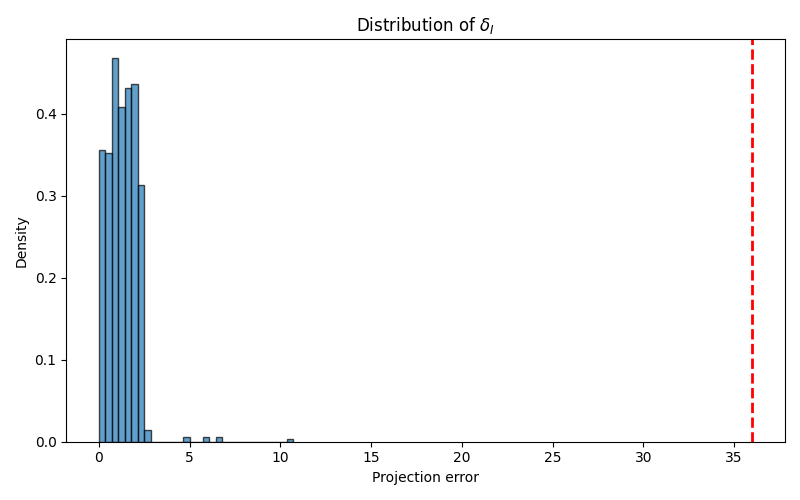}
    \caption{$\epsilon=2,\ M_{\max}=20$}
  \end{subfigure}%
  \hfill
  \begin{subfigure}[b]{0.3\textwidth}
    \centering
    \includegraphics[width=\textwidth]{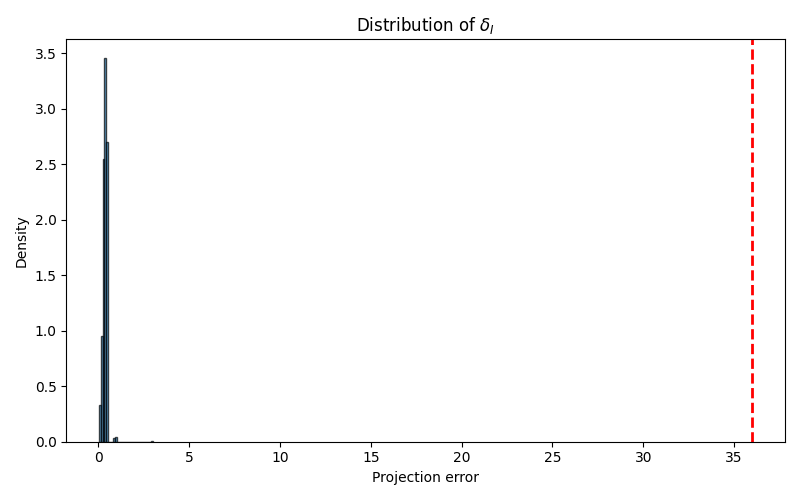}
    \caption{$\epsilon=2,\ M_{\max}=100$}
  \end{subfigure}
  \caption{The distribution of $\delta_I$ and theoretical error bound with different tolerance $\epsilon$ and max grid size $M_{\max}$.}
  \label{fig:eps_M_grid}
\end{figure}
From Figure~\ref{fig:eps_M_grid}, we can observe that the theoretical error bound for hyper-parameter discretization decreases as $\epsilon$ is reduced from 
$2$ to $1$, the real errors decrease as $M_{\max}$ increasing and are significantly smaller than 
the theoretical error bounds.
This justifies the practicality and efficiency of discretization as an effective approximation strategy for solving the BCR-SOC/MDP problem.}

To facilitate comparison with other MDP models, we 
examine the performance
of the model with different true parameter values $\theta^c$.
Figure \ref{fig:0} 
depicts 
the performance when the learned policy is deployed in various environments
with the demand
following different Poisson distributions.
To illustrate the convergence of optimal value function, we use Algorithm \ref{alg:B} to find the unique optimal value function $V^*$ of the Bellman equation corresponding to \eqref{numer-2}. 
Then, we generate sample sequences $\boldsymbol{\xi}^{t}$ across different episodes $t$, compute $\mu_{m_t,d_t}$ and obtain the corresponding optimal value $V^*(s,\mu_{m_t,d_t})$. The procedure is replicated 100 times to ensure robustness, and the results are shown in Figures \ref{fig:1} and \ref{fig:2}. Specifically, we plot boxplot of 
the sup-norm of the gap between the estimation $V^*(\cdot,\mu_{m_t,d_t})$ and the true optimal value $V^*(\cdot,\delta_{\theta^c})$, which is calculated as $\sup_{s \in \mathcal{S}} |V^*(s,\mu_{m_t,d_t}) - V^*(s,\delta_{\theta^c})|$. Moreover, we also use line chart to show the average and standard deviation of the sup-norm of the gap between the estimated value $V^*(\cdot,\mu_{m_t,d_t})$ and the true optimal value $V^*(\cdot,\delta_{\theta^c})$ across different episodes, providing insights into the robustness and variability of the estimation process.
We can make the following observations from Figures \ref{fig:0}, \ref{fig:1} and \ref{fig:2}.

\begin{figure}[ht]
    \centering
        \includegraphics[width=0.8\textwidth]{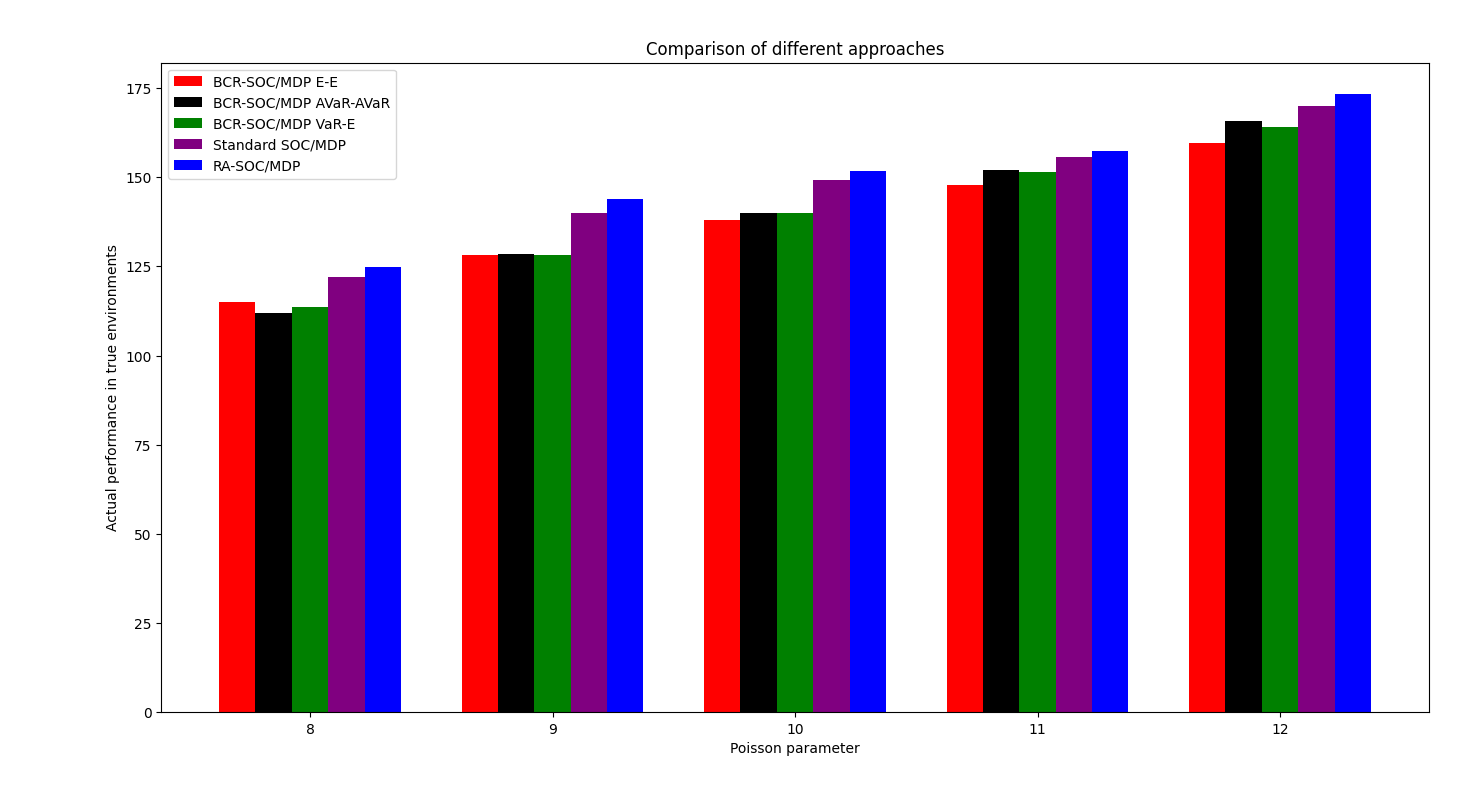}
    \caption{Actual performances of different approaches with different $\theta^c$. }
    	\label{fig:0}
\end{figure}

\begin{figure}[ht]
    \centering
    \begin{minipage}{0.5\textwidth}
        \centering
        \includegraphics[width=1.1\textwidth]{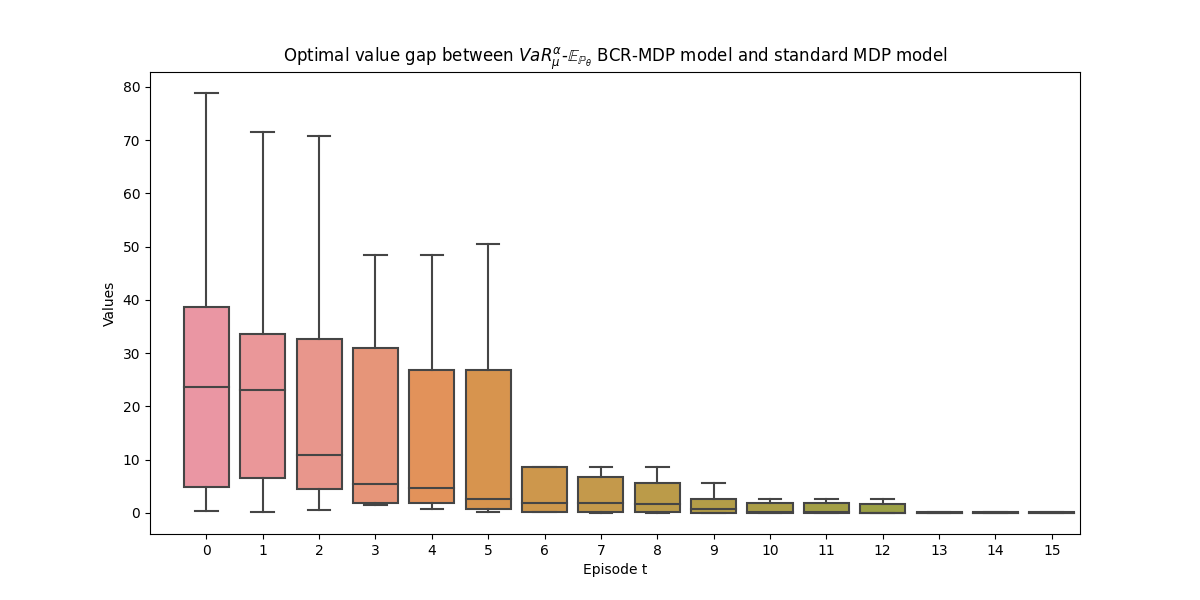}
    \end{minipage}%
    \begin{minipage}{0.5\textwidth}
        \centering
        \includegraphics[width=1.1\textwidth]{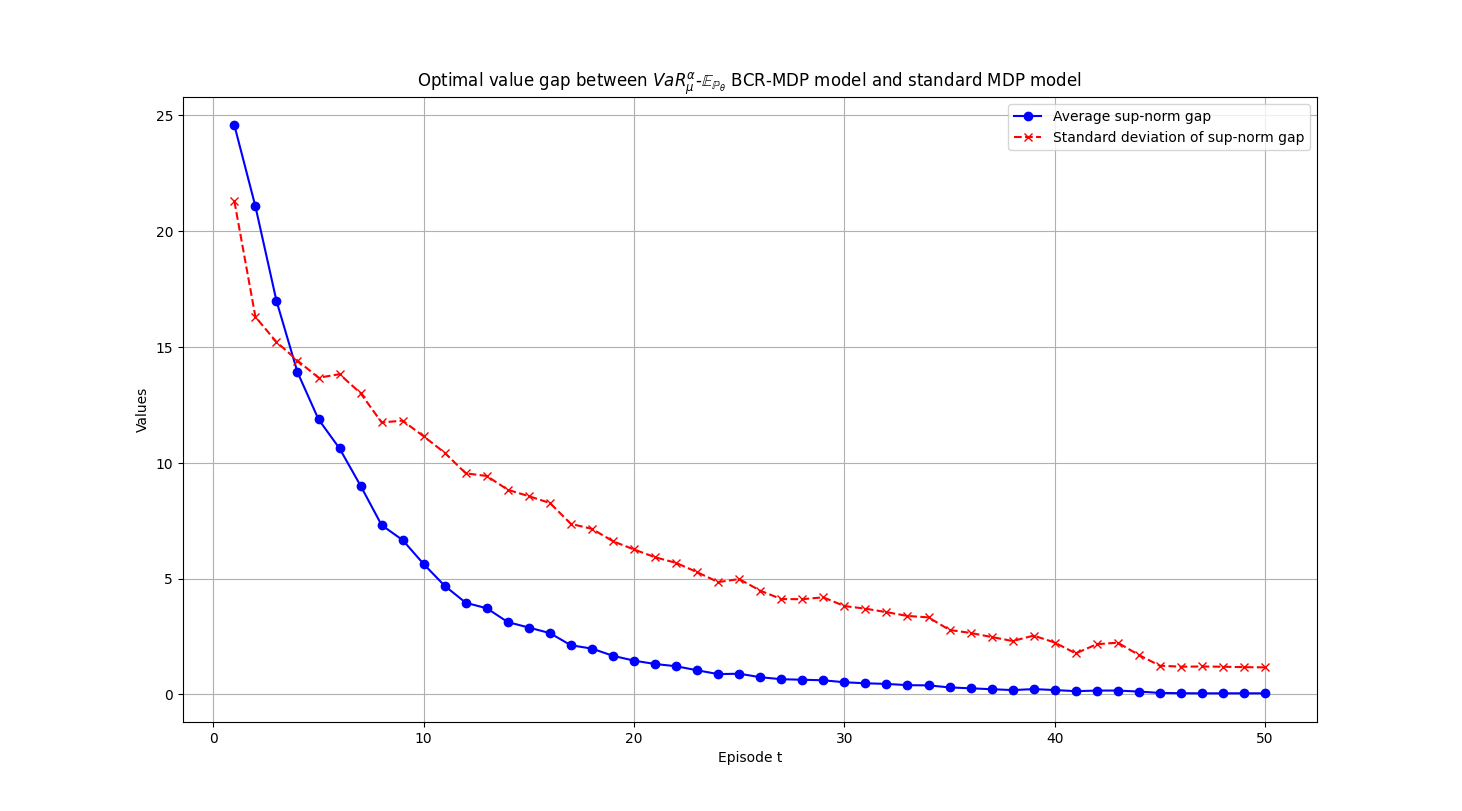}
    \end{minipage}
    \caption{Convergence of the optimal value function of the $\text{VaR}_{\mu}^{\alpha}$-$\mathbb{E}_{P_{\theta}}$ BCR-SOC/MDP model} 
    	\label{fig:1}
\end{figure}

\begin{figure}[ht]
    \centering
    \begin{minipage}{0.5\textwidth}
        \centering
        \includegraphics[width=1.1\textwidth]{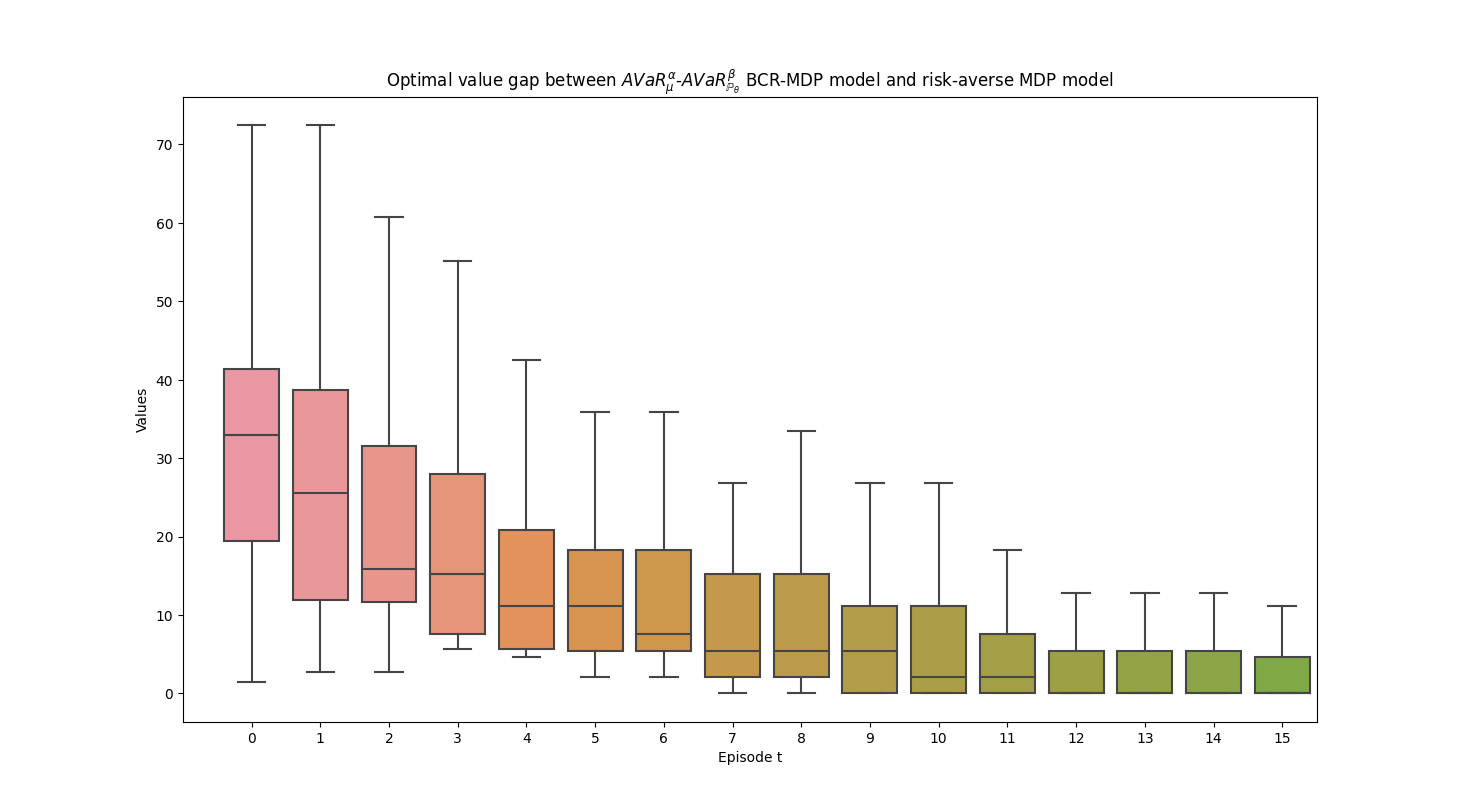}
    \end{minipage}%
    \begin{minipage}{0.5\textwidth}
        \centering
        \includegraphics[width=1.1\textwidth]{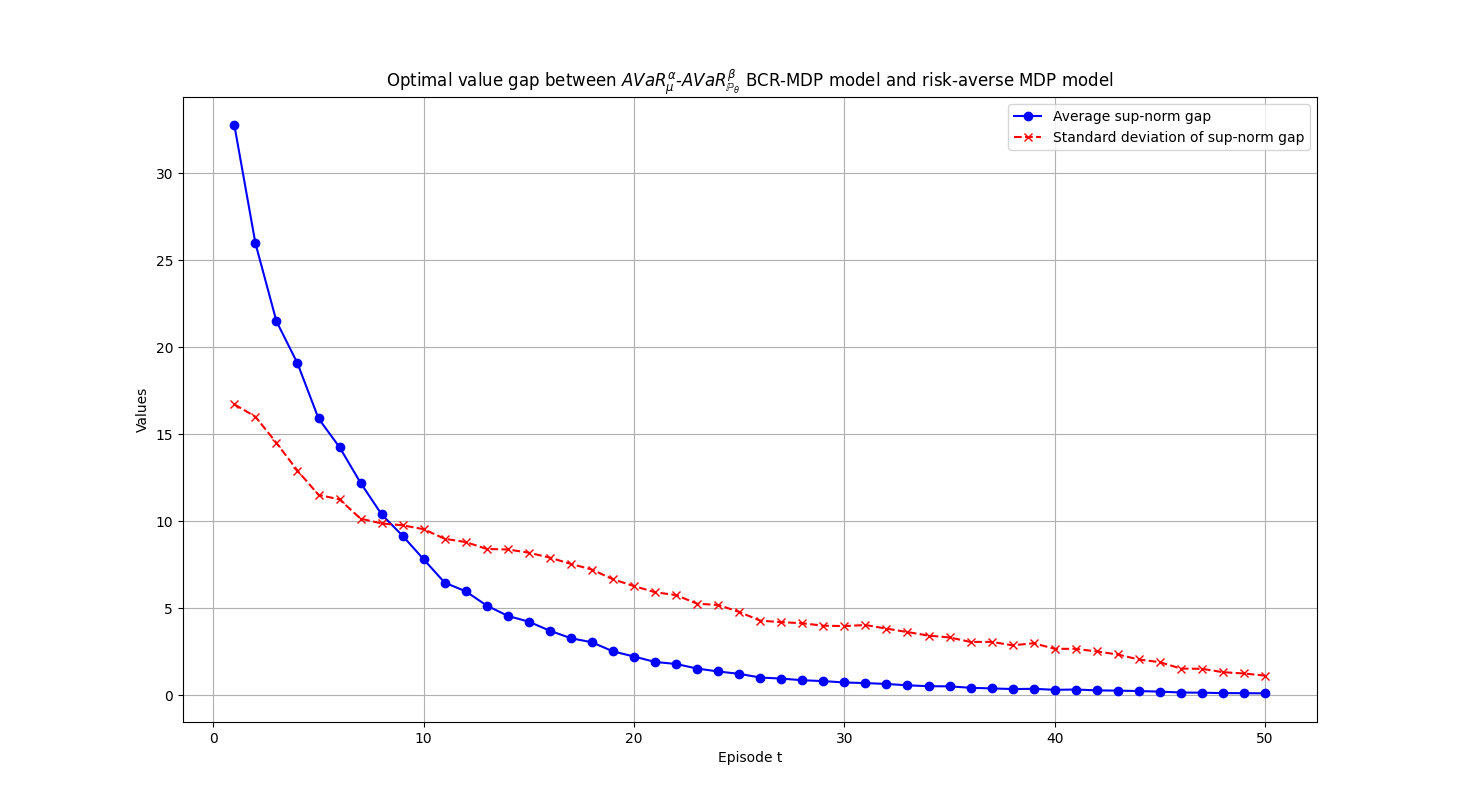}
    \end{minipage}
\caption{Convergence of the optimal value function of the $\text{AVaR}_{\mu}^{\alpha}$-$\text{AVaR}_{P_\theta}^{\beta}$ BCR-SOC/MDP model}
\label{fig:2}
\end{figure}

\begin{itemize}
\item  {\color{black}The BCR-SOC/MDP model outperforms the standard SOC/MDP and RA-SOC/MDP models 
in diverse environments and various BCR measures, such as $\mathbb{E}_{\mu}$-$\mathbb{E}_{P_{\theta}}$, $\text{VaR}_{\mu}^{\alpha}$-$\mathbb{E}_{P_{\theta}}$ and $\text{AVaR}_{\mu}^{\alpha}$-$\text{AVaR}_{P_\theta}^{\beta}$.
Moreover, when compared to the risk-neutral Bayes-adaptive SOC/MDP model, the BCR-SOC/MDP models incorporating risk-averse preferences achieve a lower optimal cost value in more adversarial scenarios (e.g., when the Poisson distribution's parameter is less than 10), emphasizing their robustness and adaptability to a wide range of challenging environments.}

\item The mean, median and the range of the sup-norm of the gap all
decrease to $0$ as $t$ increases, which means that 
the optimal value functions of all $\text{VaR}_{\mu}^{\alpha}$-$ \mathbb{E}_{P_{\theta}} $ and  $\text{AVaR}_{\mu}^{\alpha} $-$\text{AVaR}_{P_\theta}^{\beta}$ BCR-SOC/MDP models
display uniform convergence as envisaged by the theoretical
results.
{\color{black} From the boxplots in Figures \ref{fig:1} and \ref{fig:2}, we observe that the convergence rate of the optimal value functions corresponding to the $\text{AVaR}_{\mu}^{\alpha} $-$\text{AVaR}_{P_\theta}^{\beta}$ case is slower than that of the $ \text{VaR}_{\mu}^{\alpha}$-$ \mathbb{E}_{P_{\theta}}$ BCR-SOC/MDP model, which is consistent with the conclusion in Theorem \ref{thm-ca-2}. Moreover, the line charts in Figures \ref{fig:1} and \ref{fig:2} show that, compared to the $ \text{VaR}_{\mu}^{\alpha}$-$ \mathbb{E}_{P_{\theta}}$ BCR-SOC/MDP model, the $\text{AVaR}_{\mu}^{\alpha}$-$\text{AVaR}_{P_\theta}^{\beta}$ BCR-SOC/MDP model exhibits smaller variability in the sup-norm gap, reflecting the more robust characteristics of the latter risk-averse attitude with respect to the value function.}

\item As the episode \( t \) 
increases, 
the elements of the stochastic process \( \boldsymbol{\xi}^{t} \) increase, 
and the optimal value function 
converges to the optimal value 
of the true model as \( \mu_t \) converges. 
This manifests
the significance and characteristics of
Bayes-adaptive MDPs. 
Specifically, the epistemic uncertainty gradually diminishes as the DM continuously 
observes/learns 
about the environment. 
Consequently, the DM's understanding of the environment becomes more precise, 
driving the gap of the 
optimal values to zero. 
Likewise, 
the variability 
of the gap reduces to zero. 
By contrast, 
for a standard SOC/MDP or risk-averse SOC/MDP, 
once the model is established, 
the optimal value function remains 
fixed 
and does not change over time. 
Specifically, once 
an MLE estimator \( \hat{\theta} \) 
is derived for the standard or risk-averse SOC/MDP,
the corresponding 
optimal value function 
takes a similar form 
to that in (\ref{optimalvalue}), 
with \( \theta^c \) being 
replaced by \( \hat{\theta} \). 
Consequently, for any episode $t$, 
the optimal value functions 
for the estimated SOC/MDP models 
remain fixed,
which explains
that these models 
do not adaptively learn in response to 
dynamic change of the environment.   
\end{itemize}

\section{Concluding remarks}
In this paper, we introduce 
a BCR-SOC/MDP model which 
subsumes a 
range of existing SOC/MDP models
and relates to preference robust SOC/MDP models.
By comparison with \cite{shapiro2023episodic},
the  
new model 
exhibits adaptive learning capabilities 
and ensures robustness against the inherent parameter randomness.
In the finite-horizon case, we derive
some important properties of the BCR-SOC/MDP model.
In the infinite-horizon case,
we derive some basic properties 
and demonstrate the model's adaptability 
by showing that as the number of episodes increases, the optimal value functions and optimal policies almost surely converge to those of the true underlying problem. 
We introduce a novel hyper-parameter algorithm for discretizing the posterior distribution space in order to ensure the practicality of proposed algorithms.
We also propose SAA-type algorithms tailored for solving the BCR-SOC/MDP model under specific composite risk measures and examine 
the robustness and adaptability
of the model 
through numerical experiments. The theoretical 
analyses 
and numerical results show that the monetary risk measure, especially the Value at Risk (VaR), robust spectral risk measure (SRM) and average VaR, also exhibits convergence properties, prompting an examination of how to relax the conditions to encompass a broader range of risk measures.

There are a few aspects that we may take from this work 
for more in-depth exploration.
First, the relationship between the BCR-SOC/MDP and preference-robust SOC/MDP models.
We have indicated in Section 3.2 that a relationship between the two models 
may be established. It might be
interesting to investigate the 
fine details of the relationship between distributionally robustness and preference robustness in the dynamic setting given that both are fundamentally related to DM's dynamic 
risk preferences.
Moreover, the BCR-SOC/MDP 
framework 
may potentially provide a more convenient 
avenue for solving multistage randomized preference-robust optimization problems
\cite{wu2024multistage}. 
Second,
it might be worth discussing SDDP-type
methods for solving 
the proposed model as in \cite{shapiro2023episodic} at least under some specific cases given the known effectiveness of SDDP approach.
Third, as for practical application 
of the BCR-SOC/MDP model, particularly in data-driven problems, 
samples are often 
drawn from empirical data which might contain some noise. In that case, it might be relevant to investigate whether
the proposed models are 
stable 
against perturbations in the uncertainty data
in terms of both optimal values and optimal policies. Stability results for conventional SOC/MDP models have been derived by Kern et al.~\cite{kern2020first}, and it would be valuable to explore whether similar stability results can be established for the BCR-SOC/MDP model.
Finally, the numerical tests were carried out 
with two conventional SOC/MDP problems. It 
will be much more interesting 
if we can find 
new applications 
particularly 
related to online decision making. 
We leave all these avenues for future research.

\section*{Declarations}
\begin{itemize}
\item 
This project is supported by the National Key R\&D Program of China (2022YFA1004000, \linebreak 2022YFA1004001) and CUHK start-up grant.
\end{itemize}

\section*{Acknowledgments}
 The authors 
 would like to 
 thank Alex Shapiro and 
 Zhiyao~Yang 
  for 
  a number of valuable 
  discussions
  during preparation 
  of this work, {\color{black}they would also like to thank Enlu Zhou and George Lan for instrumental comments/suggestions 
  during a workshop on stochastic optimization in the  Kunming International Tianyuan Mathematics Centre from 7-10 July 2025}.

\bibliographystyle{siam}
\bibliography{ref}

\section{Appendix}

\begin{definition}[Fortet-Mourier metric]\label{D-Fort-Mou-metric}Let 
\begin{eqnarray}
\label{eq:define_L}
\mathcal{F}_{p}(Z):=\left\{h: Z\rightarrow \mathbb{R}: |h({\bm z}')-h({{\bm z}''})|\leq L_{p}({\bm z}',{{\bm z}''})\|{\bm z}'-{{\bm z}''}\|,\ \forall{\bm z}',{{\bm z}''}\in Z\right\}
\end{eqnarray}
{\color{black}
be the set of 
locally Lipschitz continuous functions of growth  order $p$,
}
where $\|\cdot\|$ denotes some norm on $Z$,
$L_{p}({\bm z}',{{\bm z}''}):=\max\{1,\|{\bm z}'\|,\|{{\bm z}''}\|\}^{p-1}$ 
for all ${\bm z}',{{\bm z}''}\in Z$,
and $p\geq 1$ describes the growth of the local Lipschitz constants. 
The $p$-th order Fortet-Mourier metric over $\mathscr{P}(Z)$ is defined by
\begin{eqnarray}
\mathsf {d\kern -0.07em l}_{\rm FM}(P',P''):=\sup_{h\in \mathcal{F}_{p}(Z)}\left|\int_{Z}h({\bm z})P'(d{\bm z})-\int_{Z}h({\bm z})P''(d{\bm z})\right|,\ \forall P',P''\in \mathscr{P}(Z).
 \label{eq:FM-Kan}
\end{eqnarray}
\end{definition} 
In the case when $p=1$, it reduces to 
Kantorovich metric, {\color{black}
in which case we denote the distance by $\mathsf {d\kern -0.07em l}_K$.}

\begin{definition}[Wasserstein distance/\,metric]
		For probability measures $P$ and $\tilde P$, the Wasserstein distance/\,metric
		of order $r\ge 1$ is
		\begin{equation}
		\mathsf {d\kern -0.07em l}_W^r(P,\tilde{P})=\left(\inf_{\pi}\iint d\left(\xi,\tilde\xi\right)^r\pi(\mathrm{d}\xi,\mathrm{d}\tilde{\xi})\right)^{\frac{1}{r}},
		\label{eq:Wasser-dist}
		\end{equation}
		where $\pi$ is among all probability measure with marginals $P$ and $\tilde{P}$,
		i.e.,
		\begin{align}
		P(A) & =\pi(A\times\Xi),\quad A\in\mathscr{B}(\Xi)\ \text{ and}\nonumber\\
		\tilde P(B) & =\pi(\Xi\times B),\quad B\in\mathscr{B}(\Xi).\label{eq:20}
		\end{align}
	\end{definition}

One of the main results	concerning the Wasserstein distance is the Kantorovich--Rubinstein Theorem (\cite{kantorovich1958space}),
	which establishes a relationship between the Kantorovich metric
		of  two probability measures and the Wasserstein distance when $r=1$, i.e.,
		\begin{equation}\label{eq:Rubinstein}
		\mathsf {d\kern -0.07em l}_W^{1}(P,\tilde{P})
		=\mathsf {d\kern -0.07em l}_K(P,\tilde{P}).
		\end{equation}
\end{document}